\documentclass{amsart}
\usepackage{amsmath}
\usepackage{graphicx} 
\usepackage{amsfonts} 
\usepackage{amssymb}
\usepackage[font=small,format=hang,labelfont={sf,bf}]{caption}
\usepackage{color}
\usepackage{mathrsfs}
\usepackage{epsfig}
\usepackage{subfig}
\usepackage{url}
\usepackage{varioref}
\theoremstyle{plain}
\usepackage{bm}
\newtheorem{theorem}{Theorem}[section]

\newtheorem{lemma}[theorem]{Lemma}
\newtheorem{proposition}[theorem]{Proposition}

\newtheorem{remark}[theorem]{Remark}

\theoremstyle{definition}
\theoremstyle{remark}
\numberwithin{equation}{section}

\newcommand{\Jgsn}{\tilde J^{g,s}_{r_n}}
\newcommand{\Jgsr}{\tilde J^{g,s}_r}

\newcommand{\kg}{k^{g,s}_r}
\newcommand{\Gbar}{\mathcal{G}}

\newcommand{\D}{\mathrm{D}}
\newcommand{\Funo}{\mathcal{F}^1_{r,\delta}}
\newcommand{\Funon}{\mathcal{F}^1_{n,\delta}}
\newcommand{\Fdue}{\mathcal{F}^2_{r,\delta}}
\newcommand{\Fduen}{\mathcal{F}^2_{n,\delta}}

\newcommand{\ep}{\varepsilon}

\newcommand{\ffi}{\varphi}

\newcommand{\R}{\mathbb{R}}
\newcommand{\N}{\mathbb{N}}

\newcommand{\E}{\mathcal{E}}

\newcommand{\di}{\textrm{dist}}

\newcommand{\Me}{\mathrm{M}(\R^d)}
\newcommand{\Mf}{\mathrm{M}_f(\R^d)}
\newcommand{\Per}{\mathrm{Per}}

\newcommand{\ud}{\;\mathrm{d}}

\newcommand{\supp}{\mathrm{supp}\,}

\newcommand{\Div}{\mathrm{Div}}

\newcommand{\weakly}{\rightharpoonup}           % weak convergence
\newcommand{\weakstar}{\stackrel{*}{\weakly}}   % weakstarconvergence

\newcommand{\loc}{\mathrm{loc}}

\newcommand{\tildeJ}{\tilde J_r^s}
\newcommand{\tildeJn}{\tilde J_{r_n}^s}
\newcommand{\Hr}{\bar{J}_{r}^{s}}
\newcommand{\Hrn}{\bar J_{r_n}^s}

\newcommand{\sca}{\sigma}
\newcommand{\Ku}{\mathcal{K}}
\newcommand{\Reg}{\mathfrak{C}}

\newcommand{\J}{\mathcal{J}}

\usepackage{latexsym}
 
\title[Convergence of supercritical fractional flows to MCF]{Convergence of supercritical fractional flows to the mean curvature flow}

\author[L. De Luca]
{L. De Luca}
\address[Lucia De Luca]{Istituto per le Applicazioni del Calcolo ``Mauro Picone" IAC-CNR. Via dei Taurini 19, I-00185 Roma, Italy}
\email[L. De Luca]{lucia.deluca@cnr.it}

\author[A. Kubin]
{A. Kubin}
\address[Andrea Kubin]{Dipartimento di Matematica ``Guido Castelnuovo", Sapienza Universit\`a di Roma. P.le  Aldo Moro 5, I-00185 Roma, Italy}
\email[A. Kubin]{kubin@mat.uniroma1.it}

\author[M. Ponsiglione]
{M. Ponsiglione}
\address[Marcello Ponsiglione]{Dipartimento di Matematica ``Guido Castelnuovo", Sapienza Universit\`a di Roma. P.le  Aldo Moro 5, I-00185 Roma, Italy}
\email[M. Ponsiglione]{ponsigli@mat.uniroma1.it}

\begin{document}
%%%%%%%%%%%%%%%%%%%%%%%%
%%%%%%%%%%%%%%%%%%%%%%%%
%%%%%%%%%%%%%%%%%%%%%%%%
\maketitle
%%%%%%%%%%%%%%%%%%%%%%%%
%%%%%%%%%%%%%%%%%%%%%%%%
%%%%%%%%%%%%%%%%%%%%%%%%
\begin{abstract}
We consider a core-radius approach to nonlocal perimeters governed by isotropic kernels having critical and supercritical exponents, extending the nowadays classical notion of  $s$-fractional perimeter, defined for $0<s<1$, to the case $s\ge 1$\,. 

We show that, as the core-radius vanishes, such core-radius regularized $s$-fractional perimeters, suitably scaled, $\Gamma$-converge to the standard Euclidean perimeter. 
Under the same scaling, the first variation of such nonlocal perimeters gives back regularized $s$-fractional curvatures which, as the core radius vanishes, converge to the standard mean curvature; as a consequence, 
we show that the level set solutions to the corresponding nonlocal geometric flows, suitably reparametrized in time, converge to the standard mean curvature flow. 

Finally, we prove analogous results in the case of anisotropic kernels with applications to dislocation dynamics.
\end{abstract}
%%%%%%%%%%%%%%%%%%%%%%%%
%%%%%%%%%%%%%%%%%%%%%%%%
%%%%%%%%%%%%%%%%%%%%%%%%
\tableofcontents
%%%%%%%%%%%%%%%%%%%%%%%%
%%%%%%%%%%%%%%%%%%%%%%%%
%%%%%%%%%%%%%%%%%%%%%%%%
\section*{Introduction} 
This paper deals with systems governed by strongly attractive nonlocal potentials. Our analysis is geometrical, so that the energy functionals are defined on measurable sets rather than on densities, and can be understood as nonlocal perimeters, whose first variation
are nonlocal curvatures, driving the corresponding geometric flows.

We focus on power law pair potentials acting on measurable sets $E\subset \R^d$, whose corresponding nonlocal energy is  of the type
\begin{equation}\label{i1}
J^{s}(E):=\int_{E} \int_{E} -\frac{1}{|x-y|^{d+s}}\ud y\ud x.
\end{equation}
For $-d<s<0$ the interaction kernel is nothing but Riesz potential; in such a case, the functionals $J^s$ are nonlocal perimeters in the sense of  \cite{CMP}. Such a geometric interpretation is supported by the fact that, as a consequence of Riesz inequality, balls are minimizers of $J^s$ under volume constraints; moreover, the first variation of $J^s$, referred to as nonlocal curvature, is monotone with respect to set inclusion. The latter provides a parabolic maximum principle which 
yields global existence and uniqueness of level set solutions
to the corresponding geometric evolutions \cite{CMP, CDNP}.

For positive $s$ the kernel in \eqref{i1} is not integrable, and the corresponding energy is infinite. Nevertheless, for $0<s<1$, changing sign to the interaction and letting $E$ interact with its complementary set instead of itself, gives a finite quantity: the well-known fractional perimeter \cite{CRS}

\begin{equation}\label{i2}
\tilde{J}^{s}(E):=\int_{E} \int_{E^c} \frac{1}{|x-y|^{d+s}}\ud y\ud x.
\end{equation}

In fact, fractional perimeters can been rigorously obtained as limits of renormalized Riesz energies by removing the infinite core energy and letting the core radius tend to zero. This has been done in \cite{DNP}, showing that the energies in \eqref{i1} and \eqref{i2} belong to a one parameter family of nonlocal $s$-perimeters, with $-d<s<1$ (see also \cite{KubPons}); in particular, for $s=0$ a new perimeter emerges, referred to as $0$-fractional perimeter.

Remarkably, as $s\to 1^{-}$, $s$-fractional perimeters, suitably scaled, converge to the standard perimeter \cite{BoBrezMir, BBM2, D, Po, ADM, CN}, and the corresponding (reparametrized in time) geometric flows converge to the standard mean curvature flow \cite{Imb,CDNP}.

For $s\ge 1$ fractional perimeters are always infinite. Nevertheless, as discussed above, the critical case $s=1$ corresponds, at least formally, to the Euclidean perimeter. Notice that for $s=1$ the fractional perimeter can be seen, again formally, as the square of the (infinite) $\dot{H}^{\frac 12}$ Gagliardo seminorm of the characteristic function of $E$. This fractional energy is particularly relevant in Materials Science, for instance in the theory of dislocations. This is why much effort has been done to derive the Euclidean perimeter directly as the limit of suitable regularizations of the  $\dot{H}^{\frac 12}$ seminorm, mainly through phase field approximations \cite{ABS}. 
%In  \cite{ABS} the authors provide a rigorous derivation via $\Gamma$-convergence of the Euclidean perimeter by a phase-field approximation of the (squared) $\dot{H}^{\frac 1 2}$ seminorm.
%energy which is given by the sum of the $\dot{H}^{\frac 1 2}$ seminorm plus a potential term penalizing non-characteristic functions. 
%Moreover, 
A specific phase field model for the energy induced by planar dislocations within the Nabarro-Peierls theory has been introduced in \cite{KCO} and studied in \cite{GM05}, where the line tension energy induced by planar dislocations is derived in terms of $\Gamma$-convergence.
%In such a model, the potential has infinite wells and the interaction kernel inherits the anisotropy, induced by the Burgers vectors, and, in turn, on the underlying lattice structure; in the limit as the phase-filed parameter vanishes, the anisotropic line tension energy induced by dislocations is derived 
%via $\Gamma$-convergence. 
%Moreover,  in \cite{DFM} dislocation dynamics is derived as the geometric flow associated to a suitable regularization of the formal first variation of the line tension energy.
In \cite{DFM,AHBR,CDFM}  the authors study dislocation dynamics within the level set formulation \'a la Slep\v cev \cite{Sl}, by considering the geometric flow associated to a suitable regularization of the formal first variation of the line tension energy. Numerical schemes have been implemented in \cite{ACMR, CFM}.
%
%We refer to \cite{GM05} for the analysis of the ground states in terms of $\Gamma$-convergence, and to \cite{DFM} for the asymptotic analysis of the corresponding geometric flows, all these results having important applications to the theory of statics and dynamics of dislocations. 

In this paper we introduce a core-radius approach to renormalize by scaling the generalized $s$-fractional perimeters and curvatures in the critical and supercritical cases $s\ge 1$. We show that,  as the core-radius tends to zero, the $\Gamma$-limit of the nonlocal perimeters is the Euclidean perimeter, the nonlocal curvatures converge to the standard mean curvature, and the corresponding geometric flows converge to the mean curvature flow. Moreover, we consider also the anisotropic variants of such perimeters, with applications to dislocation dynamics. 
 Now we discuss our results in more detail.
 
 \vskip5pt 
In Section \ref{scp} we introduce the core-radius regularized critical and supercritical perimeters (see \eqref{Jtildefrac}). In Theorem \ref{mainthm}  we show that, suitably scaled, they $\Gamma$-converge to the Euclidean perimeter. This analysis is very related with, and in some respects generalizes, many results scattered in the literature, mainly for $s>1$ \cite{MRT, BV, P}. 
%We refer to Mazon, Pagliari and to Palatucci, for a phase field approach. 

\vskip5pt
Sections \ref{sc:comp} and \ref{sc:gammalim} are devoted to the proof of Theorem \ref{mainthm}. 
The proof of the lower bounds providing compactness and $\Gamma$-liminf inequality  rely on  techniques developed in \cite{AB} and for $s=1$ in \cite{GM05}.   As a byproduct of our $\Gamma$-convergence analysis, we provide a characterization of finite perimeter sets (Theorem \ref{thmfondcarPer}) in terms of uniformly bounded renormalized supercritical fractional perimeters. Analogous results for $0<s<1$ have been obtained in \cite{BoBrezMir, D, Po, LS}.  
%The proofs of compactness and $\Gamma$-liminf inequality rely on lower bounds whose proof is based on techniques developed in \cite{AB} and for $s=1$ in \cite{GM05}.     

\vskip5pt
In Section \ref{curf} we compute the first variations of the renormalized critical and supercritical perimeters, and we show that they converge, as the core-radius vanishes, to the standard mean curvature. The estimates are robust enough to apply the theory of stability for geomertic flows developed in \cite{CDNP}. As a consequence, in Theorem \ref{maincur} we prove that the level-set solutions of supercritical fractional geometric flows, suitably reparametrized in time, locally uniformly converge to the level set solution of the classical mean curvature flow. This result extends somehow the analysis done in \cite{DFM} for $s=1$ and in \cite{CS} for $1\le s<2$\,; in the latter, the authors consider a threshold dynamics based on the $s$-parabolic flow and, in turn, on the notion of $s$-Laplacian, which is well defined only for $0<s<2$\,.
%; our level set formulation based on $s$-fractional supercritical perimeters works for all $s\ge 1$. 

\vskip5pt
In Section \ref{doli} we show that our renormalization procedures are robust enough to treat also the double limit as $s\to 1^{+}$ and the core-radius vanishes simultaneously (see Theorem \ref{mainthm111}). 
 
 \vskip5pt
 In Section \ref{anke} we generalize our results to the case of possibly anisotropic kernels (Subsection \ref{sec:ani}) and we present a relevant application to dislocation dynamics  (Subsection \ref{ssc:appdis}). It is well known that planar dislocation loops formally induce an infinite elastic energy that can be seen as an anisotropic version of the (squared) $\dot{H}^{\frac 12}$ seminorm of the characteristic function of the slip region enclosed by the dislocation curve. As mentioned above, renormalization procedures are needed to cut off the infinite core energy. In \cite{DFM,AHBR,CDFM}, the authors consider the geometric evolution of dislocation loops and face the corresponding renormalization issues: their approach consists in formally computing the first variation of the infinite energy induced by dislocations, deriving a nonlocal infinite curvature. 
 Then, they regularize such a curvature through convolution kernels, obtaining a finite curvature driving the dynamics. 
 As the convolution regularization kernel concentrates to a Dirac mass, they recover in the limit a local anisotropic mean curvature flow. The main issue in their analysis is that the convolution regularization produces a positive part in the nonlocal curvature (corresponding to a negative contribution in the normal velocity), concentrated on the scale of the core of the dislocation, giving back an evolution which does not satisfy the inclusion principle. Therefore, solutions exist only for short time. Moreover, adding strong enough forcing terms, or assuming that the positive part of the curvature is already concentrated on a point (instead of being diffused on the core region), they show that the curvature is in fact monotone with respect to inclusion of sets; as a consequence, they get a globally defined dynamics, converging, as the core-radius vanishes, to the correct anisotropic local mean curvature flow. Here we show that,  if one first regularizes the nonlocal perimeters removing the core energy and then computes the corresponding first variation, then the positive part of the curvature is actually concentrated on a point (see Remark \ref{rm:ovvi}), so that the mathematical assumption in \cite{DFM} is physically correct and fully justified through the solid core-radius formalism. 
% actually the mathematical assumption, in \cite{DFM}, that the negative part of the curvature is concentrated on a point, is
%  physically correct: this is seen if one first regularizes the nonlocal perimeters removing the core energy, as done in this paper, and then computes the corresponding first variation. 
Finally, in Subsection \ref{ssc:appdis} we show that 
  the convergence analysis of the geometric flows done in \cite{DFM} using the approach \cite{Sl} can be directly deduced from the analysis developed  in Section \ref{curf} and Subsection \ref{sec:ani}, providing then a self-contained proof relying on the general theory of nonlocal evolutions and their stability developed in \cite{CMP, CDNP}.

 \vskip10pt
 This paper, together with \cite{DNP, CDNP} completes the variational analysis of nonlocal interactions \eqref{i1} and of the corresponding geometric flows for all positive values of $d+s$. Such functionals account Riesz functionals for  $-d<s<0$. After suitable renormalization procedures, the case $s=0$ leads to  the $0$-fractional perimeter, while for $0<s<1$ we have the nowadays classical $s$-fractional perimeters. This paper focuses on the  supercritical case $s\ge 1$ leading to the classical perimeter and mean curvature flow, together with their anisotropic variants,  in the limit as the core-radius regularization vanishes. 
 
 \vskip20pt
 
%%%%%%%%%%%%%%%%%%%%%%%%%%%%%%
%%%%%%%%%%%%%%%%%%%%%%%%%%%%%%
%%%%%%%%%%%%%%%%%%%%%%%%%%%%%%
 
 \textsc{Acknoledgments:} The authors are grateful to Annalisa Cesaroni and Matteo Novaga for preliminary discussions at the early stage of the project. The authors are members of the Gruppo Nazionale per l'Analisi Matematica, la Probabilit\`a e le loro Applicazioni (GNAMPA) of the Istituto Nazionale di Alta Matematica (INdAM).
 \vskip20pt
  
%%%%%%%%%%%%%%%%%%%%%%%%%%%%%%
%%%%%%%%%%%%%%%%%%%%%%%%%%%%%%
%%%%%%%%%%%%%%%%%%%%%%%%%%%%%%
  
{\bf Notation:} We work in the space $\R^d$ where $d\ge 2$ and we denote by $\{e_j\}_{j=1,\ldots,d}$ the canonical basis of $\R^d$.
We denote by $\mathbb{R}^{d \times p}$ the set of the matrices having $d$ rows and $p$ columns. 
The symbol $|\cdot|$ stands for the Lebesgue measure in $\R^d$, $\Me$ is the family of measurable subsets of $\R^d$, whereas $\Mf\subset\Me$ is the family of subsets if $\R^d$ having finite measure.  We will always assume that every measurable set $E$ coincides with its Lebesgue representative, i.e., with the set of points at which $E$ has density equal to one.
 Moreover, for every $p>0$, we denote by $\mathcal H^p$ the $p$-Hausdorff measure.

For every $x\in\R^d$ and for every $r>0$, we denote by $B(x,r)$ the open ball of radius $r$ centered at $x$ and by $\overline{B}(x,r)$ its closure. Moreover, we set $\mathbb{S}^{d-1}:=\partial B(0,1)$. Following the standard convention, we set $\omega_d:=|B(x,1)|$ and we recall that $\mathcal H^{d-1}(\partial B(x,r))=d\omega_d r^{d-1}$. Sometimes, we will consider also subsets of $\R^{d-1}$\,. In such a case, we denote by $B'(\xi,\rho)$ the ball centered at $\xi\in\R^{d-1}$ and having radius equal to $\rho>0$\,; we set $\omega_{d-1}:=\mathcal H^{d-1}(B'(\xi,1))$ so that 
$\mathcal H^{d-1}(B'(\xi,\rho))=\omega_{d-1} \rho^{d-1}$ and $\mathcal H^{d-2}(\partial B'(\xi,\rho))=(d-1)\omega_{d-1} \rho^{d-2}$\,.
Furthermore, we set $Q:=[-\frac 1 2,\frac 1 2)^d$ and for every $\nu\in\mathbb{S}^{d-1}$ we 
set $Q^\nu:=R^\nu Q$\,, where $R^\nu$ is a (arbitrarily chosen) rotation such that $R^\nu e_d=\nu$\,.

For every set $E\in\Me$ we denote by $\Per(E)$ the De Giorgi perimeter of $E$ defined by
\begin{equation*}
\Per(E):=\sup \bigg\{ \int_{E} \Div\Phi(x)\ud x: \; \Phi \in C_{0}^1(\mathbb{R}^d; \; \mathbb{R}^d), \; \|\Phi \|_{\infty} \le 1 \bigg\}.
\end{equation*}
 For every $E\in\Me$\,, the set $\partial^* E$ identifies the reduced boundary of $E$ and $\nu_{E}: \partial^* E \rightarrow \mathbb{R}^d$ the outer normal vector field. For all $y \in \R^d$ and for every $\nu\in \mathbb{S}^{d-1}$ we set 
\begin{eqnarray}\label{semispaziotang}
H_\nu^{-}(x)&:=& \bigl\{y \in \mathbb{R}^d: \; \nu\cdot (y-x) \leq 0 \bigr\}\,,\\ \label{uppersemispazio}
H_\nu^{+}(x)&:=& \bigl\{y \in \mathbb{R}^d: \; \nu\cdot (y-x) \geq 0 \bigr\}\,,\\ \label{piatto}
H_\nu^{0}(x)&:=& \bigl\{ y \in \mathbb{R}^d: \; \nu\cdot (y-x) = 0 \bigr\}\, .
\end{eqnarray}
For every subset $E\subset \R^d$ the symbol $E^c$ denotes its complementary set in $\R^d$\,, i.e., $E^c:=\R^d\setminus E$\,.

Finally, we denote by $C(*,\cdots,*)$ a constant that depends on $*,\cdots,*$; such a constant may change from line to line.

%%%%%%%%%%%%%%%%%%%%%%%%%%%%%%
%%%%%%%%%%%%%%%%%%%%%%%%%%%%%%
%%%%%%%%%%%%%%%%%%%%%%%%%%%%%%
\section{Supercritical perimeters}\label{scp}
Let $s\ge 1$\,. For every $r>0$,  we define the interaction kernel $k^s_r:[0,+\infty)\to [0,+\infty)$ as
\begin{equation}\label{kernel}
k^s_{r}(t):=
\left\{
\begin{array}{ll}
\displaystyle \frac{1}{r^{d+s}} & \text{ for $ 0 \leq t  \leq r $ }  \, , \\
\displaystyle \frac{1}{ t^{d+s}} & \text{ for }   t  > r \, , 
\end{array}
\right.
\end{equation}
We note that 
\begin{equation}\label{scaling}
k^s_r(lt)=l^{-d-s}k^s_{\frac{r}{l}}(t) \qquad\textrm{for every }r,l,t>0.
\end{equation} 
For all $r>0$, we define the functional $J_r^{s}:\Me\to [-\infty,0]$ as 
$$
J_{r}^{s}(E):=\int_{E} \int_{E} -k^s_{r}(\vert x-y \vert ) \ud y\ud x
$$
and for every $E\in\Mf$ we set
\begin{equation}\label{defJtilde}
\tildeJ(E):=J_{r}^{s}(E) +\lambda_{r}^{s}\vert E \vert\,, 
\end{equation}
where
$$
\lambda_{r}^{s}:=\int_{\R^d} k^s_r(\vert z \vert )\ud z=\frac{(d+s)\omega_d}{sr^s}\,. 
$$
Notice that for every $E\in\Mf$ 
\begin{align*}
J^s_r(E)
\ge -\int_E\int_{\R^d}k^s_r(|x-y|)\ud y\ud x=-\lambda^s_r|E|\,,
\end{align*}
and hence $\tilde J^{s}_r:\Mf\to[0,+\infty)$\,.
Moreover, by the very definition of $\tildeJ$ in \eqref{defJtilde}, for every $E\in\Mf$ we have
\begin{equation}\label{Jtildefrac}
\tildeJ(E)=\int_{E}\int_{E^c}k^s_r(|x-y|)\ud y\ud x\,.
\end{equation}

We first state the following result concerning the pointwise limit of the functionals $\tildeJ$ as $r\to 0^+$. To this purpose, for every $s\ge 1$ we set
\begin{equation}\label{scalingper}
\sca^s(r):=\left\{\begin{array}{ll}
\displaystyle |\log r|&\textrm{if }s=1\\
&\\
\displaystyle \frac{d+s}{d+1}\frac{r^{1-s}}{s-1} &\textrm{if }s>1\,.
\end{array}\right.
\end{equation}
%%%%%%%%%%%%%%%%%%%%
%%%%%%%%%%%%%%%%%%%%
%%%%%%%%%%%%%%%%%%%%
\begin{proposition}\label{prop:point}
Let $s\ge 1$ and let $E\in\Mf$ be a smooth set. Then,
\begin{equation}\label{pointwise_limit}
\lim_{r\to 0^+}\frac{\tildeJ(E)}{\sca^s(r)}=\omega_{d-1}\Per(E)\,,
\end{equation}
where $\sca^s$ is defined in \eqref{scalingper}. In fact, for $s>1$ formula \eqref{pointwise_limit} holds for every set $E\in\Mf$ of finite perimeter. 
\end{proposition}
%%%%%%%%%%%%%%%%%%%%
%%%%%%%%%%%%%%%%%%%%
%%%%%%%%%%%%%%%%%%%%
The proof of Proposition \ref{prop:point} is postponed and will use, in particular, 
Proposition \ref{porpsvilen} below.
%%%%%%%%%%%%%%%%%%%%
%%%%%%%%%%%%%%%%%%%%
%%%%%%%%%%%%%%%%%%%%
For every $E\in\Mf$ we define the functionals
\begin{align}\label{F11}
& {F}_{1}^{s}(E):=\int_{E} \int_{E^c \cap B^c(x,1)} \frac{1}{\vert x-y \vert^{d+s}} \ud y \ud x\,, \\ \label{G1r}
& G_{r}^{s}(E):= \int_{E} \int_{E^c \cap B(x,1)} k^s_{r}(\vert x-y \vert )\ud y \ud x\,,
\end{align}
and we notice that for every $0<r<1$ it holds
\begin{equation}\label{deco}
\tilde{J}_r^s(E)= F_{1}^{s}(E)+ G_{r}^{s}(E) \,. 
\end{equation}
%%%%%%%%%%%%%%%%%%%%
%%%%%%%%%%%%%%%%%%%%
%%%%%%%%%%%%%%%%%%%%
\begin{remark}\label{zeroor}
\rm{It is easy to see that, for every $E\in\Mf$, it holds
$$
F^s_1(E)\le \int_{E}\int_{B^c(0,1)}\frac{1}{|z|^{d+s}}\ud z=\frac{d\omega_d}{s}|E|\,.
$$
}
\end{remark}
%%%%%%%%%%%%%%%%%%%%
%%%%%%%%%%%%%%%%%%%%
%%%%%%%%%%%%%%%%%%%%
Let $s\ge 1$\,. For all $r>0$ we define the function $T^s_r :\mathbb{R}^d \setminus \{0\} \rightarrow \mathbb{R}^d $ as
\begin{equation}\label{Trfunzdef}
T^s_r(x):=
\left\{
\begin{aligned}
& -\frac 1 s\frac{x}{\vert x \vert^{d+s}} & \text{ if }\vert x \vert \in [r, +\infty)   \, , \\
& \frac{x}{d r^{d+s}} - \frac{d+s}{ ds r^s}\frac{x}{\vert x \vert^{d}} & \text{ if }  \vert x \vert \in (0,r) \, .
\end{aligned}
\right.
\end{equation} 
A direct computation shows that
\begin{equation} \label{divcamp} 
\Div(T^s_r(x))= k^s_r(\vert x \vert).
\end{equation}
%%%%%%%%%%%%%%%%%%%%%%
%%%%%%%%%%%%%%%%%%%%%%
%%%%%%%%%%%%%%%%%%%%%%
\begin{lemma}\label{lemmaprimosviluttoenergia}
Let $E\in\Mf$ be a set of finite perimeter. Then, for every $0<r<1$, we have
	\begin{equation}\label{ensvil1}
	\begin{split}
%&	\tilde J_{r}^{1}(E)= \int_{E} \int_{E^c  } k_{r}(\vert x-y \vert) \ud y \ud x\\
G_r^s(E)    =&  \frac{d+s}{d s r^s}\int_{\partial^* E} \ud\mathcal{H}^{d-1}(y) \int_{E \cap B(y,r)}\frac{y-x}{|x-y|^d}\cdot \nu_{E}(y) \ud x\\
& - \frac{1}{dr^{d+s}}\int_{\partial^* E} \ud\mathcal{H}^{d-1}(y) \int_{E \cap B(y,r)}(y-x)\cdot \nu_{E}(y) \ud x\\
    &+\frac 1 s \int_{\partial^* E} \ud\mathcal{H}^{d-1}(y) \int_{E\cap (B(y,1)\setminus B(y,r))} \frac{y-x}{\vert x-y \vert^{d+s}}\cdot \nu_{E}(y)\ud x\\
    &-\frac 1 s\int_{E}\mathcal{H}^{d-1}(E^c\cap \partial B(x,1))\ud x\,,
	\end{split}
	\end{equation} 
	where in the last addendum we recall that $E$ coincides with its Lebesgue representative.
\end{lemma}
%%%%%%%%%%%%%%%%%%%%%%
%%%%%%%%%%%%%%%%%%%%%%
%%%%%%%%%%%%%%%%%%%%%%
\begin{proof}
Let $T_{r}^{s}$ be the function defined in \eqref{Trfunzdef}; then, by Gauss-Green formula and equation \eqref{divcamp}, for every $x\in E$ (and, in fact, for every $x\in\R^d$) we have
\begin{equation*}
\begin{split}
\int_{E^c \cap B(x,1)}k^s_r(\vert x-y \vert)\ud y= &\int_{E^c \cap B(x,1)} \Div(T^s_r(y-x))\ud y \\ 
=& \frac 1 s\int_{\partial^* E\cap(B(x,1)\setminus B(x,r))} \frac{(y-x)\cdot \nu_{E}(y)}{\vert x-y \vert^{d+s}} \ud\mathcal{H}^{d-1}(y) \\
&+ \frac{d+s}{dsr^s} \int_{\partial^* E \cap B(x,r)}  \frac{(y-x)\cdot \nu_{E}(y)}{\vert x-y \vert^{d}} \ud \mathcal{H}^{d-1}(y)\\
& - \frac{1}{dr^{d+s}} \int_{\partial^* E \cap B(x,r)}  (y-x)\cdot \nu_{E}(y) \ud \mathcal{H}^{d-1}(y) 
\\
& -\frac 1 s\mathcal{H}^{d-1}(E^c\cap\partial B(x,1))\,.
\end{split}
\end{equation*}
The conclusion comes by integrating with respect to $x \in E$, noticing that $\chi_{B(x,R)}(y)= \chi_{B(y,R)}(x)$ for all $x,  y\in\R^d, \, R>0$ and exchanging the order
of integration.
\end{proof}
%%%%%%%%%%%%%%%%%%%%%%
%%%%%%%%%%%%%%%%%%%%%%
%%%%%%%%%%%%%%%%%%%%%%
For every $s\ge 1$ we set
\begin{equation}\label{error}
\bm{\alpha}^s:=\left\{\begin{array}{ll}
\frac{d+2}{d+1}&\textrm{if }s=1\\
-\frac{1}{s-1}&\textrm{if }s>1\,.
\end{array}\right.
\end{equation}
%%%%%%%%%%%%%%%%%%%%%%
%%%%%%%%%%%%%%%%%%%%%%
%%%%%%%%%%%%%%%%%%%%%%
\begin{lemma}\label{porpsvilen}
Let $E \in\Mf$ be a set of finite perimeter. Then, for every $0<r<1$, the following formula holds true
\begin{equation*}
\begin{split}
\tilde J_{r}^{s}(E) =&  \omega_{d-1}\Per(E)\big( \sca^s(r) +\bm{\alpha}^s \big)+F_1^{s}(E)  \\ 
& - \frac{d+s}{dsr^s }\int_{\partial^* E} \ud\mathcal{H}^{d-1}(y) \int_{\big(E \triangle H_{\nu_E(y)}^{-}(y)\big)\cap B(y,r)} \frac{\vert (y-x)\cdot \nu_{E}(y) \vert} {\vert x-y \vert^{d}}
\ud x \\
&+\frac{1}{dr^{d+s}} \int_{\partial^* E} \ud\mathcal{H}^{d-1}(y) \int_{\big(E \triangle H_{\nu_E(y)}^{-}(y)\big)\cap B(y,r)} \vert (y-x)\cdot \nu_{E}(y) \vert \ud x
\\
&- \frac 1 s\int_{\partial^* E} \ud\mathcal{H}^{d-1}(y) \int_{\big(E \triangle H_{\nu_E(y)}^{-}(y)\big)\cap\big(B(y,1)\setminus B(y,r)\big)} \frac{\vert (y-x)\cdot \nu_{E}(y)\vert }{\vert x-y \vert^{d+s}}\ud x 
\\
& -\frac 1 s\int_{E}\mathcal{H}^{d-1}(E^c \cap \partial B(x,1))\ud x\,.
\end{split}
\end{equation*}	
\end{lemma}
%%%%%%%%%%%%%%%%%%%%%%
%%%%%%%%%%%%%%%%%%%%%%
%%%%%%%%%%%%%%%%%%%%%%
\begin{proof}
First, we notice that, using polar coordinates, for every $0<r<1$ and for every $\nu\in\mathbb{S}^{d-1}$, it holds
\begin{eqnarray}
\label{engsvilint1}
  &&\int_{H_{\nu}^{-}(0)\cap B(0,r)}  \frac{x}{\vert x \vert^{d}}\cdot\nu\ud x
=-\omega_{d-1}r\,,\\ \label{engsvilint3}
&&\int_{H_{\nu}^{-}(0)\cap B(0,r)} (x)\cdot \nu\ud x=
-\frac{\omega_{d-1}}{{d+1}}r^{d+1}\,,\\ \label{engsvilint2}
&&
 \int_{H_{\nu}^{-}(0)\cap(B(0,1)\setminus B(0,r))} \frac{x}{\vert x \vert^{d+s}}\cdot \nu\ud x
=-\omega_{d-1}\gamma^s(r)\,,
\end{eqnarray}
where
\begin{equation*}%\label{newcon}
\gamma^s(r):=\left\{\begin{array}{ll}
|\log r|&\textrm{ if }s=1\,,\\
\frac{r^{1-s}-1}{s-1}&\textrm{ if }s>1\,.
\end{array}\right.
\end{equation*}
Now we rewrite in a more convenient way the first three addends in the righthand side of \eqref{ensvil1}.
By \eqref{engsvilint1}, we get
\begin{equation}\label{ensvil2}
\begin{split}
&\frac{d+s}{dsr^s} \int_{\partial^* E} \ud\mathcal{H}^{d-1}(y) \int_{E \cap B(y,r)} \frac{(y-x)\cdot \nu_{E}(y)}{\vert x-y \vert^{d}}\ud x  \\
=&\frac{d+s}{dsr^s} \int_{\partial^* E} \ud\mathcal{H}^{d-1}(y) \int_{\big(E\setminus H_{\nu_E(y)}^{-}(y)\big) \cap B(y,r)} \frac{(y-x)\cdot \nu_{E}(y)}{\vert x-y \vert^{d}}\ud x\\
&-\frac{d+s}{dsr^s} \int_{\partial^* E} \ud\mathcal{H}^{d-1}(y) \int_{\big(H_{\nu_E(y)}^{-}(y)\setminus E\big) \cap B(y,r)} \frac{(y-x)\cdot \nu_{E}(y)}{\vert x-y \vert^{d}}\ud x\\
&+\frac{d+s}{dsr^s} \int_{\partial^* E} \ud\mathcal{H}^{d-1}(y) \int_{H_{\nu_E(y)}^{-}(y) \cap B(y,r)} \frac{(y-x)\cdot \nu_{E}(y)}{\vert x-y \vert^{d}}\ud x\\
=&-\frac{d+s}{dsr^s} \int_{\partial^* E} \ud \mathcal{H}^{d-1}(y) \int_{\big(E \triangle H_{\nu_E(y)}^{-}(y)\big) \cap B(y,r)} \frac{\vert (y-x)\cdot \nu_{E}(y) \vert }{\vert x-y \vert^{d}}\ud x \\
&+ \omega_{d-1}\Per(E)\frac{d+s}{ds}r^{1-s}\,.
\end{split}
\end{equation}
Analogously,  by \eqref{engsvilint3}, we have 
\begin{equation}\label{ensvil4}
\begin{split}
& -\frac{1}{dr^{d+s}} \int_{\partial^* E} \ud \mathcal{H}^{d-1}(y)\int_{E \cap B(y,r)}(y-x)\cdot \nu_{E}(y)\ud x \\ 
=& \frac{1}{dr^{d+s}} \int_{\partial^* E} \ud \mathcal{H}^{d-1}(y) \int_{\big(E \triangle H_{\nu_E(y)}^{-}(y)\big) \cap B(y,r)} \vert (y-x) \cdot \nu_{E}(y) \vert \ud x\\
&-\omega_{d-1}\Per(E) \frac{1}{d(d+1)}r^{1-s}.
\end{split}
\end{equation}
Furthermore, by using \eqref{engsvilint2}, we obtain
\begin{equation}\label{ensvil3}
\begin{split}
&\frac 1 s \int_{\partial^* E} \ud\mathcal{H}^{d-1}(y) \int_{E\cap\big(B(y,1) \setminus B(y,r)\big)} \frac{(y-x)\cdot \nu_{E}(y)}{\vert x-y \vert^{d+s}}\ud x \\
=& - \frac 1 s\int_{\partial^* E} \ud\mathcal{H}^{d-1}(y) \int_{\big(E \triangle H_{\nu_E(y)}^{-}(y)\big)\cap\big(B(y,1)\setminus B(y,r)\big)} \frac{\vert (y-x)\cdot \nu_{E}(y)\vert }{\vert x-y \vert^{d+s}}\ud x\\
&+\omega_{d-1}\Per(E)\frac 1 s\gamma^s(r)\,.
\end{split}
\end{equation}
We notice that
$$
 \frac 1 s\gamma^s(r)+\frac{d+s}{ds}r^{1-s}- \frac{1}{d(d+1)}r^{1-s}=\sca^s(r)+\bm{\alpha}^s\,;
$$
therefore, plugging \eqref{ensvil2}, \eqref{ensvil4}, \eqref{ensvil3} into \eqref{ensvil1}, and using \eqref{deco}, we obtain the claim.
\end{proof}
%%%%%%%%%%%%%%%%%%%%%%
%%%%%%%%%%%%%%%%%%%%%%
%%%%%%%%%%%%%%%%%%%%%%
We are now in a position to prove Proposition \ref{prop:point}.
%%%%%%%%%%%%%%%%%%%%%%
%%%%%%%%%%%%%%%%%%%%%%
%%%%%%%%%%%%%%%%%%%%%%
\begin{proof}[Proof of Proposition \ref{prop:point}]
We prove the claim under the assumption that $E$ is smooth. For $s>1$, the same proof, with $\partial E$ replaced by $\partial^* E$, works also for sets $E\in\Mf$ having finite perimeter.
We will use the decomposition of $\tildeJ$ in Lemma  \ref{porpsvilen}.
Clearly the first contribution 
$
\omega_{d-1}\Per(E)\big(\sca^s(r)+\bm{\alpha}^s)
$\,, once scaled by $\sca^s(r)$ converges to $\omega_{d-1}\Per(E)$\,.
Now we will prove that all the other contributions, scaled by $\sca^s(r)$, vanish as $r\to 0^+$\,.
%%%%%%%%%%%%%%%%%%%%
\vskip5pt
{\it $2^{\mathrm{nd}}$ addend:} By Remark \ref{zeroor}, we have that 
$$
\lim_{r\to 0^+}\frac{F^s_1(E)}{\sca^s(r)}=0\,.
$$
%%%%%%%%%%%%%%%%%%%%
\vskip5pt
 	\textit{$3^{\mathrm{rd}}$ addend}. 
	By the very definition of $\sca^s(r)$ in \eqref{scaling} we have that $\sca^s(r)r^{s-1}$ is uniformly bounded from below by a positive constant for every $0<r<\frac 1 2$, so that by the change of variable $z=\frac{x-y}{r}$, we have
	 	\begin{equation*}%\label{118}
 	 \begin{split}
	 &\frac{1}{\sca^s(r)}\frac{d+s}{dsr^s}\int_{\partial E} \ud\mathcal{H}^{d-1}(y)\int_{\big(E \triangle H^{-}_{\nu_E(y)}(y) \big)\cap B(y,r)} \frac{\vert (y-x) \cdot \nu_{E}(y) \vert  }{\vert x-y \vert^{d}}\ud x \\
\le & C(d,s)\int_{\partial E} \ud\mathcal{H}^{d-1}(y) \int_{ \big(E \triangle H^{-}_{\nu_E(y)}(y)\big) \cap B(y,r)} \frac{1}{r} \frac{\vert (y-x)\cdot \nu_{E}(y) \vert }{\vert x-y \vert^{d}}\ud x\\
=& C(d,s) \int_{\partial E} \ud\mathcal{H}^{d-1}(y)\int_{\big(\frac{E-y}{r}\triangle (H_{\nu_E(y)}^-(y)-\frac y r)\big)\cap B(0,1)}  \frac{\vert z\cdot \nu_{E}(y) \vert }{ \vert z \vert^{d}}\ud z\,,
 \end{split}
 \end{equation*}
where the last integral vanishes as $r\to 0^+$ in virtue of the Lebesgue's Dominated Convergence Theorem since ${\chi}_{\frac{E-y}{r}}\to {\chi}_{H_{\nu_E(y)}^-(y)-\frac y r}$ in $L^1_{\loc}$\,.
%%%%%%%%%%%%%%%%%%%%%%%%%% 	
\vskip5pt
 	\textit{$4^{\mathrm{th}}$ addend}. Trivially, we have
 	\begin{equation*}
 	\begin{split}
 	&\frac{1}{\sca^s(r)}\frac{1}{dr^s}\int_{\partial E} \ud\mathcal{H}^{d-1}(y) \int_{\big (E \triangle H_{\nu_E(y)}^{-}(y)\big) \cap B(y,r)}  \frac{\vert (y-x)\cdot \nu_{E}(y) \vert }{r^{d}}\ud x\\
	 \leq &\frac{1}{\sca^s(r)} \frac{1}{dr^s}\int_{\partial E} \ud\mathcal{H}^{d-1}(y)
 \int_{ \big(E \triangle H_{\nu_E(y)}^{-}(y)\big) \cap B(y,r)} \frac{\vert (y-x)\cdot \nu_{E}(y) \vert }{\vert x-y \vert^{d}}\ud x, 
 	\end{split}
 	\end{equation*}  
	where the last integral vanishes as shown above.
%%%%%%%%%%%%%%%%%%%%%%%%%%%%
 	\vskip5pt
 	\textit{$5^{\mathrm{th}}$ addend}.
	We first discuss the simpler case $s>1$\,.
	In such a case, for every $y\in\partial E$, using again the change of variable $z=\frac{x-y}{r}$, we have
\begin{equation*}
\begin{split}
&\frac{1}{\sca^s(r)}\frac{1}{s}\int_{\partial E}\ud \mathcal{H}^{d-1}(y)\int_{\big(E \triangle H_{\nu_E(y)}^{-}(y)\big)\cap \big(B(y,1)\setminus B(y,r)\big)} \frac{\vert (y-x)\cdot \nu_{E}(y)\vert }{\vert x-y \vert^{d+s}}\ud x\\
=&\frac{r^{1-s}}{\sca^s(r)}\frac{1}{s}\int_{\partial E}\ud \mathcal{H}^{d-1}(y)\int_{\big(\frac{E-y}{r}\triangle (H_{\nu_E(y)}^-(y)-\frac{y}{r})\big)\cap \big(B(0,\frac 1 r)\setminus B(0,1)\big)} \frac{|z\cdot \nu_E(y)|}{|z|^{d+s}}\ud z\\
\le&C(d,s)\int_{\partial E}\ud \mathcal{H}^{d-1}(y)\int_{\big(\frac{E-y}{r}\triangle (H_{\nu_E(y)}^-(y)-\frac{y}{r})\big)\setminus B(0,1)} \frac{1}{|z|^{d+s-1}}\ud z\,,
\end{split}
\end{equation*}
where the last double integral vanishes as $r\to 0^+$ in virtue of the Lebesgue's Dominated Convergence Theorem  using that ${\chi}_{\frac{E-y}{r}}\to {\chi}_{H_{\nu_E(y)}^-(y)-\frac y r}$ in $\mathrm{L}^1_{\loc}$ as $r\to 0^+$ and the fact that
the function $h(z):=\frac{1}{|z|^{d+s-1}}$ is in $\mathrm{L}^{1}(\R^d\setminus B(0,1))$ for $s>1$.
%%%%%%%%%%%%%%%%%%%%%%%%%%%%%%%%%%

Notice that the reasoning above does not apply to the case $s=1$ since for $s=1$ the function $h(z)=\frac{1}{|z|^d}$ is not in   $\mathrm{L}^{1}(\R^d\setminus B(0,1))$\,. Let now $s=1$ and recall that $\sca^1(r)=|\log r|$\,. 
	Since $E$ has smooth boundary, there exists $0<\delta<1$ such that for all $y \in \partial E$ the sets $B^-:= B(y-\delta \nu_{E}(y), \delta) $ and $B^+:= B(y+\delta \nu_{E}(y), \delta)$ satisfy
$$
B^-\subset E \setminus \partial E\,,\qquad B^+	 \subset E^c \setminus \partial E\,,\qquad y \in \partial B^- \cap \partial B^+\,.
$$	
Therefore, we have that 
 \begin{equation}\label{useful}
 E \triangle H_{\nu_E(y)}^{-}(y) \subset (H_{\nu_E(y)}^{-}(y)\setminus B^-) \cup (H_{\nu_E(y)}^{+}(y) \setminus B^+)\,,
\end{equation}
where $H_\nu^\pm(y)$ are defined in \eqref{uppersemispazio} and \eqref{semispaziotang}.
 Fix $y \in \partial E$ and let $R_y$ be a rotation of $\R^d$ such that $ R_y \nu_{E}(y)= e_d$\,. Moreover, we denote by  $z=(z' ,z_d )$ the points in $\R^d$, so that $z'=(z_1,\ldots,z_{d-1})\in\R^{d-1}$. Furthermore, we set $\R^d_+:=\{z\in\R^d\,:\,z_d\ge 0\}$\,. By \eqref{useful} we have	
 \begin{equation}\label{ciao}
 \begin{aligned}
 	& \frac{1}{\vert \log r \vert} \int_{\partial E}\ud \mathcal{H}^{d-1}(y)\int_{\big(E \triangle H_{\nu_E(y)}^{-}(y)\big) \cap \big(B(y,1) \setminus B(y,r)\big)} \frac{\vert (y-x) \cdot \nu_{E}(y) \vert }{\vert x-y \vert^{d+1}}\ud x   \\
 	\le & \frac{1}{\vert \log r \vert}\int_{\partial E}\ud \mathcal{H}^{d-1}(y) \int_{\big(H_{\nu_E(y)}^{-}(y) \setminus B^-\big) \cap B(y,1)} \frac{\vert (y-x) \cdot \nu_{E}(y) \vert }{\vert x-y \vert^{d+1}}\ud x   \\
 	& +\frac{1}{\vert \log r \vert}\int_{\partial E}\ud \mathcal{H}^{d-1}(y) \int_{\big(H_{\nu_E(y)}^{+}(y) \setminus B^+\big) \cap B(y,1)} \frac{\vert (y-x) \cdot \nu_{E}(y) \vert }{\vert x-y \vert^{d+1}}\ud x  \\
=&	 \frac{2}{\vert \log r \vert}\Per(E)\int_{\R^{d}_+\cap\big(B(0,1)\setminus B(\delta e_d,\delta)\big)}\frac{z_d}{(\vert z' \vert^2 + z_d^2)^{\frac{d+1}{2}}} \ud z_d\,.
 	\end{aligned}
	\end{equation} 
Therefore, in order to prove that the first double integral in \eqref{ciao} vanishes as $r\to 0^+$, it is enough to show that
\begin{equation}\label{forsepeter}
\int_{\R^{d}_+\cap\big(B(0,1)\setminus B(\delta e_d,\delta)\big)}\frac{z_d}{(\vert z' \vert^2 + z_d^2)^{\frac{d+1}{2}}} \ud z_d\le C(d,\delta)\,,
\end{equation}	
for some finite constant $C(d,\delta)>0$\,.
To this purpose, setting
$$
A_\delta:=\{z=(z',z_d)\in\R^d_+\setminus B(\delta e_d,\delta)\,:\,|z'|<\delta\,, z_d<\delta\}\,,
$$
we notice that 
\begin{equation}\label{hans1}
\R^{d}_+\cap\big(B(0,1)\setminus  B(\delta e_d,\delta)\big)\subset \big(B(0,1)\setminus B(0,\delta)\big)\cup  A_\delta\,.
\end{equation}	
Moreover, there exists a constant $c_\delta$ (take, for instance, $c_\delta=\frac 1 \delta$)
such that
\begin{equation}\label{hans2}
A_\delta\subset \tilde A_\delta:= \{z=(z',z_d)\in\R^d_+\,:\, |z'|<\delta\,,\, z_d< c_\delta|z'|^2\}\,.
\end{equation}
Therefore, by \eqref{hans1} and \eqref{hans2}, we get 	
\begin{align*}	
&\int_{\R^{d}_+\cap\big(B(0,1)\setminus B(\delta e_d,\delta)\big)}\frac{z_d}{(\vert z' \vert^2 + z_d^2)^{\frac{d+1}{2}}} \ud z_d\\
\le& \int_{\tilde A_\delta} \frac{z_d}{(\vert z' \vert^2 + z_d^2)^{\frac{d+1}{2}}} \ud z_d+\int_{B(0,1)\setminus B(0,\delta)}\frac{z_d}{(\vert z' \vert^2 + z_d^2)^{\frac{d+1}{2}}} \ud z_d\\
\le
& 
\int_{B'(0,\delta)} \ud z'  \int_{0}^{\delta-\sqrt{\delta^2-|z'|^2}} \frac{c_\delta|z'|^2}{\vert z' \vert^{d+1}} \ud z_d 
+\int_{B(0,1)\setminus B(0,\delta)}  \frac{1}{|z|^d} \ud z\\
\le&	\frac{c_\delta}{\delta}\int_{B'(0,\delta)}|z'|^{3-d}\ud z'+|\log\delta|=:C(d,\delta)\,,
\end{align*}	
i.e., \eqref{forsepeter}.
%%%%%%%%%%%%%%%%%%%%%%%%%%%%
 \vskip5pt
 	\textit{$6^{\mathrm{th}}$ addend:} We have that
 	\begin{equation*}
 	\frac{1}{\sca^s(r) } \int_{E } \mathcal{H}^{d-1}(E^c \cap \partial B(x,1))\ud x \leq \frac{1}{\sca^s(r)} d\omega_{d} \vert E \vert \rightarrow 0 \quad \text{ as } r \rightarrow 0^+\,.
 	\end{equation*} 
Thus, the proof of Lemma \ref{prop:point} is concluded.
 \end{proof}

%%%%%%%%%%%%%%%%%%%%%%
%%%%%%%%%%%%%%%%%%%%%%
%%%%%%%%%%%%%%%%%%%%%%
We will show that the limit \eqref{pointwise_limit} is actually a $\Gamma$-limit.
 %%%%%%%%%%%%%%%%%%%%%%
%%%%%%%%%%%%%%%%%%%%%%
%%%%%%%%%%%%%%%%%%%%%%
\begin{theorem}\label{mainthm}
Let $s\ge 1$ and let $\{r_n\}_{n\in\N}\subset (0,+\infty)$ be such that $r_n\to 0^+$ as $n\to +\infty$. The following $\Gamma$-convergence result holds true.
\begin{itemize}
%%%%%%%%%%%%%%%%%%%%%%%%
\item[(i)] (Compactness) Let $U\subset\R^d$ be an open bounded set and let $\{E_n\}_{n\in\N}\subset\Me$ be such that $E_n\subset U$ for every $n\in\N$ and
\begin{equation}\label{energy_bound}
 \tildeJn(E_n)\le M \sca^s(r_n)\qquad\textrm{for every }n\in\N,
 \end{equation}
for some constant $M$ independent of $n$.
Then, up to a subsequence, $\chi_{E_n}\to \chi_E$ strongly in $L^1(\R^d)$ for some set $E\in\Mf$ with $\Per(E)<+\infty$.
%%%%%%%%%%%%%%%%%%%%%%%%%
\item[(ii)] (Lower bound) Let $E\in\Mf$. For every $\{E_n\}_{n\in\N}\subset\Mf$ with $\chi_{E_n}\to\chi_E$ strongly in $L^1(\R^d)$ it holds
\begin{equation}\label{trueliminf}
\omega_{d-1}\Per(E)\le\liminf_{n\to +\infty}\frac{\tildeJn(E_n)}{\sca^s(r_n)}\,.
\end{equation}
%%%%%%%%%%%%%%%%%%%%%%%%%%
\item[(iii)] (Upper bound) For every $E\in\Mf$ there exists $\{E_n\}_{n\in\N}\subset\Mf$ such that $\chi_{E_n}\to\chi_E$ strongly in $L^1(\R^d)$ and
\begin{equation*}%\label{limsup}
\omega_{d-1}\Per(E)=\lim_{n\to +\infty}\frac{\tildeJn(E_n)}{\sca^s(r_n)}\,.
\end{equation*}
%%%%%%%%%%%%%%%%%%%%%%%%%%
\end{itemize}
\end{theorem}
%%%%%%%%%%%%%%%%%%%%%%
%%%%%%%%%%%%%%%%%%%%%%
%%%%%%%%%%%%%%%%%%%%%%
The proof of Theorem \ref{mainthm} will be done in Sections \ref{sc:comp} and \ref{sc:gammalim} below.

%%%%%%%%%%%%%%%%%%%%%%
%%%%%%%%%%%%%%%%%%%%%%
%%%%%%%%%%%%%%%%%%%%%%
To ease notation, for every $r>0$ we set $\Hr(\cdot):=\frac{\tildeJ(\cdot)}{\sca^s(r)}$\,. In view of \eqref{Jtildefrac}, for every $E\in\Mf$ we have
\begin{equation*}\label{Hfrac}
\begin{aligned}
\Hr(E)=&\frac{1}{\sca^s(r)}\int_{E}\int_{E^c}k^s_r(|x-y|)\ud y\ud x\\
=&\frac{1}{2\sca^s(r)}\int_{\R^d}\int_{\R^d}k^s_r(|x-y|)|\chi_E(x)-\chi_E(y)|\ud y\ud x\,.
\end{aligned}
 \end{equation*}
%%%%%%%%%%%%%%%%%
%%%%%%%%%%%%%%%%%
%%%%%%%%%%%%%%%%%
%%%%%%%%%%%%%%%%%
%%%%%%%%%%%%%%%%%
%%%%%%%%%%%%%%%%%
\section{Proof of Compactness}\label{sc:comp}
This section is devoted to the proof of Theorem \ref{mainthm}(i).
To accomplish this task we will need some preliminary results that are collected in Subsection \ref{ssc:preliminary} below.

%%%%%%%%%%%%%%%%%
%%%%%%%%%%%%%%%%%
%%%%%%%%%%%%%%%%% 
\subsection{Preliminary results}\label{ssc:preliminary}
 We first recall the following classical result (see also \cite[Theorem 3.23]{AmbFuscPall}).
%%%%%%%%%%%%%%%%%
%%%%%%%%%%%%%%%%%
%%%%%%%%%%%%%%%%% 
\begin{theorem}[Compactness in $\mathrm{BV}$] \label{compatezzainBV}
	Let $\Omega \subset \mathbb{R}^d$ be an open set and let $\{u_n\}_{n\in \mathbb{N}} \subset \mathrm{BV}_{\mathrm{loc}}(\Omega)$ with
	\begin{equation*}
	\sup_{n \in \mathbb{N}} \biggl\{ \int_{A} \vert u_n(x) \vert \ud x+ \vert D u_n \vert(A) \biggr\} < +\infty \quad \forall A \subset \subset \Omega \; \text{open}\,.
	\end{equation*}
Then, there exist a subsequence $\{n_k\}_{k \in \mathbb{N}}$ and a function $u \in \mathrm{BV}_{\mathrm{loc}}(\Omega)$ such that $u_{n_k} \rightarrow u $ in $\mathrm{L}_{\mathrm{loc}}^1(\Omega)$ as $k \rightarrow + \infty$.
\end{theorem}
%%%%%%%%%%%%%%%%%
Now we prove a non-local Poincar\'e-Wirtinger type inequality.
%%%%%%%%%%%%%%%%%
%%%%%%%%%%%%%%%%%
%%%%%%%%%%%%%%%%% 
\begin{lemma}\label{lemmatipofracpoinc}
	Let $0<r<l $ be such that $\omega_{d}r^d \leq \frac{l^d}{2}$. Let  $\xi\in\R^d$ and let $u \in \mathrm{L}^{1}(lQ+\xi)$. Then, for every $s\ge 1$ we have
	\begin{equation}\label{frazpoinctipo}
	\begin{aligned}
\int_{lQ+\xi} \bigg|u(y)- \frac{1}{\vert (lQ+\xi)\setminus B(y,r) \vert} \int_{(lQ+\xi)\setminus B(y,r)} u(x)\ud x   \bigg|\ud y  \\ \le 2 d^{\frac{d+s}{2}} l^s
 \int_{lQ+\xi} \int_{lQ+\xi } \vert u(y)-u(x)\vert k^s_r(\vert x-y \vert) \ud y \ud x\,.
	\end{aligned} 
	\end{equation}
\end{lemma}
%%%%%%%%%%%%%%%%%
%%%%%%%%%%%%%%%%%
%%%%%%%%%%%%%%%%% 
\begin{proof}
By translational invariance, it is enough to prove the claim only for $\xi=0$. 
By assumption, for every $y \in lQ$ we have 
$$ \vert lQ\setminus B(y,r) \vert \geq l^d - \omega_{d} r^d \geq \frac{l^d}{2}. $$	
As a consequence, we have
\begin{equation*}
\begin{split}
& \int_{lQ} \bigg|u(y)- \frac{1}{\vert lQ\setminus B(y,r) \vert} \int_{Q\setminus B(y,r)} u(x)\ud x   \bigg|\ud y  \\
\le& \int_{lQ} \frac{1}{\vert lQ\setminus B(y,r) \vert} \int_{lQ\setminus B(y,r)} \vert u(y)-u(x) \vert \ud x\ud y  \\
=& \int_{lQ} \frac{1}{\vert lQ\setminus B(y,r) \vert} \int_{lQ\setminus B(y,r)} \frac{\vert u(y)-u(x) \vert}{\vert y-x \vert^{d+s}} \vert y-x \vert^{d+s}  \ud x\ud y\\
\le & \int_{lQ} \frac{2 d^{\frac{d+s}{2}}}{l^d} \int_{lQ\setminus B(y,r)} \frac{\vert u(y)-u(x) \vert}{\vert y-x \vert^{d+s}} l^{d+s}  \ud x\ud y  \\
\le & 2d^{\frac{d+s}{2}} l^s  \int_{lQ} \int_{lQ  } \vert u(y)-u(x)\vert k^s_r(\vert y-x \vert ) \ud y \ud x,
\end{split}
\end{equation*}
i.e., \eqref{frazpoinctipo}.
\end{proof}
%%%%%%%%%%%%%%%%%%
%%%%%%%%%%%%%%%%%%
%%%%%%%%%%%%%%%%%%
\begin{lemma}\label{lemmastimacubi}
	Let $0<r<l$ be such that $\omega_{d}r^d < \frac{l^d}{4}$. For every $\xi\in\R^d$ and for every $E\in\Mf$, it holds
	\begin{multline}\label{stimaneicubi}
	\frac{1}{l^d}\vert (lQ+\xi)\setminus E \vert \vert (lQ+\xi) \cap E \vert \\
	\leq \int_{lQ+\xi} \bigg| \chi_{E}(x)- \frac{1}{\vert (lQ+\xi) \setminus B(x,r) \vert} \int_{(lQ+\xi) \setminus B(x,r)} \chi_{E}(y)\ud y \bigg|\ud x\,. 
	\end{multline} 
\end{lemma}
%%%%%%%%%%%%%%%%%%%
%%%%%%%%%%%%%%%%%%%
%%%%%%%%%%%%%%%%%%%
\begin{proof}
We can assume without loss of generality that $\xi=0$\,. 
It is enough to prove \eqref{stimaneicubi} only in the case $\vert lQ \cap E \vert \geq \frac{l^d}{2}$\,; indeed, once proven the inequality \eqref{stimaneicubi} in such a case, 
 if  $\vert lQ \setminus E \vert \geq \frac{l^d}{2}$\,, then the set $\tilde E=lQ\setminus E$ satisfies $\vert lQ \cap \tilde E \vert \geq \frac{l^d}{2}$\,, and hence $\tilde E$ and, in turn, $E$ satisfy \eqref{stimaneicubi}.

Let $\vert lQ \cap E \vert \geq \frac{l^d}{2}$\,; then, for every $x\in\R^d$ we have
\begin{equation}\label{eq1stimaneicubi}
\vert (lQ \cap E) \setminus B(x,r)\vert \geq \vert lQ \cap E \vert - \omega_{d} r^d \geq \vert lQ \cap E \vert -\frac{l^d}{4} \geq \frac{\vert lQ \cap E \vert}{2}\,,
\end{equation}
so that
\begin{equation*}
\begin{split}
& \int_{lQ} \bigg| \chi_{E}(x)- \frac{1}{\vert lQ \setminus B(x,r) \vert} \int_{lQ \setminus B(x,r)} \chi_{E}(y)\ud y \bigg|\ud x\\
=& \int_{lQ \cap E} \bigg| 1- \frac{\vert (lQ\setminus B(x,r)) \cap E \vert}{\vert lQ \setminus B(x,r) \vert }\bigg|\ud x+ \int_{lQ \setminus E} \frac{\vert (lQ \setminus B(x,r)) \cap E \vert}{\vert lQ \setminus B(x,r)\vert}\ud x \\
\ge& \frac{1}{l^d}\biggl(\int_{lQ \cap E} \vert (lQ   \setminus B(x,r)) \setminus E \vert \ud x+ \int_{lQ \setminus E} \vert (lQ \setminus B(x,r) )\cap E \vert \ud x \biggr)\\
=&\frac{2}{l^d}\int_{lQ \setminus E} \vert (lQ \cap E)\setminus B(x,r) \vert \ud x\\
\ge&\frac{1}{l^d}|lQ\cap E|\,|lQ\setminus E|\,,
\end{split}
\end{equation*}	
where in the last inequality we have used formula \eqref{eq1stimaneicubi}.
\end{proof}
%%%%%%%%%%%%%%%%%%
%%%%%%%%%%%%%%%%%%
%%%%%%%%%%%%%%%%%%
The following result is a localized isoperimetric inequality for the non-local perimeters $\tildeJ$\,. 
%%%%%%%%%%%%%%%%%%
%%%%%%%%%%%%%%%%%%
%%%%%%%%%%%%%%%%%%
\begin{lemma}\label{quasiiso}
Let $s\ge 1$ and let $\Omega \in \mathrm{M}_f(\mathbb{R}^d)$ be a bounded set with Lipschitz continuous boundary and $\vert \Omega \vert=1$. 
For every $\eta\in(0,1)$ there exist a constant $C(\eta,d,s)>0$ and $r_0>0$ such that for every measurable set $A\subset\Omega$ with $\eta\le |A|\le 1-\eta$, it holds
\begin{equation}\label{f:quasiiso}
\int_{A}\int_{\Omega\setminus A}k^s_r(|x-y|)\ud y\ud x\ge C(\Omega,d,s,\eta)\sca^s(r)\qquad\textrm{for every }r\in (0,r_0).
\end{equation}	
\end{lemma}
%%%%%%%%%%%%%%%%%%%%%%%%%%%%%%%
%%%%%%%%%%%%%%%%%%%%%%%%%%%%%%%
%%%%%%%%%%%%%%%%%%%%%%%%%%%%%%%
The proof of Lemma \ref{quasiiso} follows along the lines of \cite[Lemma 15]{GM05}, with slight differences due to the core radius approach adopted in this paper. 
Before  proving Lemma \ref{quasiiso}, we state the following result
 which is a consequence of \cite[Theorem 1.4]{AB}.
%%%%%%%%%%%%%%%%%%%%%%%%%%%%%%%
%%%%%%%%%%%%%%%%%%%%%%%%%%%%%%%
%%%%%%%%%%%%%%%%%%%%%%%%%%%%%%% 
\begin{lemma} [\cite{GM05}] \label{lm:l1ker}
Let $\Omega \in \mathrm{M}_f(\mathbb{R}^d)$ be a bounded set with Lipschitz continuous boundary and $\vert \Omega \vert=1$ and 
let $\phi\in C_{c}^{\infty}(B(0,1);[0,+\infty))$ be such that $\int\phi\ud x=1$ and $\phi>0$ in $B(0,\frac 1 2)$. 
For every $\delta>0$ we set $\phi_\delta(\cdot):=\frac{1}{\delta^d}\phi(\frac{\cdot}{\delta})$.
For every $\eta\in(0,1)$ there exists a constant $C(\phi,\eta)>0$ such that for every measurable set $A\subset\Omega$ with $\eta\le |A|\le 1-\eta$ and for every $\delta \in (0,1)$ it holds
\begin{equation*}
\frac 1 {\delta}\int_{A}\int_{\Omega\setminus A}\phi_\delta(|x-y|)\ud y\ud x\ge C(\Omega,\phi,\eta).
\end{equation*}
 \end{lemma}
%%%%%%%%%%%%%%%%%%%%%%%%%%%%%%%
%%%%%%%%%%%%%%%%%%%%%%%%%%%%%%%
%%%%%%%%%%%%%%%%%%%%%%%%%%%%%%% 
The above lemma has been stated and proven in \cite[Proposition 14]{GM05} in the case $d=2$ with $\Omega=(-\frac 1 2,\frac 1 2)^2$ but in fact the same proof is not affected neither by the dimension $d$ nor by the specific shape of $\Omega$.
We are now in a position to prove Lemma \ref{quasiiso}.
%%%%%%%%%%%%%%%%%%%%%%%%%%%%%%%%
%%%%%%%%%%%%%%%%%%%%%%%%%%%%%%%%
%%%%%%%%%%%%%%%%%%%%%%%%%%%%%%%% 
\begin{proof}[Proof of Lemma \ref{quasiiso}]

Fix $\eta\in (0,1)$, $r\in (0,1)$ and let $I\in\N$ be such that $2^{-I-1}\le r\le 2^{-I}$\,.
Notice that
\begin{equation}\label{newf}
k^s_r(|z|)\ge (2^{d+s})^{\min\{i,I\}}\qquad\textrm{ if }0\le |z|\le 2^{-i}\,, \textrm{ with }i\in\N\,.
\end{equation}
Let $\phi$ and $\phi_\delta$ (for every $\delta>0$) be as in Lemma \ref{lm:l1ker}. 
Now we claim that there exists a constant $C(\phi,d,s)$ such that 
\begin{equation}\label{stimasututti}
k^s_r(|z|)\ge C(\phi,d,s)\sum_{i=0}^{I}(2^s)^i\phi_{2^{-i}}(z)\quad\textrm{for every }z\in\R^d\,.
\end{equation}
Before proving the claim we show that \eqref{stimasututti} implies \eqref{f:quasiiso}.
Indeed, first notice that
 \begin{equation*}
 \frac{|\log r|}{\log 2}-1\le I\le \frac{|\log r|}{\log 2}
 \end{equation*} 
and hence
\begin{equation*}
\begin{aligned}
\sum_{i=0}^{I}(2^{s-1})^i=\left\{\begin{array}{ll}
I+1\ge \frac{|\log r|}{\log 2}&\textrm{ if }s=1\,,\\
\frac{(2^{I+1})^{s-1}-1}{2^{s-1}-1}\ge \frac{r^{1-s}-1}{2^{s-1}-1}&\textrm{ if }s>1\,.
\end{array}\right.
\end{aligned}
\end{equation*}
so that, recalling the very definition of $\sca^s(r)$ in \eqref{scalingper}, for $r$ small enough we have
\begin{equation}\label{persomI}
\sum_{i=0}^{I}(2^{s-1})^i\ge C(d,s)\sca^s(r)\,.
\end{equation}
Therefore, by applying \eqref{stimasututti} and Lemma \ref{lm:l1ker} with $\delta$ replaced by $2^{-i}$, we get
\begin{equation}\label{endproof}
\begin{aligned}
&\int_{A}\int_{\Omega\setminus A}k_r^s(|x-y|)\ud y\ud x\\
\ge& C(\Omega, \phi,d,s)\sum_{i=0}^I(2^{s-1})^i\, 2^i\int_{A}\int_{\Omega\setminus A}\phi_{2^{-i}}(x-y)\ud y\ud x\\
\ge& C(\Omega,\phi,d,s,\eta) \sum_{i=0}^I(2^{s-1})^i\ge C(\phi,d,s,\eta)\sca^s(r)\,,
\end{aligned}
\end{equation}
where the last inequality follows from \eqref{persomI}.

Now we prove the claim \eqref{stimasututti}. Suppose first that $0\le |z|\le 2^{-I}$. 
By applying \eqref{newf} with $i=I$, we get
\begin{equation}\label{psst}
\begin{aligned}
\sum_{i=0}^{I}(2^{s})^i\phi_{2^{-i}}(z)\le&\sup\phi\sum_{i=0}^I(2^{d+s})^i=\sup\phi\sum_{i=0}^I\frac{1}{(2^{d+s})^{I-i}}(2^{d+s})^I\\
\le&\sup\phi\sum_{j=0}^{+\infty}\frac{1}{(2^{d+s})^j}\, (2^{d+s})^I
= \frac{2^{{d+s}}}{2^{{d+s}}-1}\sup\phi\, (2^{d+s})^I\\
\le& C(\phi,d,s) k^s_r(|z|)\,.
\end{aligned}
\end{equation}
Analogously, if $2^{-\bar\imath-1}<  |z|\le 2^{-\bar\imath}$ for some $\bar\imath=0,1,\ldots, I-1$\,, using that $\phi_{2^{-i}}(z)=0$ for every $i=\bar\imath+1,\ldots,I$\,, we have
\begin{equation}\label{ssst}
\begin{aligned}
\sum_{i=0}^{I}(2^{s})^i\phi_{2^{-i}}(z)=&\sum_{i=0}^{\bar\imath}(2^{s})^i\phi_{2^{-i}}(z)
\le \sup\phi\sum_{i=0}^{\bar\imath}(2^{d+s})^i
\\
\le&\sup\phi\sum_{j=0}^{+\infty}\frac{1}{(2^{d+s})^j}(2^{d+s})^{\bar\imath} = \frac{2^{{d+s}}}{2^{{d+s}}-1}\sup\phi\,(2^{d+s})^{\bar\imath}\\
\le& C(\phi,d,s)k^s_r(|z|)\,,
\end{aligned}
\end{equation}
where the last inequality is a consequence of \eqref{newf}.

Finally, if  $|z|\ge 1$ we have that $\phi_{2^{-i}}(z)=0$ for every $i$ so that
\begin{equation}\label{tsst}
\sum_{i=0}^I (2^{s})^i\phi_{2^{-i}}(z)=0\le k^s_r(|z|)\,.
\end{equation}
Therefore, by \eqref{psst}, \eqref{ssst} and \eqref{tsst},  we deduce \eqref{stimasututti}, thus concluding the proof of the lemma\,.
\end{proof}
%%%%%%%%%%%%%%%%%%%%%%%%%%%%%
%%%%%%%%%%%%%%%%%%%%%%%%%%%%%
%%%%%%%%%%%%%%%%%%%%%%%%%%%%%
\subsection{Proof of Theorem \ref{mainthm}(i)}\label{maincomp}
We are now in a position to prove Theorem \ref{mainthm}(i).
%%%%%%%%%%%%%%%%%%%%%%%%%%%%%
%%%%%%%%%%%%%%%%%%%%%%%%%%%%%
%%%%%%%%%%%%%%%%%%%%%%%%%%%%%
\begin{proof}
We divide the proof into three steps.
\vskip4pt
\textit{Step 1.} Let $\alpha\in (0,1)$ and set $l_n:= r_n^{\alpha}$ for every $n\in\N$\,.
Let $ \{ Q_h^{n} \}_{h \in \mathbb{N}}$ be a disjoint family of cubes of sidelength $l_n$ such that $\bigcup_{h \in \mathbb{N}} Q_h^{n}= \mathbb{R}^d$. Since $\vert E_n \vert \le |U|$, there exists $H(n)\in\N$, such that, up to permutation of indices,
\begin{equation}\label{famigliacubi}
\begin{aligned}
\vert Q_h^{n}\cap E_n \vert \geq \frac{l_n^d}{2} \qquad\textrm{for every }h=1, \cdots, H(n),\\
\vert Q_h^n \setminus E_n \vert >\frac{l_n^d}{2} \qquad\textrm{for every } h\ge H(n)+1\,.
\end{aligned}
\end{equation}
For every $n \in \mathbb{N}$, we set 
\begin{equation*}%\label{defEtilde}
\tilde{E}_n:= \bigcup_{h=1}^{H(n)} Q_h^n\,.
\end{equation*} 
Let $\tilde{n} \in \mathbb{N}$ be such that for all $n>\tilde{n}$ the pair $(r_n,l_n)$ satisfies the hypothesis of Lemmas \ref{lemmatipofracpoinc} and \ref{lemmastimacubi}. 
We claim that there exists a constant $C(d,s)>0$ such that
\begin{equation}\label{difsimmEeEtild}
\vert \tilde{E}_n \triangle E_n \vert
  \leq C(d,s) l_n^s \sca^s(r_n) M\qquad\textrm{for every }n\ge\tilde n,
\end{equation}
where $M$ is the constant in \eqref{energy_bound}.
Indeed,
\begin{equation*}
\begin{split}
\vert E_n \triangle \tilde{E}_n \vert= &  \vert \tilde{E}_n \setminus E_n \vert+  \vert E_n \setminus \tilde{E}_n \vert  \\
=& \sum_{h=1}^{H(n)} \vert Q_h^n \setminus E_n \vert+ \sum_{h=H(n)+1}^{\infty} \vert E_n \cap Q_h^{n} \vert \\
=& 2\sum_{h=1}^{H(n)} \frac{1}{l_n^d}\vert Q_h^n \setminus E_n \vert \frac{l_n^d}{2}+ 2\sum_{h=H(n)+1}^{\infty} \frac{1}{l_n^d}\vert E_n \cap Q_h^{n} \vert \frac{l_n^d}{2}\\
\le& 2\sum_{h =1}^{+\infty} \frac{1}{l_n^d}\vert  Q_h^n\setminus E_n \vert \vert Q_h^n \cap E_n \vert  \\
\le& 2\sum_{h=1}^{+\infty} \int_{ Q_h^n} \bigg| \chi_{E_n}(x)- \frac{1}{\vert  Q_h^n \setminus B(x,r_n) \vert} \int_{ Q_h^n \setminus B(x,r_n)} \chi_{E_n}(y)\ud y \bigg|\ud x  \\
\le&  \sum_{h=1}^{+\infty} 8 d^{\frac{d+s}{2}} l_n^s \int_{Q_h^n \cap E_n} \int_{Q_h^n \setminus E_n } k^s_{r_n}(\vert x-y \vert)\ud y\ud x \\
\le& C(d,s) l^s_{n}  \tildeJn(E_n) \leq C(d,s) l_{n}^s\sca^{s}(r_n) M,
\end{split}
\end{equation*}
where the second inequality follows by formula \eqref{stimaneicubi}, the third inequality is a consequence of \eqref{frazpoinctipo}, whereas the last one follows directly by \eqref{energy_bound}.
%%%%%%%%%%%%%%%%%%%%%%%%%%%%%%%%%%%
\vskip4pt
\textit{Step 2.} 
For every $n\in\N$ let $ l_n$  and  $\tilde{E}_n:= \bigcup_{h=1}^{H(n)}Q^n_h$ be as in Step 1.
 We claim that there exists a constant $C(\alpha,d,s)$ such that for $n$ large enough
\begin{equation}\label{limitatezPcubet}
 \Per(\tilde{E}_n) \leq C(\alpha,d,s) \Hrn(E_n)\,.
\end{equation}
To ease notation, we omit the dependence on $n$ by setting $r:=r_n$\,, $l:=l_n$\,, $E:=E_n$\,, $Q_h:=Q^n_h$\,, $H:=H(n)$,  and  $\tilde E:=\tilde E_n$\,.

We define the family $\mathcal R$ of rectangles $R=\tilde Q\cup\hat Q$ such that $\tilde Q$ and $\hat Q$ are adjacent cubes (of the type $Q_h$ introduced above),
 $\tilde Q\subset\tilde E$ and $\hat Q\subset \tilde E^c$\,.

Notice that
\begin{equation}\label{primastimaPscubet}
\begin{aligned}
\Per(\tilde{E})\leq & 2d l^{d-1}\sharp\mathcal R\,,\\
 \Hr(E)\ge &\frac{1}{2d\,\sca^s(r)}\sum_{R\in\mathcal R}\int_{R\cap E}\int_{R\setminus E}k^s_r(|x-y|)\ud y\ud x\,.
 \end{aligned}
\end{equation} 
We recall that, by Lemma \ref{quasiiso}, for every rectangle $\bar R$ given by the union of two adjacent unitary cubes in $\R^d$, there exists $\rho_0>0$ such that 
\begin{equation}\label{applem}
\begin{aligned}
 C(d,s):=& \inf \biggl\{ \frac{1}{\sca^s(\rho) }\int_{F} \int_{\bar R\setminus F} k^s_{\rho}(\vert x-y \vert)\ud y \ud x: \\
&\qquad 0<\rho<\rho_0,\; F\in \Mf\,,\; F\subset \bar R\,, \; \frac 1 2\le \vert F \vert\le  \frac{3}{2} \biggr\}>0\,.
\end{aligned}
\end{equation}
Furthermore, by the very definition of $\sca^s(r)$ in  \eqref{scalingper}, using that $l=r^{\alpha}$ we have
\begin{equation*}
\begin{aligned}
\frac {\sca^s(r)}{l^{1-s}} =&\left\{\begin{array}{ll}
\displaystyle \frac{|\log (r^{1-\alpha})|} {1-\alpha}&\textrm{if }s=1\\ 
\displaystyle \frac{d+s}{d+1}\frac{r^{(1-\alpha)(1-s)}}{s-1}&\textrm{if }s>1
\end{array}\right.\\
=&\left\{\begin{array}{ll}
\displaystyle \frac{1}{1-\alpha}\sca^s(r^{1-\alpha})&\textrm{if }s=1\\
\displaystyle \sca^{s}(r^{1-\alpha})&\textrm{if }s>1\,,
\end{array}
\right.
\end{aligned}
\end{equation*}
so that
\begin{equation}\label{scalingl}
\frac{l^{1-s}}{\sca^{s}(r)}\ge C(\alpha)\frac{1}{\sca^s(r^{1-\alpha})}=C(\alpha)\frac{1}{\sca^s(\frac r l)}\,.
\end{equation}
For every set $O\in\Mf$ we set $O^l:=\frac{O}{l}$.
By \eqref{primastimaPscubet}, \eqref{scaling}, \eqref{scalingl} and by applying \eqref{applem} with $\bar R=R^l$ for every $R\in\mathcal R$, for $r$ small enough we obtain
\begin{equation*}%\label{dimstimaperimetrocubi}
\begin{split}
 \Hr(E)\ge 
& \frac{C(d)}{\sca^s(r) }l^{2d}\sum_{R\in\mathcal R} \int_{R^l\cap E^l}\int_{R^l\setminus E^l}k^s_{r}(|l(x-y)|)\ud y\ud x
\\
= & C(d)\frac{l^{1-s}}{\sca^s(r) }l^{d-1}\sum_{R\in\mathcal R} \int_{R^l\cap E^l}\int_{R^l\setminus E^l}k^s_\frac{r}{l}(|x-y|)\ud y\ud x\\
 \ge &C(\alpha,d)l^{d-1}\sum_{R\in\mathcal R} \frac{1}{\sca^s(\frac r l)} \int_{R^l\cap E^l}\int_{R^l\setminus E^l}k^s_\frac{r}{l}(|x-y|)\ud y\ud x\\
\ge & C(\alpha,d)l^{d-1}\sharp\mathcal R \,C(d,s)
 \ge C(\alpha,d,s) \Per(\tilde{E})\,,
\end{split}
\end{equation*}
i.e., \eqref{limitatezPcubet}.
%%%%%%%%%%%%%%%%%%%%%%%%%%%%%%%%%%%%%%%%%%
\vskip4pt
\textit{Step 3.} Here we conclude the proof of the compactness result.
We fix $\alpha\in (1-\frac 1 s,1)$ so that, by \eqref{difsimmEeEtild}, $|E_n\triangle\tilde E_n|\to 0$ as $n\to +\infty$\,.

By assumption and by the very definition of $\tilde{E}_{n}$ in Step 1, we have that $ \tilde{E}_n \subset U $ for all $n \in \mathbb{N}$. Moreover, by formula \eqref{limitatezPcubet} and by \eqref{energy_bound} for $n$ large enough we have 
$$
\Per(\tilde{E}_n)\leq C(\alpha,d,s)\Hrn(E_n)\leq C(\alpha,d,s)M\,.
$$
It follows that the sequence $\{\chi_{\tilde{E}_n} \}_{n \in \mathbb{N}}$ satisfies the assumption of Theorem \ref{compatezzainBV}, and hence there exists a set $E \subset \mathbb{R}^d$ with $\Per(E)<+\infty$ such that, up to a subsequence, $\chi_{\tilde{E}_{n}} \rightarrow \chi_{E}$ in $\mathrm{L}^1(\R^d)$ as $n\to +\infty$.
Since $|E_n\triangle\tilde E_n|\to 0$ as $n\to +\infty$ we obtain that $\chi_{E_{n}} \rightarrow \chi_{E}$ in $\mathrm{L}^1(U)$, i.e., the claim of Theorem \ref{mainthm}(i).
\end{proof}
%%%%%%%%%%%%%%%%%%%%%%%%%%%%%
%%%%%%%%%%%%%%%%%%%%%%%%%%%%%
%%%%%%%%%%%%%%%%%%%%%%%%%%%%%
The following result follows by the proof of Theorem \ref{mainthm}(i).
%%%%%%%%%%%%%%%%%%%%%%%%%%%%%
%%%%%%%%%%%%%%%%%%%%%%%%%%%%%
%%%%%%%%%%%%%%%%%%%%%%%%%%%%%
\begin{proposition}\label{corcompthm}
	Let $s \geq 1$. Let $ \{ r_{n}\}_{n \in \mathbb{N}} \subset (0,+\infty)$ be such that $r_n \rightarrow 0^+$ as $n \rightarrow +\infty$. Let $\{ E_n\}_{n \in \mathbb{N}} \subset \mathrm{M}_{f}(\mathbb{R}^d)$ be such that $ \chi_{E_n} \rightarrow \chi_{E}$ in $\mathrm{L}^1(\mathbb{R}^d)$ as $n \rightarrow + \infty$, for some $ E \in \mathrm{M}_f(\mathbb{R}^d)$. If
	$$ 
	\limsup_{n\to +\infty} \frac{\tildeJn(E_n)}{\sigma^s(r_n)}  \le M\,,
	$$
	then $E$ has finite perimeter.
\end{proposition}
%%%%%%%%%%%%%%%%%%%%%%%%%%%%%
%%%%%%%%%%%%%%%%%%%%%%%%%%%%%
%%%%%%%%%%%%%%%%%%%%%%%%%%%%%
\begin{proof}
The proof of this corollary is fully analogous to the proof of Theorem \ref{mainthm}(i), and we adopt the same notation introduced there. 
Arguing as in the proof of Steps 1 and 2 we have that for $n$ large enough
$$
\Per(\tilde{E}_n) \leq C(\alpha,d,s) \limsup_{n\to +\infty} \frac{\tildeJn(E_n)}{\sigma^s(r_n)}  \le C(\alpha,d,s) M\,, $$
and that if $\alpha\in (1-\frac 1 s, 1)$\,, then $|\tilde E_n\triangle E_n|\to 0$ as $n\to +\infty$.
By assumption, this implies that
$$ \chi_{\tilde{E}_n} \rightarrow \chi_{E}, \; \text{ in $ \mathrm{L}^{1}(\mathbb{R}^d)$} \quad n \rightarrow + \infty\,,
$$
and by the lower semicontinuity of the perimeter, 
$$
\Per(E)\le \liminf_{n\to +\infty}\Per(\tilde E_n)\le C(\alpha,d,s) M\,.
$$
\end{proof}
%%%%%%%%%%%%%%%%%%%%%%%%%%%%%%%%%%%%%%%%%%
%%%%%%%%%%%%%%%%%%%%%%%%%%%%%%%%%%%%%%%%%%
%%%%%%%%%%%%%%%%%%%%%%%%%%%%%%%%%%%%%%%%%%
\section{Proof of the $\Gamma$-limit}\label{sc:gammalim}
This section is devoted to the proofs of Theorem \ref{mainthm}(ii) and (iii), which are the content of Subsections \ref{lowerb} and \ref{upperb} respectively.

\subsection{Proof of the lower bound}\label{lowerb}
 The proof of  Theorem \ref{mainthm}(ii) closely follows the strategy used in \cite{GM05}.
We recall that for every $\nu\in\mathbb{S}^{d-1}$\,, $Q^{\nu}$ is a unit square centered at the origin with one face orthogonal to $\nu$. Moreover, we recall that $H^+_\nu(0)=\{x\in\R^d\,:\,x\cdot \nu\ge 0\}$.

The following result is the adaptation to our setting of \cite[Lemma 18]{GM05}.
%%%%%%%%%%%%%%%%%%%
%%%%%%%%%%%%%%%%%%%
%%%%%%%%%%%%%%%%%%%
\begin{lemma}\label{lemma18}
Let $s\ge 1$\,.
For every $\ep>0$, there exist $r_0,\delta_0>0$ such that for every $\nu\in\mathbb{S}^1$,  for every $E\in\Mf$ with 
\begin{equation}\label{hp}
|(E\triangle H^-_\nu(0))\cap Q^\nu|\le\delta_0\,,
\end{equation}
and for every $r<r_0$ it holds
\begin{equation}\label{quasiliminf}
\int_{Q^{\nu}\cap E}\int_{Q^\nu\cap E^c} k^s_r(|x-y|) \ud y\ud x\ge \omega_{d-1}(1-\ep)\sca^s(r)\,.
\end{equation}
\end{lemma}
%%%%%%%%%%%%%%%%%%%
%%%%%%%%%%%%%%%%%%%
%%%%%%%%%%%%%%%%%%%
\begin{proof}
Up to a rotation, we can assume that $\nu=-e_d$ so that $Q^\nu\equiv Q=[-\frac 1 2,\frac 1 2)^d$ and $H^{-}_\nu(0)=:\R^d_+$. 
Let $0<r<1$. 
We can assume without loss of generality that $E\subset Q$\,.
Using the change of variable $y=x+z$ we have
\begin{equation}\label{start}
\begin{aligned}
&\int_{Q\cap E^c} \ud x\int_{Q\cap E} k^s_r(|x-y|)\ud y\\
=&\int_{Q\cap E^c}\ud x\int_{\{z\in\R^d\,:\,x+z\in E\}}k^s_r(|-z|)\ud z\\
=&\int_{Q\cap E^c}\ud x\int_{\R^d}k^s_r(|z|)\chi_{E}(x+z)\ud z\\
=& \int_{\R^d}k^s_r(|z|)\int_{\R^d}\chi_{E^c\cap Q}(x)\chi_{E}(x+z)\ud x\ud z\\
=&
\int_{\R^d}k^s_r(|z|)|E^c\cap (E-z)\cap Q|\ud z= \int_{\R^d}k^s_r(|z|)m(z)\ud z\,,
\end{aligned}
\end{equation}
where we have set $m(z):=|E^c\cap (E-z)\cap Q|$\,.

Let $\frac 1 2 <\lambda <1$  and let $z\in\R^d$
be such that $|z|_\infty\le \frac{1-\lambda}{2}$ and $z_d>0$. Since $|(E-z)\cap \lambda Q|=|E\cap (\lambda Q+z)|$, by triangular inequality, we get
\begin{equation*}%\label{transl}
\begin{aligned}
&|(E-z)\cap \lambda Q|-|E\cap \lambda Q|=\int_{\lambda Q+z}\chi_{E}\ud x-\int_{\lambda Q}\chi_{E}\ud x\\
\ge&\int_{\lambda Q+z}\chi_{\R^d_+}\ud x-\int_{\lambda Q}\chi_{\R^d_+}\ud x-\int_{(\lambda Q+z)\triangle \lambda Q}|\chi_E-\chi_{\R^d_+}|\ud x\\
\ge&\lambda^{d-1}z_d-\int_{U_{\lambda,z}}|\chi_E-\chi_{\R^d_+}|\ud x\,,
\end{aligned}
\end{equation*}
where we have set $U_{\lambda,z}:=(\lambda+|z|_\infty)Q\setminus (\lambda-|z|_{\infty})Q$ and we have used that $(\lambda Q+z)\triangle \lambda Q\subset U_{\lambda,z}$. As a consequence,  we deduce that
\begin{equation}\label{est1}
\begin{aligned}
m(z)=&|E^c\cap (E-z)\cap Q|\ge|E^c\cap (E-z)\cap \lambda Q|\\
\ge& |E^c\cap \lambda Q|+|(E-z)\cap \lambda Q|-|\lambda Q|\\
\ge& |E^c\cap \lambda Q|+|E\cap \lambda Q|+\lambda^{d-1}z_d-\int_{U_{\lambda,z}}|\chi_E-\chi_{\R^d_+}|\ud x-|\lambda Q|\\
=& \lambda^{d-1}z_d-\int_{U_{\lambda,z}}|\chi_E-\chi_{\R^d_+}|\ud x\, ,
\end{aligned}
\end{equation}
where the last equality follows by noticing that  $|E\cap \lambda Q|+|E^c\cap \lambda Q|=|\lambda Q|$.

Let now  $0<\delta_0<\frac{1}{64}$ to be chosen later on  and 
 set 
\begin{equation*}%\label{a+delta0}
A^+_{\sqrt{\delta_0}}:=\Big\{z\in\R^d\,:\,|z|_\infty\le \frac{\sqrt{\delta_0}}{2},\,z_d>0\Big\}\,.
\end{equation*}
We fix $z\in A^+_{\sqrt{\delta_0}}$ and we set  $J:=\lfloor \frac{\sqrt{\delta_0}}{|z|_\infty}\rfloor$\,.
We set $\lambda_0:=1-4\sqrt{\delta_0}$  and we cover $(\lambda_0+2J|z|_\infty)Q\setminus \lambda_0 Q$ with $J$ squared annuli of thickness $2|z|_\infty$, namely we set $\lambda_j:=\lambda_0+2j|z|_\infty$ and $U_{j}:=\lambda_j Q\setminus \lambda_{j-1} Q$ for $j=1,\ldots,J$\,. 
Moreover,  we set  $\tilde\lambda_j:=\lambda_0+(2j-1)|z|_\infty$
for every $j=1,\ldots,J$ and we notice that $\frac 1 2<\tilde\lambda_j<1$ for every $j=1,\ldots,J$\,.
Since $z\in A^+_{\sqrt{\delta_0}}$, we have that $|z|_\infty\le \frac{1-\tilde \lambda_J}{2}\le \frac{1-\tilde \lambda_j}{2}$ for every $j=1,\ldots,J$\,.
Therefore, for every $j=1,\ldots, J$ we can apply \eqref{est1} with $\lambda=\tilde\lambda_j$ in order to get
\begin{equation}\label{spero}
\begin{aligned}
m(z)\ge&  z_d \tilde\lambda_{j}^{d-1}-\int_{U_{\tilde\lambda_j,z}}|\chi_E-\chi_{\R^d_+}|\ud x\\
\ge & z_d \lambda_{j-1}^{d-1}-\int_{U_j}|\chi_E-\chi_{\R^d_+}|\ud x \, ,
\end{aligned}
\end{equation}
where we have used also that $\tilde\lambda_j-|z|_\infty=\lambda_{j-1}$ and $\tilde\lambda_j+|z|_\infty=\lambda_j$ so that $U_{\tilde\lambda_j,z}=U_j$\,.
Summing \eqref{spero} over $j=1,\ldots,J$ we get
$$
Jm(z)\ge z_d\sum_{j=1}^J\lambda_{j-1}^{d-1}-\int_{Q}|\chi_E-\chi_{\R^d_+}|\ud x\,,
$$
which, dividing by $J$ and using discrete Jensen inequality (namely, convextiy), yields
\begin{equation}\label{est2}
m(z)\ge z_d\Big(\frac 1 J\sum_{j=1}^J\lambda_{j-1}\Big)^{d-1}-\frac 1 J \int_{Q}|\chi_E-\chi_{\R^d_+}|\ud x\ge z_d\lambda_0^{d-1}- 2|z|_\infty\sqrt{\delta_0}\,,
\end{equation}
where in the last inequality we have used \eqref{hp} and the fact that $J\ge \frac{\sqrt{\delta_0}}{|z|_\infty}-1$. 
Therefore, we have proven that \eqref{est2} holds true whenever $z\in A^+_{\sqrt{\delta_0}}$\,, which combined with \eqref{start}, yields
\begin{multline}\label{start2}
\int_{Q\cap E^c} \ud x \int_{Q\cap E} k^s_r(|x-y|) \ud y\\
\ge \lambda_0^{d-1}\int_{A^+_{\sqrt{\delta_0}}}z_dk^s_r(|z|)\ud z- 2\sqrt{\delta_0}\int_{A^+_{\sqrt{\delta_0}}}|z|_\infty k^s_r(|z|)\ud z\,.
\end{multline} 
As for the first integral on the right hand side of \eqref{start2}, by using polar coordinates $z=\rho\theta$ with $\rho>0$ and $\theta\in\mathbb{S}^{d-1}$  and using the very definition of $\sca^s(r)$ in \eqref{scalingper}, for $\delta_0$ small enough and for all $r<\delta_0$ we have
\begin{equation}\label{start3}
\begin{aligned}
 \int_{A^+_{\sqrt{\delta_0}}}z_dk^s_r(|z|)\ud z\ge&\int_{B(0,r)\cap \R^d_+}\frac{z_d}{r^{d+s}}\ud z+\int_{(B(0,\delta_0)\setminus B(0,r))\cap \R^d_+}\frac{z_d}{|z|^{d+s}}\ud z\\
 =&\frac{1}{r^{d+s}}\int_{0}^{r}\rho^{d}\ud\rho\int_{\mathbb{S}^{d-1}\cap \R^d_+}\theta_d\ud\mathcal{H}^{d-1}(\theta)\\
 &+\int_{r}^{\delta_0}\rho^{-s}\ud\rho\int_{\mathbb{S}^{d-1}\cap \R^d_+}\theta_d\ud\mathcal{H}^{d-1}(\theta)\\
 =&\omega_{d-1}\frac{r^{1-s}}{d+1}+\omega_{d-1}\int_{r}^{\delta_0}\rho^{-s}\ud\rho\\
 \ge&\omega_{d-1}\sca^s(r)-\omega_{d-1}C(\delta_0,s)\,,\\
\end{aligned}
\end{equation}
where 
\begin{equation*}%\label{cospos}
C(\delta_0,s):=\left\{\begin{array}{ll}
|\log\delta_0|&\textrm{if }s=1\\
\displaystyle \frac{\delta_0^{1-s}}{s-1}&\textrm{if }s>1\,.
\end{array}\right.
\end{equation*}

Moreover, since $|z|_{\infty}\le |z|$, it holds
 \begin{equation}\label{sint}
 \begin{aligned}
 &\int_{A^+_{\sqrt{\delta_0}}}|z|_\infty k^s_r(|z|)\ud z\le \int_{B(0,1)}|z|k^s_r(|z|)\ud z\\
 =&\frac{1}{r^{d+s}}\int_{B(0,r)}|z|\ud z+\int_{B(0,1)\setminus B(0,r)}\frac{1}{|z|^{d+s-1}}\ud z
 \le&C(d,s)\sca^s(r)\,,
  \end{aligned}
 \end{equation}
for some $C(d,s)>0$\,.
 
Now we define the function $\eta(t):=1-(1-4\sqrt{t})^{d-1}$\,, and we notice that 
$\eta(t)\to 0$ as $t\to 0^+$\,.
Therefore, by \eqref{start2}, \eqref{start3}  and \eqref{sint},
using that $\lambda_0^{d-1}=1-\eta(\delta_0)$\,,
we deduce that
\begin{equation}\label{incr2}
\begin{aligned}
&\int_{Q\cap E^c} \ud x \int_{Q\cap E} k^s_r(|x-y|) \ud y\\
\ge& \omega_{d-1}\sca^s(r) \Big(1-\eta(\delta_0)-\big(1-\eta(\delta_0)\big)\frac{C(\delta_0,s)}{\sca^s(r)}-2\sqrt{\delta_0}\frac{C(d,s)}{\omega_{d-1}}\Big)\,,
\end{aligned}
\end{equation}
so that, choosing $\delta_0>0$ such that 
$$
\eta(\delta_0)+2\sqrt{\delta_0}\frac{C(d,s)}{\omega_{d-1}}\le \frac\ep 2
$$
and $r_0>0$ such that (for every $0<r<r_0$)
$$
 \big(1-\eta(\delta_0)\big)\frac{C(\delta_0,s)}{\sca^s(r)}\le  \big(1-\eta(\delta_0)\big)\frac{C(\delta_0,s)}{\sca^s(r_0)}\le \frac{\ep} 2\,,
$$
by \eqref{incr2} we deduce \eqref{quasiliminf}, thus concluding the proof of the lemma.
\end{proof}
%%%%%%%%%%%%%
%%%%%%%%%%%%%
%%%%%%%%%%%%%
We are now in a position to prove the $\Gamma$-liminf inequality in Theorem \ref{mainthm}.
%%%%%%%%%%%%%
%%%%%%%%%%%%%
%%%%%%%%%%%%%
\begin{proof}[Proof of Theorem \ref{mainthm}(ii)]
We can assume without loss of generality that
\begin{equation}\label{ben0}
\Hrn(E_n)=\frac{1}{2\sca^s(r_n)}\int_{\R^d}\int_{\R^d}k^s_{r_n}(|x-y|)|\chi_{E_n}(x)-\chi_{E_n}(y)|\ud y\ud x\le C\,,
\end{equation}
for some constant $C>0$ independent of $n$. Then, by Corollary \ref{corcompthm} we have that $E$ has finite perimeter.
For every $n\in\N$ let $\mu_n$ be the measure on the product space $\R^d\times \R^d$ defined by
$$
\mu^s_n(A\times B):=\frac{1}{2\sca^s(r_n)}\int_{A}\int_{B}k^s_{r_n}(|x-y|)|\chi_{E_n}(x)-\chi_{E_n}(y)|\ud y\ud x
$$
for every $A,B\in\mathrm{M}(\R^d)$. Then by \eqref{ben0}, up to a subsequence, $\mu_{n}^{s}\weakstar\mu^s$ for some measure $\mu^s$. Now we show that $\mu^s$ is concentrated on the set $D:=\{(x,x)\,:\,x\in\R^d\}$\,, i.e., that $\mu^s(\Omega)=0$ if $\Omega\cap D=\emptyset$. 
Indeed, let $\ffi\in C_c(\R^d\times \R^d;[0,+\infty))$ be such that $\di(\supp\ffi, D)=\delta$ for some $\delta>0$\,; then
\begin{equation*}
\begin{split}
&\int_{\R^d\times \R^d}\ffi(x,y)\ud \mu^s(x,y)\\
=&\lim_{n\to +\infty}\frac{1}{2\sca^s(r_n)}\int_{\R^d\times \R^d}\ffi(x,y)k^s_{r_n}(|x-y|)|\chi_{E_n}(x)-\chi_{E_n}(y)|\ud y\ud x\\
\le &\lim_{n\to +\infty}\frac{1}{2\sca^s(r_n)}\frac{1}{\delta^{d+s}}\int_{\R^d\times \R^d}\ffi(x,y)\ud y\ud x=0\,.
\end{split}
\end{equation*}
Now we define the measure $\lambda^s$ on $\R^d$ as $\lambda^s(A)=\mu^s(\{(x,x)\,:\,x\in A\})$
and we claim that 
for $\mathcal H^{d-1}$ - a.e. $x_0\in\partial^* E$ it holds
\begin{equation}\label{tesi}
\liminf_{l\to 0^+}\frac{\lambda^s(\overline{Q}_l^\nu(x_0))}{l^{d-1}}\ge\liminf_{l\to 0^+}\liminf_{n\to +\infty}\frac{\mu^s_n({Q}_l^{\nu}(x_0)\times {Q}_l^{\nu}(x_0))}{l^{d-1}}\ge \omega_{d-1}\,,
\end{equation}
where we have set $\nu=\nu_E(x_0)$ and $Q_l^{\nu}(x_0)=x_0+lQ^\nu$. 
By \eqref{tesi} and Radon-Nikodym Theorem, using the lower semicontinuity of the total variation of measures with respect to the weak star convergence,  we get \eqref{trueliminf}.

We conclude by proving the claim \eqref{tesi}. We preliminarily notice that the first inequality is a consequence of the upper semicontinuity of the total variation of measures on compact sets with respect to the weak star convergence. We pass to prove the second inequality in \eqref{tesi}.
For all $x_0\in\partial^* E$, we have
\begin{equation}\label{blowup}
\lim_{l\to 0^+}\int_{Q^\nu}|\chi_{E}(x_0+lx)-\chi_{H^-_\nu(0)}(x)|\ud x=0\,.
\end{equation}
Fix such a $x_0\in\partial^* E$. We will adopt a blow-up argument. Consider the sequence of sets $\{F_{n,l}\}_{n\in\N}$ defined by $F_{n,l}=x_0+l E_n$. By the change of variable $x=x_0+l\xi$ and $y=x_0+l\eta$ we have
\begin{equation}\label{chavar}
\begin{aligned}
&\frac{1}{l^{d-1}}\mu^s_n(Q^\nu_l(x_0)\times Q^\nu_l(x_0))\\
=&\frac{1}{2\sca^s(r_n)}\int_{Q^\nu}\int_{Q^\nu}l^{d+1}k^s_{r_n}(|l\xi-l\eta|)|\chi_{F_{n,l}}(\xi)-\chi_{F_{n,l}}(\eta)|\ud\xi\ud\eta\\
=&\frac{l^{1-s}}{2\sca^s(r_n)} \int_{Q^\nu}\int_{Q^\nu}k^s_{\frac{r_n}{l}}(|\xi-\eta|) |\chi_{F_{n,l}}(\xi)-\chi_{F_{n,l}}(\eta)|\ud\xi\ud\eta\,,
\end{aligned}
\end{equation}
where in the last equality we have used \eqref{scaling}.
Let $0<\ep<1$ and let $\delta_0,r_0>0$ be the constants provided by Lemma \ref{lemma18}.  In view of \eqref{blowup} for $l$ small enough we have
\begin{equation}\label{tri1}
\int_{Q^\nu}|\chi_{E}(x_0+lx)-\chi_{H^{-}_{\nu}(0)}(x)|\ud x\le \frac{\delta_0}{2}\,.
\end{equation}
Fix such an $l$\,; then, there exists $n(l)\in\N$ such that for $n\ge n(l)$\,, it holds
\begin{equation}\label{tri2}
\int_{Q^\nu}|\chi_{F_{n,l}}(x)-\chi_{E}(x_0+lx)|\ud x=\frac{1}{l^{d}}\int_{Q_l^\nu(x_0)}|\chi_{E_n}(x)-\chi_E(x)|\ud x\le \frac{\delta_0}{2}\,.
\end{equation}
By \eqref{tri1} and \eqref{tri2}, using triangular inequality, we obtain
\begin{equation*}%\label{tri}
|(F_{n,l}\triangle H^{-}_{\nu}(0))\cap Q^\nu|=\int_{Q^\nu}|\chi_{F_{n,l}}-\chi_{H^{-}_{\nu}(0)}|\ud x\le \delta_0\,.
\end{equation*}
Therefore, by applying Lemma \ref{lemma18} with $k^s_r=k^s_{\frac{r_n}{l}}$ and $E=F_{n,l}$\,, for 
$n$ large enough (i.e., in such a way that $n\ge n(l)$ and
$r_n<r_0 l$) we have that
\begin{equation}\label{apple}
\frac 1 2\int_{Q^\nu}\int_{Q^\nu}k^s_{\frac{r_n}{l}}(|\xi-\eta|) |\chi_{F_{n,l}}(\xi)-\chi_{F_{n,l}}(\eta)|\ud\xi\ud\eta\ge \omega_{d-1}(1-\ep)\sca^s\Big(\frac{r_n}{l}\Big)\,.
\end{equation}
Now, by the very definition of $\sca^s$ in \eqref{scalingper},  we have that  
\begin{equation*}
\frac{l^{1-s}}{\sca^s(r_n)}\sca^s\Big(\frac{r_n}{l}\Big)=\left\{
\begin{array}{ll}
\frac{\log l+|\log r_n|}{|\log r_n|}&\textrm{if }s=1\\
1&\textrm{if }s>1\,,
\end{array}
\right.
\end{equation*}
so that, in view of \eqref{chavar} and \eqref{apple}, we deduce that for every $0<\ep<1$ and for every $l$ small enough (depending on $\ep$), it holds
\begin{equation*}
\liminf_{n\to +\infty}\frac{1}{l^{d-1}}\mu_n^s(Q^\nu_l(x_0)\times Q^\nu_l(x_0))\ge \omega_{d-1}(1-\ep)\,,
\end{equation*}
whence the second inequality in claim \eqref{tesi} follows by the arbitrariness of $\ep$.
\end{proof}
%%%%%%%%%%%%%
%%%%%%%%%%%%%
%%%%%%%%%%%%%
%%%%%%%%%%%%%
%%%%%%%%%%%%%
%%%%%%%%%%%%%
\subsection{Proof of the upper bound}\label{upperb}

The $\Gamma$-limsup inequality will be a consequence of Proposition \ref{prop:point} and of 
standard density results for sets of finite perimeter.

We first recall the following fundamental approximation theorem (see, for instance, \cite[Theorem 13.8]{Maggi}).
\begin{theorem}[Approximation of set with finite perimeter by smooth sets]\label{appro}
A set $E \in \Mf$ has finite perimeter if and only if there exists a sequence $\{F_k\}_{k \in \mathbb{N}} \subset \Mf$ of open bounded sets with smooth boundary, such that
	\begin{equation}\label{appro_formula}
    \begin{split}
	& \chi_{F_k} \rightarrow \chi_{E} \quad \text{ (strongly) in $\mathrm{L}^1(\R^d)$ as $k \rightarrow +\infty$,} \\
	& \Per(F_k) \rightarrow \Per(E) \quad \text{as $k \rightarrow +\infty$} .
	\end{split}
	\end{equation} 
\end{theorem}
%%%%%%%%%%%%%%%%%%%%%%%
%%%%%%%%%%%%%%%%%%%%%%%
%%%%%%%%%%%%%%%%%%%%%%%
\begin{proof}[Proof of Theorem \ref{mainthm}(iii)]
Let $E\in\Mf$ be a set with finite perimeter. By Theorem \ref{appro}, there exists a sequence $\{F_k\}_{k\in\N}$ of open bounded sets with smooth boundary satisfying \eqref{appro_formula}\,.
In view of Proposition \ref{prop:point} we have that
\begin{equation*}%\label{asinprop}
\lim_{n\to +\infty}\frac{\tildeJn(F_k)}{\sca^s(r_n)}=\omega_{d-1}\Per(F_k)\qquad\textrm{for every }k\in\N\,.
\end{equation*}
Therefore, by a standard diagonal argument there exists a sequence $\{E_n\}_{n\in\N}$ with $E_n=F_{k(n)}$ for every $n\in\N$ satisfying the desired properties.
\end{proof}
%%%%%%%%%%%%%%%%%%%%%%%
%%%%%%%%%%%%%%%%%%%%%%%
%%%%%%%%%%%%%%%%%%%%%%%
%%%%%%%%%%%%%%%%%%%%%%%
%%%%%%%%%%%%%%%%%%%%%%%
%%%%%%%%%%%%%%%%%%%%%%%
\subsection{Characterization of sets of finite perimeter}
As a byproduct of our $\Gamma$-convergence analysis, we prove that a set $E\in\Mf$ has finite perimeter if and only if for all $s \geq 1$
$$ \limsup_{r \rightarrow 0^+} \frac{\tildeJ(E)}{\sca^s(r)} < + \infty.$$ 

We recall the following classical theorem.
%%%%%%%%%%%%%%%%%%%%%%%%%%%%%%%%%%%%%%%%
%%%%%%%%%%%%%%%%%%%%%%%%%%%%%%%%%%%%%%%%
%%%%%%%%%%%%%%%%%%%%%%%%%%%%%%%%%%%%%%%%
\begin{theorem}[Characterization via difference quotients]\label{bvpertrasthm}
	Let $E \in\Mf$\,. Then $E$ has finite perimeter if and only if there exists $C>0$ such that
	\begin{equation*}%\label{setPFtras}
	\int_{\mathbb{R}^d} \vert \chi_{E}(x+z)-\chi_{E}(x) \vert \ud x \leq C \vert z \vert \qquad \textrm{for every } z \in \mathbb{R}^d.
	\end{equation*}
Specifically, it is possible to choose $C=\Per(E)$.
\end{theorem}
%%%%%%%%%%%%%%%%%%%%%%%%%%%%%%%%%%%%%%%%
%%%%%%%%%%%%%%%%%%%%%%%%%%%%%%%%%%%%%%%%
%%%%%%%%%%%%%%%%%%%%%%%%%%%%%%%%%%%%%%%%
\begin{theorem}\label{thmfondcarPer}
	Let $E \in\Mf$\,. The following statements hold true.
	\begin{itemize}
		\item[(i)] If  $\displaystyle  \limsup_{r \rightarrow 0^+} \frac{\tildeJ(E)}{\sca^s(r)} < + \infty$ for some $s\ge 1$\,, then $E$ is a set of finite perimeter.
		\item[(ii)] If $E$ is a set of finite perimeter then $\displaystyle  \limsup_{r \rightarrow 0^+} \frac{\tildeJ(E)}{\sca^s(r)} < + \infty$ for every $s\ge 1$\,.
	More precisely, 
	\begin{equation}\label{catdiscarseteng}
	 \omega_{d-1} \Per(E) \leq \liminf_{r \rightarrow 0^+} \frac{\tildeJ(E)}{\sca^s(r)} \leq \limsup_{r \rightarrow 0^+} \frac{\tildeJ(E)}{\sca^s(r)} \leq M(s,d) \Per(E),
	 \end{equation}
	 where
	 	\begin{equation*}%\label{M(s,d)carseteng}
	 M(s,d)=\left\{
	 \begin{array}{ll}
	 \frac{d \omega_{d}}{2}&\textrm{if }s=1\\
	 \omega_{d-1}&\textrm{if }s>1\,.
	 \end{array}
	 \right.
	 \end{equation*}
	 In particular, for  $s>1$ we have that 
	 \begin{equation}\label{carasma1}
	 \lim_{r\to 0^+} \frac{\tildeJ(E)}{\sca^s(r)} =\omega_{d-1}\Per(E)\,.
	 \end{equation}
	 \end{itemize}
\end{theorem}
%%%%%%%%%%%%%%%%%%%%%%%%%%%%%%%%%%%%%%%%
%%%%%%%%%%%%%%%%%%%%%%%%%%%%%%%%%%%%%%%%
%%%%%%%%%%%%%%%%%%%%%%%%%%%%%%%%%%%%%%%%
\begin{remark}
		 \rm{We notice that in the case $s=1$ the constant $M(1,d)=\frac{d \omega_{d}}{2} > \omega_{d-1}$\,,
	so that the existence of the limit \eqref{carasma1} is not proven in such a case.
	}
\end{remark}
%%%%%%%%%%%%%%%%%%%%%%%%%%%%%%%%%%%%%%%%
%%%%%%%%%%%%%%%%%%%%%%%%%%%%%%%%%%%%%%%%
%%%%%%%%%%%%%%%%%%%%%%%%%%%%%%%%%%%%%%%%
\begin{proof}[Proof Theorem \ref{thmfondcarPer}: ]
We notice that (i) is an immediate consequence of Proposition \ref{corcompthm} taking $E_n\equiv E$ for every $n\in\N$\,.
We prove (ii). The $\Gamma$-liminf inequality Theorem \ref{mainthm}(ii) implies the first inequality in \eqref{catdiscarseteng}. Being the second inequality obvious we pass to the proof of the last one.
	If $s>1$ then, by Proposition \ref{prop:point}, we have
	\begin{equation}\label{casos>1carPer}
	\lim_{r \to 0^+} \frac{\tilde{J}_r^s(E)}{\sca^s(r)}= \omega_{d-1} \Per(E).
	\end{equation}
	Let now $s=1$. Let $ G_r^1 $ be the functional defined in \eqref{G1r}; by Theorem \ref{bvpertrasthm} we obtain
	\begin{equation}\label{svilcarseteng}
	\begin{split}
	\frac{G_r^1(E) }{\sca^1(r)}=& 	\frac{1 }{\vert \log r \vert} \int_{E} \int_{E^c \cap B(x,1)} k_r^1(\vert x- y\vert)\ud y \ud x\\
	= & \frac{1 }{2\vert \log r \vert} \int_{\mathbb{R}^d} \int_{B(x,1)} \vert \chi_{E}(x)-\chi_{E}(y) \vert k_r^1(\vert x-y \vert)\ud y \ud x\\
	=&  \frac{1 }{2\vert \log r \vert}\int _{B(0,1)}  k_r^1(\vert h \vert) \int_{\mathbb{R}^d}\vert \chi_{E}(x+h)-\chi_{E}(x)\vert  \ud x \ud h\\
	\leq & \frac{1 }{2\vert \log r \vert}\Per(E) \int_{B(0,1)} \vert h \vert k_r^1(\vert h \vert) \ud h \\
	= &  \frac{d \omega_{d}}{2}\Per(E) \bigg(1+ \frac{1}{(d+1)\vert \log r \vert}\bigg).
	\end{split}
	\end{equation}
	Moreover, by Remark \ref{zeroor} we have that
	\begin{equation}\label{limsupzeroF1s(E)}
	\lim_{r \rightarrow 0^+} \frac{F_1^1(E)}{\sigma^1(r)}=0\,,
	\end{equation}
	where $F_1^1 $ is the functional defined in formula \eqref{F11}.
	
Therefore by formulas \eqref{deco}, \eqref{svilcarseteng}, and \eqref{limsupzeroF1s(E)} we have
\begin{equation}\label{limsuplimM(d,s)P}
\limsup_{r \rightarrow 0^+} \frac{\tilde{J}_r^1(E)}{\sca^1(r)}= \limsup_{r \rightarrow 0^+}\frac{G_r^1(E)}{\sca^1(r)} +\lim_{r \rightarrow 0^+}\frac{F_r^s(E)}{\sca^1(r)}
\leq  \frac{d \omega_{d}}{2}\Per(E)\,.
\end{equation}
thus concluding the proof of (ii).
By \eqref{casos>1carPer} and \eqref{limsuplimM(d,s)P} we conclude the proof of (ii).
\end{proof}
 %%%%%%%%%%%%%%%%%%%%%%%%%%%%%%%%%%%%%%%%
 %%%%%%%%%%%%%%%%%%%%%%%%%%%%%%%%%%%%%%%%
 %%%%%%%%%%%%%%%%%%%%%%%%%%%%%%%%%%%%%%%%
 %%%%%%%%%%%%%%%%%%%%%%%%%%%%%%%%%%%%%%%%
 %%%%%%%%%%%%%%%%%%%%%%%%%%%%%%%%%%%%%%%%
 %%%%%%%%%%%%%%%%%%%%%%%%%%%%%%%%%%%%%%%%
 %%%%%%%%%%%%%%%%%%%%%%%%%%%%%%%%%%%%%%%%
 \section{Convergence of curvatures and mean curvature flows}\label{curf}
 In this section we study the behavior of the non-local curvatures corresponding to the functionals $\tildeJ$ and of the corresponding geometric flows. Using the approach in \cite{CMP,CDNP}, it is enough to focus on smooth enough sets. To this purpose, we introduce  the class $\Reg$ as the class of the  subsets of $\R^d$, which are closures of open sets with  compact $C^2$ boundary. 
 Moreover, we define a notion of convergence in $\mathfrak{C}$ as follows.
 Let $\{E_n\}_{n\in\N} \subset \mathfrak{C}$ we say that $E_{n} \rightarrow E $ in $\mathfrak{C}$ as $n\to +\infty$, for some $E \in \mathfrak{C}$, if there exists a sequence of diffeomorphisms $\{\Phi_n\}_{n\in\N}$ converging to the identity in $C^2$ as $n\to +\infty$, such that $\Phi_{n}(E)=E_n$ for every $n\in\N$.
 In the following, we will extend this notion of convergence (in the obvious way) to families of sets $\{E_\rho\}_{\rho\in (0,1)}\subset\mathfrak{C}$ as the parameter $\rho\to 0^+$\,.
 
 Notice that if $E\in\mathfrak{C}$, then either $E$ or $E^c$ is compact. Therefore, in order to define the supercritical perimeters and the corresponding curvatures on the whole $\mathfrak{C}$, it is convenient to set $\tildeJ(E):=\tildeJ(E^c)$ for every set $E\in\Me$ with $E^c\in\Mf$.
  %%%%%%%%%%%%%%%%%%%%%%%%%%%%%%%%%%%%%%%%
 %%%%%%%%%%%%%%%%%%%%%%%%%%%%%%%%%%%%%%%%
 %%%%%%%%%%%%%%%%%%%%%%%%%%%%%%%%%%%%%%%%
 \subsection{Non-local $k_{r}^s$-curvatures}
Let $s\ge 1$, $r>0$ and $E\in\Reg$\,.
 For every $x\in\partial E$ we define the {\it $k^s_r$-curvature of $E$ at $x$} as
\begin{equation}\label{defcurv}
\Ku^s_r(x,E):=\int_{\R^d}(\chi_{E^c}(y)-\chi_{E}(y))k^s_r(|x-y|)\ud y.
\end{equation}
 %%%%%%%%%%%%%%%%%%%%%%%%%%%%%%%%%%%%%%%%
 %%%%%%%%%%%%%%%%%%%%%%%%%%%%%%%%%%%%%%%%
 %%%%%%%%%%%%%%%%%%%%%%%%%%%%%%%%%%%%%%%%
Although this fact may be immediate for the experts, we show that $\Ku^s_r$ is the first variation of the functional $\tildeJ$ in the sense specified by the following proposition. 
%%%%%%%%%%%%%%%%%%%%%%%%%%%%%%%%%%%%%%%%
%%%%%%%%%%%%%%%%%%%%%%%%%%%%%%%%%%%%%%%%
%%%%%%%%%%%%%%%%%%%%%%%%%%%%%%%%%%%%%%%%
\begin{proposition}[First variation]\label{curvfirstvarper}
	Let $s\geq 1$, $r>0$, and $E \in \mathfrak{C}$.   
	Let $\Phi: \mathbb{R} \times \mathbb{R}^d \rightarrow \mathbb{R}^d$ be a smooth function, and let $\{\Phi_t\}_{t\in\R}$ be defined by $\Phi_t(\cdot):= \Phi(t,\cdot)$ for every $t\in\R$\,.
	Assume that $\{\Phi_t\}_{t\in\R}$ is a family of  diffeomorphisms with 
	 $\Phi_0= \mathrm{Id}$  and that there exists an open bounded set $A \subset \mathbb{R}^d$ such that
	\begin{equation}\label{assu}
	\{x \in \mathbb{R}^d: \; x \neq \Phi_t(x) \} \subset A \quad \text{for all $t \in \mathbb{R}$}\,.
	\end{equation}
	Setting $E_t:= \Phi_t(E)$ and $\Psi(\cdot):= \frac{\partial }{\partial t}\Phi_t(\cdot) \big|_{t=0}$\,, we have
	\begin{equation}\label{clai}
	\frac{\ud}{\ud t} \tildeJ(E_t){\bigg|_{t=0}}= \int_{\partial E} \Ku^s_r(x,E) \Psi(x)\cdot \nu_{E}(x) \ud\mathcal{H}^{d-1}(x)\,.
	\end{equation}	
\end{proposition}
%%%%%%%%%%%%%%%%%%%%%%%%%%%%%%%%%%%%%%%%
%%%%%%%%%%%%%%%%%%%%%%%%%%%%%%%%%%%%%%%%
%%%%%%%%%%%%%%%%%%%%%%%%%%%%%%%%%%%%%%%%
\begin{proof}
	By Taylor expansion for every $x\in\R^d$ we have that $\Phi_t(x)=x +t \Psi(x)+\mathrm{o}(t)$. Therefore the Jacobian $\J\Phi_t$ of $\Phi_t$ is equal to
	\begin{equation*}
	\J\Phi_t(x):= \sqrt{\det(\nabla \Phi_t(x)\, \nabla\Phi_t(x)^* )}= 1+ t \Div (\Psi(x))+ \mathrm{o}(t)\,,
	\end{equation*}
	where, for every $A\in\R^{m\times k}$ ($m,k\in\N$), the symbol $A^*$ denotes the transpose of the matrix $A$\,.
By change of variable, it follows that
	\begin{equation}\label{fvsviluppodiJPHIt}
	\begin{split}
	\tildeJ(E_t) =& \int_{\Phi_t(E)} \int_{\Phi_t(E^c)} k_{r}^{s}(\vert x-y \vert)\ud y \ud x  \\
	 =& \int_{E} \int_{E^c} k_{r}^{s}(\vert \Phi_t(x)- \Phi_t(y) \vert ) \J\Phi_t(x) \J\Phi_t(y) \ud y \ud x \\
 =&\int_{E} \int_{E^c}k_r^s(\vert \Phi_t(x)-\Phi_t(y) \vert)\ud y\ud x\\
 & +t\int_{E} \int_{E^c}k_r^s(\vert \Phi_t(x)-\Phi_t(y) \vert)  \Big(\Div \Psi(x)+\Div \Psi(y)\Big)\ud y\ud x\\
 &+\mathrm{o}(t)\int_{E} \int_{E^c}k_r^s(\vert \Phi_t(x)-\Phi_t(y) \vert) \ud y\ud x\,.
	\end{split}
	\end{equation}
Let $(k_{r}^{s})': (0,+\infty) \rightarrow \mathbb{R}$ be the weak derivative of $k_{r}^s: (0,+\infty) \rightarrow \mathbb{R}$, that is equal a.e. to
\begin{equation*}
(k^s_{r})'(h):=
\left\{
\begin{array}{ll}
\displaystyle 0 & \text{ for $ 0 < h  < r $ }  \, , \\
\displaystyle -(d+s)\frac{1}{ h^{d+s+1}} & \text{ for }   h > r \, .
\end{array}
\right.
\end{equation*}
Notice that $k_r^s \in W^{1,1}(\mathbb{R})$.	
We set
\begin{equation*}%\label{nuode}
\begin{aligned}
K(t):=&
\int_{E}\int_{E^c}\Big( k_{r}^{s}(\vert \Phi_t(x)-\Phi_t(y) \vert)-k_{r}^{s}(\vert x-y \vert)\\
&-t(k_r^s)'(\vert x-y \vert) \frac{x-y}{\vert x-y \vert} \cdot (\Psi(x)-\Psi(y))\Big)\ud y\ud x
\end{aligned}
\end{equation*}
and we claim that
\begin{equation}\label{DerdiKrs1}
\lim_{t\to 0}\frac{K(t)}{t}=0\,.
\end{equation}
By the fundamental theorem of calculus, we have
\begin{equation*}
\begin{aligned}
&\int_{E}\int_{E^c} \Big(k_{r}^{s}(\vert \Phi_t(x)-\Phi_t(y) \vert)-k_{r}^{s}(\vert x-y \vert)\Big)\ud y \ud x\\
=&\int_{0}^{t}\ud\tau\bigg[\int_{E}\int_{E^c}(k_r^s)'(\vert \Phi_\tau(x)-\Phi_\tau(y) \vert) \frac{\Phi_\tau(x)-\Phi_\tau(y)}{\vert \Phi_\tau(x)-\Phi_\tau(y) \vert} \cdot \Big(\frac{\partial \Phi_{\tau}}{\partial \tau} (x)-\frac{\partial \Phi_{\tau}}{\partial \tau} (y)\Big)\ud y\ud x\bigg]\,,
\end{aligned}
\end{equation*}
so that
\begin{equation*}
\begin{split}
\bigg\vert\frac{K(t)}{t}\bigg\vert=
 & \bigg| \frac{1}{t}\int_{0}^{t} \biggl[\int_{E} \int_{E^c} (k_r^s)'(\vert \Phi_\tau(x)-\Phi_{\tau}(y) \vert) \frac{\Phi_\tau(x)-\Phi_\tau(y)}{\vert \Phi_\tau(x)-\Phi_\tau(y) \vert} \cdot \Big(\frac{\partial \Phi_{\tau}}{\partial \tau} (x)-\frac{\partial \Phi_{\tau}}{\partial \tau} (y)\Big) \\
 & - (k_r^s)'(\vert x-y \vert) \frac{x-y}{\vert x-y \vert} \cdot (\Psi(x)-\Psi(y)) \ud y \ud x \biggr]\ud \tau \bigg| \\
 \leq & \frac{1}{|t|}\int_{0}^{|t|} \bigg| \int_{E} \int_{E^c} (k_r^s)'(\vert \Phi_\tau(x)-\Phi_{\tau}(y) \vert) \frac{\Phi_\tau(x)-\Phi_\tau(y)}{\vert \Phi_\tau(x)-\Phi_\tau(y) \vert} \cdot \Big(\frac{\partial \Phi_{\tau}}{\partial \tau} (x)-\frac{\partial \Phi_{\tau}}{\partial \tau} (y)\Big) \\
 & - (k_r^s)'(\vert x-y \vert) \frac{x-y}{\vert x-y \vert} \cdot (\Psi(x)-\Psi(y)) \ud y \ud x \bigg|\ud \tau \\
 =:&  \frac{1}{|t|}\int_{0}^{|t|}\bigg| \int_{E} \int_{E^c} \big(f_\tau(x,y)-f_0(x,y)\big) \ud y\ud x\bigg|\ud\tau\,,
\end{split}
\end{equation*}
where in the last line 
 for every $\tau\in\R$ and for every $(x,y)\in\R^d\times\R^{d}$ we have set
\begin{equation*}
f_\tau(x,y):=(k_r^s)'(\vert \Phi_\tau(x)-\Phi_{\tau}(y) \vert) \frac{\Phi_\tau(x)-\Phi_\tau(y)}{\vert \Phi_\tau(x)-\Phi_\tau(y) \vert} \cdot \Big(\frac{\partial \Phi_{\tau}}{\partial \tau} (x)-\frac{\partial \Phi_{\tau}}{\partial \tau} (y)\Big)\,.
\end{equation*}
Notice that \eqref{DerdiKrs1} follows if we show that
\begin{equation}\label{ficl}
\int_{E} \int_{E^c} f_\tau(x,y) \ud y\ud x\to \int_{E} \int_{E^c}f_0(x,y) \ud y\ud x\qquad\textrm{as }\tau\to 0\,.
\end{equation}
By change of variable, we get
\begin{multline}\label{chv}
\int_{E} \int_{E^c} f_\tau(x, y)\ud y \ud x \\ =\int_{\mathbb{R}^{d}}\int_{\R^d} f_0(x,y) \J \Phi_{\tau}^{-1}(x)\J \Phi_{\tau}^{-1}(y)\chi_{E \times E^c}(\Phi_{\tau}^{-1}(x), \Phi_{\tau}^{-1}(y))\ud y \ud x\,.
\end{multline}
Setting $C(\J):=\sup_{\tau\in (0,1)}\ \| \J \Phi_{\tau}^{-1} \|_{L^\infty}$\,, by \eqref{assu}, we have that 
\begin{multline*}
|f_0(x,y)|  \J \Phi_{\tau}^{-1}(x)\J \Phi_{\tau}^{-1}(y)\chi_{E \times E^c}(\Phi_{\tau}^{-1}(x), \Phi_{\tau}^{-1}(y)) \\
  \leq [C(\J)]^2 |f_0(x,y)|[\chi_{E \cap A}( \Phi_{\tau}^{-1}(x))+\chi_{E\cap A^c}(x) ] [ \chi_{E^c \cap A}( \Phi_{\tau}^{-1}(y))+\chi_{E^c\cap A^c}(y) ]\,\\
    \leq [C(\J)]^2 |f_0(x,y)|[\chi_{A}(x)+\chi_{E}(x) ] [ \chi_{A}(y)+\chi_{E^c}(y) ]\, ,
\end{multline*}
where the right hand side term is clearly in  $\mathrm{L}^{1}(\mathbb{R}^{2d})$\,.
This fact together with \eqref{chv} and 
$$ 
\J \Phi_{\tau}^{-1}(x)\J \Phi_{\tau}^{-1}(y)\chi_{E \times E^c}(\Phi_{\tau}^{-1}(x), \Phi_{\tau}^{-1}(y)) \rightarrow \chi_{E\times E^c}(x,y) \quad \text{a.e. as $\tau \rightarrow 0$\,,}
$$
yields by Lebesgue Dominated Convergence Theorem,  \eqref{ficl} and, in turn,  \eqref{DerdiKrs1}\,.
By \eqref{fvsviluppodiJPHIt} and \eqref{DerdiKrs1}, and using the divergence theorem, we obtain that
\begin{equation*}
\begin{split}
\frac{\ud}{\ud t} \tildeJ(E_t) \bigg|_{t=0}= &  \int_{E} \int_{E^c} (k_r^s)'(\vert x-y \vert) \frac{x-y}{\vert x-y \vert} \cdot (\Psi(x)-\Psi(y))\ud y\ud x \\
 & +\int_{E} \int_{E^c}k_r^s(\vert x-y \vert) (  \Div \Psi(x)+ \Div \Psi(y))\ud y\ud x \\
=& \int_{E} \int_{E^c} (k_r^s)'(\vert x-y \vert) \frac{x-y}{\vert x-y \vert} \cdot (\Psi(x)-\Psi(y))\ud y \ud x \\
&+ \int_{E^c} \biggl[- \int_{E} (k_r^s)'(\vert x-y \vert )\,\Psi(x) \cdot   \frac{x-y}{\vert x-y \vert}\ud x \\
& + \int_{\partial E} k_{r}^{s}(\vert x-y \vert)\, \Psi(x)\cdot \nu_{E}(x) \ud  \mathcal{H}^{d-1}(x)  \biggr]\ud y \\
&+  \int_{E} \biggl[ \int_{E^c}  (k_r^s)'(\vert x-y \vert )\,\Psi(y) \cdot  \frac{x-y}{\vert x-y \vert}\ud y \\
& - \int_{\partial E}  k_{r}^{s}(\vert x-y \vert) \,\Psi(y)\cdot \nu_{E}(y)\ud  \mathcal{H}^{d-1}(y)  \biggr]\ud x \\
=& \int_{\partial E} \Psi(x)\cdot \nu_{E}(x) \int_{\R^d}(\chi_{E^c}(y)-\chi_{E}(y))k^s_r(|x-y|)\ud y \ud \mathcal{H}^{d-1}(x) \\
=& \int_{\partial E} \Ku^s_r(x,E) \Psi(x)\cdot\nu_{E}(x) \ud\mathcal{H}^{d-1}(x), 
\end{split}
\end{equation*}
whence \eqref{clai} follows.
\end{proof}
%%%%%%%%%%%%%%%%%%%%%%%%%%%%%%%%%%%%%%
%%%%%%%%%%%%%%%%%%%%%%%%%%%%%%%%%%%%%%
%%%%%%%%%%%%%%%%%%%%%%%%%%%%%%%%%%%%%%
In Proposition \ref{proprietàcurvature1} we prove some qualitative properties of the curvatures $\Ku^{s}_r$ defined in \eqref{defcurv}, which imply in particular that $\Ku^{s}_r$ are non-local curvatures in the sense of \cite{CMP,CDNP}.
%%%%%%%%%%%%%%%%%%%%%%%%%%%%%%%%%%%%%%
%%%%%%%%%%%%%%%%%%%%%%%%%%%%%%%%%%%%%%
%%%%%%%%%%%%%%%%%%%%%%%%%%%%%%%%%%%%%%
 \begin{proposition}\label{proprietàcurvature1}
 For every $s\ge 1$, $0<r<1$ the functionals $\Ku^s_r$ defined in \eqref{defcurv} satisfy the following properties:
 \begin{itemize}
 \item[(M)] Monotonicity: If 
$E, F\in\Reg$ with $E\subseteq F$, and if $x\in \partial F\cap \partial E$,
then $\Ku^s_r(x,F)\le\Ku^s_r(x,E)$;
%%%%%%%%%%%%%%%%%%%%%%%%%%%%%%%%%%%%%%
\item[(T)] Translational invariance: for any $E\in\Reg$, $x\in\partial E$,
$y\in\R^d$, $\Ku^s_r(x,E)=\Ku^s_r(x+y,E+y)$;
%%%%%%%%%%%%%%%%%%%%%%%%%%%%%%%%%%%%%%
\item[(S)] Symmetry: 
For all $E\in\Reg$ and for every $x\in\partial E$ it holds 
$$
\Ku^s_r(x,E)= -\Ku^s_r(x,\R^d\setminus \overset{\circ}{E})\,,
$$
where $\overset{\circ}{E}$ denotes the interior of $E$.
%%%%%%%%%%%%%%%%%%%%%%%%%%%%%%%%%%%%%%
\item[(B)] Lower bound on the  curvature of the balls: 
\begin{equation}\label{lwrbdkappa}
\Ku^s_r(x,\overline{B}(0,\rho)) \geq  0   \quad \text{ for all  } x\in\partial B(0,\rho)\,,\, \rho>0\,;
\end{equation}
%%%%%%%%%%%%%%%%%%%%%%%%%%%%%%%%%%%%%%
\item[(UC)] Uniform continuity: There exists a modulus of continuity $\omega_r$ such that the following holds.
For every $E\in\Reg$, $x\in\partial E$,
 and for every diffeomorphism $\Phi:\R^d\to \R^d$ of class {$C^2$},
with $\Phi= \mathrm{Id}$ in $\R^d\setminus B(0,1)$, we have  
$$
|\Ku^s_r(x,E) - \Ku^s_r(\Phi(x),\Phi(E))|\le \omega_r(\|\Phi - \mathrm{Id}\|_{C^2})\,.
$$
%%%%%%%%%%%%%%%%%%%%%%%%%%%%%%%%%%%%%%
 \end{itemize}
 \end{proposition}
 %%%%%%%%%%%%%%%%%%%%%%%%%%%%%%%%%%%%%%
 %%%%%%%%%%%%%%%%%%%%%%%%%%%%%%%%%%%%%%
 %%%%%%%%%%%%%%%%%%%%%%%%%%%%%%%%%%%%%%
\begin{proof}
	We prove separately the properties above.
%%%%%%%%%%%%%%%%%%%%%%%%%%%%%%%%%%%%%%
%%%%%%%%%%%%%%%%%%%%%%%%%%%%%%%%%%%%%%
%%%%%%%%%%%%%%%%%%%%%%%%%%%%%%%%%%%%%%
\vskip4pt
\noindent
	\textit{Property} \textmd{(M):} Let $E,F \in \mathfrak{C}$ such that $E \subseteq F$, then  $-\chi_{F} \leq -\chi_{E}$ and 
	$\chi_{F^c} \leq \chi_{E^c}$. Therefore for all $x \in \partial E \cap \partial F$, we have
	\begin{equation*}
	\begin{split}
	& \Ku^s_r(x,F)=\int_{\R^d}(\chi_{F^c}(y)-\chi_{F}(y))k^s_r(|x-y|)\ud y \\
	&\leq \int_{\R^d}(\chi_{E^c}(y)-\chi_{E}(y))k^s_r(|x-y|)\ud y=\Ku^s_r(x,E)\,.
	\end{split}
	\end{equation*}
%%%%%%%%%%%%%%%%%%%%%%%%%%%%%%%%%%%%%%
%%%%%%%%%%%%%%%%%%%%%%%%%%%%%%%%%%%%%%
%%%%%%%%%%%%%%%%%%%%%%%%%%%%%%%%%%%%%%
\vskip4pt
\noindent
	\textit{Property} \textmd{(T):} Let $E \in \mathfrak{C}$\,, $x \in \partial E$ and $y \in \mathbb{R}^d$. By the change of variable $ \zeta= \eta-y$, we obtain
	\begin{equation*}
	\begin{split}
	\Ku^s_r(x+y,E+y)=&\int_{\R^d}(\chi_{E^c+y}(\eta)-\chi_{E+y}(\eta))k^s_r(|x+y-\eta|)\ud \eta \\
	= &\int_{\R^d}(\chi_{E^c}(\zeta)-\chi_{E}(\zeta))k^s_r(|x-\zeta|)\ud \zeta=\Ku^s_r(x,E)\,.
	\end{split}
	\end{equation*}
%%%%%%%%%%%%%%%%%%%%%%%%%%%%%%%%%%%%%%
%%%%%%%%%%%%%%%%%%%%%%%%%%%%%%%%%%%%%%
%%%%%%%%%%%%%%%%%%%%%%%%%%%%%%%%%%%%%%
\vskip4pt
\noindent
	\textit{Property} \textmd{(S):} Let $E \in \mathfrak{C}$ and $x \in \partial E$, then we have
	\begin{equation*}
	\begin{split}
	 \Ku_{r}^{s}(x,E)= &\int_{\mathbb{R}^d} (\chi_{E^c}(y)-\chi_{E}(y))k_{r}^{s}(\vert y-x \vert)\ud y \\
	 =&-\int_{\mathbb{R}^d}(\chi_{E}(y)-\chi_{E^c}(y))k_{r}^{s}(\vert x-y \vert)\ud y= -\Ku_{r}^{s}(x,\R^d\setminus\overset{\circ}{E}).
	\end{split}
	\end{equation*}
%%%%%%%%%%%%%%%%%%%%%%%%%%%%%%%%%%%%%%
%%%%%%%%%%%%%%%%%%%%%%%%%%%%%%%%%%%%%%
%%%%%%%%%%%%%%%%%%%%%%%%%%%%%%%%%%%%%%
\vskip4pt
\noindent
	\textit{Property} \textmd{(B):} Let $\rho >0$ and $\bar x \in \partial B(0,\rho)$\,.
	Since $B(2\bar x,\rho)\subset B^{c}(0,\rho)=\R^d\setminus \overline{B}(0,\rho)$\,, we get
	\begin{equation}\label{curBpos2}
	\begin{split}
	\Ku_{r}^{s}(\bar x,\overline{B}(0,\rho)) =& \int_{\mathbb{R}^d} (\chi_{B^c(0,\rho)}(y)-\chi_{B(0,\rho)}(y))k_{r}^{s}(\vert \bar x-y \vert)\ud y \\
	\geq &\int_{\mathbb{R}^d} (\chi_{B(2\bar x,\rho)}(y)-\chi_{B(0,\rho)}(y))k_{r}^{s}(\vert \bar x-y \vert)\ud y 	= 
	0,
	\end{split}
	\end{equation}
	where in the last equality we have used the change of variable $z=2\bar x-y$ and the radial symmetry of $k^s_r$ to deduce that
	\begin{equation*}
	\int_{\mathbb{R}^d} \chi_{B(2\bar x,\rho)}(y) k_{r}^{s}(\vert \bar x-y \vert)\ud y 
	=\int_{\R^d}\chi_{B(0,\rho)}(z)k^s_r(\vert \bar x-z \vert)\ud z\,.
	\end{equation*}
	Hence, by formula \eqref{curBpos2}, \eqref{lwrbdkappa} follows.
%%%%%%%%%%%%%%%%%
\vskip4pt
\noindent
\textit{Property} \textmd{(UC):} 
Let $\Phi: \mathbb{R}^d \rightarrow \mathbb{R}^d$ be a diffeomorphism of class $C^2$, with $\Phi(y)=y$ for all $\vert y-x \vert \geq 1$\,. We set $\E:=\Phi(E)$\,.
Let moreover $ \theta_{k_{r}^{s}} : [0,+\infty) \rightarrow \mathbb{R}$ be the function defined by $\theta_{k_{r}^{s}}(\eta):= \int_{B(0,\eta)}k_{r}^{s}(\vert z \vert)\ud z$.
Fix $\ep>0$ and let $\eta_\ep>0$ be small enough such that 
\begin{equation}\label{defetaep}
 2 \theta_{k_{r}^{s}}(\eta_\ep), \; \theta_{k_{r}^{s}}(2 \eta_\ep) \leq \frac{\varepsilon}{3}.
 \end{equation}
Notice that
\begin{eqnarray}\label{pri}
&&\bigg| \int_{B(x,\eta_\ep)} (\chi_{E^c}(y)- \chi_{E}(y))k_{r}^{s}(\vert x-y \vert) \ud y \bigg|\le \theta_{k_{r}^{s}}(\eta_\ep)\,,\\ 	\label{pri2}
&& \bigg| \int_{B(\Phi(x),\eta_\ep)} (\chi_{\E^c}(y)- \chi_{\E}(y))k_{r}^{s}(\vert \Phi(x)-y \vert) \ud y \bigg|\le \theta_{k_{r}^{s}}(\eta_\ep)\,,\\
 \label{sei}
&& \bigg| \int_{B(\Phi(x),2\eta_\ep)}(\chi_{\E^c}(y)-\chi_{\E}(y))k_{r}^{s}(\vert \Phi(x)-y \vert)\ud y \bigg| 
 \le \theta_{k_{r}^{s}}(2\eta_\ep)\,.
\end{eqnarray}
By \eqref{pri}, \eqref{pri2}, and \eqref{defetaep}, using  triangular inequality, we have
\begin{equation}\label{propritCeq1}
\begin{split}
&\vert \Ku_{r}^s(x, E)- \Ku_{r}^{s}(\Phi(x), \Phi(E)) \vert \\
\leq & \frac{\varepsilon}{3}+\bigg| \int_{B^c(x,\eta_\ep)} (\chi_{E^c}(y)- \chi_{E}(y))k_{r}^{s}(\vert x-y \vert) \ud y  \\
&- \int_{B^c(\Phi(x),\eta_\ep)}(\chi_{\E^c}(y)-\chi_{\E}(y))k_{r}^{s}(\vert \Phi(x)-y \vert)\ud y \bigg|\\
\leq & \frac{\ep}{3}+\bigg| \int_{B^c(x,\eta_\ep)} (\chi_{E^c}(y)- \chi_{E}(y))k_{r}^{s}(\vert x-y \vert) \ud y  \\
&- \int_{\Phi(B^c(x,\eta_\ep))}(\chi_{\E^c}(y)-\chi_{\E}(y))k_{r}^{s}(\vert \Phi(x)-y \vert)\ud y \bigg| \\
+ & \bigg| \int_{\Phi(B^c(x,\eta_\ep))} (\chi_{\E^c}(y)- \chi_{\E}(y))k_{r}^{s}(\vert \Phi(x)-y \vert) \ud y  \\
&- \int_{B^c(\Phi(x),\eta_\ep)}(\chi_{\E^c}(y)-\chi_{\E}(y))k_{r}^{s}(\vert \Phi(x)-y \vert)\ud y \bigg|.
\end{split}
\end{equation}
By the change of variable $z=\Phi(y)$ and using that $\Phi(y)=y$ if $|y-x|\ge 1$\,, we have
\begin{eqnarray}
\nonumber
&& \bigg| \int_{B^c(x,\eta_\ep)} (\chi_{E^c}(y)- \chi_{E}(y))k_{r}^{s}(\vert x-y \vert) \ud y  \\ \nonumber
&&- \int_{\Phi(B^c(x,\eta_\ep))}(\chi_{\E^c}(y)-\chi_{\E}(y))k_{r}^{s}(\vert \Phi(x)-y \vert)\ud y \bigg| \\ \nonumber
&= & \bigg| \int_{B^c(x,\eta_\ep)} (\chi_{E^c}(y)- \chi_{E}(y))k_{r}^{s}(\vert x-y \vert) \ud y  \\ \nonumber
&&- \int_{B^c(x,\eta_\ep)}(\chi_{E^c}(z)-\chi_{E}(z))k_{r}^{s}(\vert \Phi(x)-\Phi(z) \vert)\J\Phi(z)\ud z \bigg| \\  \nonumber%\label{proprietàCeq3}
&\le & \int_{B^c(x,\eta_\ep)}\Big|k_{r}^{s}(\vert x-y \vert) -k_{r}^{s}(\vert \Phi(x)-\Phi(y) \vert)\J\Phi(y)\Big|\ud y\\ \label{prii}
&= &\int_{B^c(x,1)}\Big|k_{r}^{s}(\vert x-y \vert)-k_{r}^{s}(\vert \Phi(x)-y \vert)\Big|\ud y\\ \label{seci}
&&+\int_{B(x,1)\setminus B(x,\eta_\ep)}\Big|k_{r}^{s}(\vert x-y \vert)-k_{r}^{s}(\vert \Phi(x)-\Phi(y) \vert)\J\Phi(y)\Big|\ud y\,.
\end{eqnarray}
Now, assuming that $\|\Phi-\mathrm{Id}\|_{C^2}$ is small enough, by using Lagrange Theorem one can show that
\begin{equation}\label{priist}
\begin{aligned}
&\int_{B^c(x,1)}\Big|k_{r}^{s}(\vert x-y \vert)-k_{r}^{s}(\vert \Phi(x)-y \vert)\Big|\ud y\\
\le &\omega(\|\Phi-\mathrm{Id}\|_{C^2}) \int_{B^c(x,1)}\frac{1}{|x-y|^{d+s+1}}\ud y
\le\frac{\ep}{6}\,,
\end{aligned}
\end{equation}
for some modulus of continuity $\omega$\,.
Analogously, for $\|\Phi-\mathrm{Id}\|_{C^2}$ small enough, one can easily check that
\begin{equation}\label{secist}
\begin{aligned}
&\int_{B(x,1)\setminus B(x,\eta_\ep)}\Big|k_{r}^{s}(\vert x-y \vert)-k_{r}^{s}(\vert \Phi(x)-\Phi(y) \vert)\J\Phi(y)\Big|\ud y\\
\le&\int_{B(x,1)\setminus B(x,\eta_\ep)}\Big|k_{r}^{s}(\vert x-y \vert)-k_{r}^{s}(\vert \Phi(x)-\Phi(y) \vert)\Big|\ud y\\
&+\int_{B(x,1)\setminus B(x,\eta_\ep)}k_{r}^{s}(\vert \Phi(x)-\Phi(y) \vert)\big|1-\J\Phi(y)\big|\ud y
\le\frac{\ep}{6}\,.
\end{aligned}
\end{equation}
Therefore, by \eqref{prii}, \eqref{seci}, \eqref{priist}, \eqref{secist} we deduce that 
\begin{equation}\label{proprietCeq3}
\begin{aligned}
& \bigg| \int_{B^c(x,\eta_\ep)} (\chi_{E^c}(y)- \chi_{E}(y))k_{r}^{s}(\vert x-y \vert) \ud y  \\ 
&- \int_{\Phi(B^c(x,\eta_\ep))}(\chi_{\E^c}(y)-\chi_{\E}(y))k_{r}^{s}(\vert \Phi(x)-y \vert)\ud y \bigg| \le\frac{\ep}{3}\,.
\end{aligned}
\end{equation}
In the end, we observe that, for $\|\Phi-\mathrm{Id}\|_{C^2}$ small enough, it holds
$$ \Phi(B^c(x, \eta_\ep)) \triangle B^c(\Phi(x),\eta_\ep) \subset B(\Phi(x), 2 \eta_\ep)\,,$$
which, in view of \eqref{sei}, yields
\begin{equation}\label{propirtàCeq4}
\begin{split}
& \bigg| \int_{\Phi(B^c(x,\eta_\ep))} (\chi_{\E^c}(y)- \chi_{\E}(y))k_{r}^{s}(\vert \Phi(x)-y \vert) \ud y  \\
&- \int_{B^c(\Phi(x),\eta_\ep)}(\chi_{\E^c}(y)-\chi_{\E}(y))k_{r}^{s}(\vert \Phi(x)-y \vert)\ud y \bigg| \leq \theta_{k_{r}^{s}}(2 \eta_\ep) \leq \frac{\varepsilon}{3}.
\end{split}
\end{equation}
Plugging \eqref{proprietCeq3} and \eqref{propirtàCeq4} into \eqref{propritCeq1}\,, we conclude the proof of property (UC) and of the whole proposition.
\end{proof}
%%%%%%%%%%%%%%%%%%%%%%%%%%%%%%%%%%%
%%%%%%%%%%%%%%%%%%%%%%%%%%%%%%%%%%%
%%%%%%%%%%%%%%%%%%%%%%%%%%%%%%%%%%%
\subsection{The classical mean curvature}
For every $E \in \mathfrak{C}$, and for every $x \in \partial E$, we denote by $\Ku^{1}(x,E)$ the (scalar) mean curvature of $\partial E$ at $x$, i.e., the sum of the principal curvatures of $\partial E$ at $x$. It is well-known that $\Ku^{1}$ is nothing but the first variation of the perimeter.
Let $ E \in \mathfrak{C}$\,, $x\in\partial E$ and assume that $\nu_E(x)=e_d$\,;
then in a neighborhood of $x=(x',x_d) \in \partial E$ we have that $\partial E$ is the graph of a $C^2$- function $f: B'(x',r) \rightarrow \mathbb{R}$, for some $r>0$ with $\mathrm{D} f(x')=0$ so that $B(x,r)\cap E=\{(x',x_d)\in B(x,r)\,:\,x_d\le f(x')\}$. In this case the mean curvature of $\partial E$ at $x$ is given by 
\begin{equation}\label{cueu}
\begin{aligned}
\Ku^1(x,E)=& \Div\bigg(\frac{-\mathrm{D} f}{\sqrt{1+\vert \mathrm{D} f \vert^2}}\bigg)(x')
= -\sum_{i=1}^{d-1} \frac{\partial^2}{\partial_{x_i^2}} f(x') \\
= &-\frac{1}{\omega_{d-1}} \int_{\mathbb{S}^{d-2}} \theta^*\mathrm{D}^2 f(x')\theta \ud \mathcal{H}^{d-2}(\theta)\,,
\end{aligned}
\end{equation}
where $\theta^*$ is the row vector obtained by transposing the (column) vector $\theta$ and $\mathrm{D}^2 f(x')$ denotes the Hessian matrix of $f$ evaluated at $x'$.
%%%%%%%%%%%%%%%%%%%%%%%%%%%%%%%%%%%%%%%
%%%%%%%%%%%%%%%%%%%%%%%%%%%%%%%%%%%%%%%
%%%%%%%%%%%%%%%%%%%%%%%%%%%%%%%%%%%%%%%
\begin{proposition}\label{procla}
The standard mean curvature $\Ku^1$ satisfies the following properties:
\begin{itemize}
	\item[(M)] Monotonicity: If 
	$E, F\in\Reg$ with $E\subseteq F$, and if $x\in \partial F\cap \partial E$,
	then $\Ku^1(x,F)\le\Ku^1(x,E)$;
	\item[(T)] Translational invariance: For every $E\in\Reg$, $x\in\partial E$,
	$y\in\R^d$, it holds: $\Ku^1(x,E)=\Ku^1(x+y,E+y)$;
	\item[(B)] Lower bound on the  curvature of the balls: 
\begin{equation*}%\label{lwrbdkappa2}
\Ku^1(x,\overline{B}(0,\rho)) \geq  0   \quad \text{ for all  } x\in\partial B(0,\rho)\,,\, \rho>0\,;
\end{equation*}
	\item [(S)] Symmetry: 
	For every $E\in\Reg$ and for every $x\in\partial E$ it holds 
	$$
	\Ku^1(x,E)= -\Ku^1(x,\R^d\setminus \overset{\circ}{E})\,.
	$$
	\item[(UC')] Uniform continuity: Given $R>0$, there exists a modulus of continuity $\omega_R$ such that the following holds.
	For every $E\in\Reg$, $x\in\partial E$, such that $E$ has both an internal
	and external ball condition of radius $R$ at $x$,
	and for every  diffeomorphism $\Phi:\R^d\to \R^d$ of class {$C^2$},
	with $\Phi=\mathrm{Id}$ in $\R^d\setminus B(0,1)$, we have  
	\begin{equation}\label{noia}
	|\Ku^1(x,E) - \Ku^1(\Phi(x),\Phi(E))|\le \omega_R(\|\Phi - \mathrm{Id}\|_{C^2})\,.
	\end{equation}
\end{itemize}
\end{proposition}
%%%%%%%%%%%%%%%%%%%%%%%%%%%%%%%%%%%%%%%
%%%%%%%%%%%%%%%%%%%%%%%%%%%%%%%%%%%%%%%
%%%%%%%%%%%%%%%%%%%%%%%%%%%%%%%%%%%%%%%
\begin{proof}
 We prove only the property \textmd{(UC')}\,, since the check of the remaining properties is straightforward. 
 %%%%%%%%%%%%%%%%%%%%%%%%%%%%%%%%%%%%%%%%%%
 Let $R>0$ and let $E \in \mathfrak{C}$ be such that $E$ satisfies both an internal
	and external ball condition of radius $R$ at a point $x \in \partial E$\,.
In order to get the claim, we can always assume without loss of generality that $\|\Phi - \mathrm{Id}\|_{C^2}\le 1$\,. 	
	
By the Implicit Function Theorem we have  that $ E \cap B(x,r)= \{ z \in B(x,r): \; g(z)<0\}$, for some $r>0$ and $g \in C^2(B(x,r))$. Moreover, in suitable coordinates we have that  $x=0$, $ \D g(0)=e_d$ and, for all $i\neq j$, with $i, \,j =1, \cdots,d$,  $ \frac{\partial^2 g}{\partial z_i \partial z_j}(0)=0$. Then, it is well known that
\begin{equation}\label{curvmediafunzg}
\begin{aligned}
\mathcal{K}^1(0, E)=&\Div_{\tau}\Big(\frac{\D g}{|\D g|}\Big)(0)=\sum_{i=1}^{d-1}\frac{\partial^2 g}{\partial z_i^2}(0)\,,
\end{aligned}
\end{equation} 
where $\Div_\tau$ denotes the tangential divergence operator.
 Since the mean curvature is invariant by translations and rotations,  up to small perturbations of $\Phi$ in $C^2$  we may assume, without loss of generality,  that $\Phi(0)=0$ and that the normal to $\Phi(E)$ at $\Phi(0)=0$ is still $e_d$. Since  
 $$ \Phi( E) \cap B(0,\tilde{r})= \{ y \in  B(0,\tilde{r}): \; g(\Phi^{-1}(y))<0\}$$
 for some $\tilde{r}>0$, setting $h:=g\circ\Phi^{-1}$,
we have
\begin{equation}\label{curvmediafunzPhig}
\begin{aligned}
  \mathcal{K}^1(0,\Phi(E))=&\frac{1}{|\D h (0)|} \sum_{j=1}^{d-1} \frac{\partial^2 h}{\partial y_j^2}(0) -\frac{1}{|\D h(0)|^2}\sum_{j=1}^{d-1}\frac{\partial h}{\partial y_j}(0)\,\frac{\partial |\D h|}{\partial y_j}(0)\,.
  \end{aligned}
    \end{equation}
Therefore, using that
\begin{equation*}
\begin{aligned}
\D h(0)=&\D g(0)\,\D\Phi^{-1}(0)=e_d\,\D\Phi^{-1}(0)\,,\\
\frac{\partial^2 h}{\partial y_j\partial y_k}(0)=&\sum_{i=1}^{d-1}\frac{\partial^2 g}{\partial z_i^2}(0)\,\frac{\partial \Phi_i^{-1}}{\partial y_j}(0)\, \frac{\partial \Phi_i^{-1}}{\partial y_k}(0)+\frac{\partial ^2 \Phi_d^{-1}}{\partial y_j\partial y_k}(0)\,,
\end{aligned}
\end{equation*}       
we have
\begin{equation}\label{a1}
\begin{aligned}
&\frac{1}{|\D h(0)|}\bigg|\sum_{j=1}^{d-1}\frac{\partial^2 h}{\partial y^2_j}(0)-\sum_{i=1}^{d-1}\frac{\partial^2 g}{\partial z^2_i}(0)\bigg|\\
\le&\frac{1}{|\D h(0)|}\bigg|\sum_{j=1}^{d-1}\frac{\partial^2 \Phi^{-1}_d}{\partial y_j^2}(0)+\sum_{i=1}^{d-1}\frac{\partial^2 g}{\partial z^2_i}(0)\Big(\sum_{j=1}^{d-1}\Big(\frac{\partial \Phi^{-1}_i}{\partial y_j}(0)\Big)^2-1\Big)\bigg|\\
&+\frac{|\D h(0)-1|}{|\D h(0)|}\bigg|\sum_{i=1}^{d-1}\frac{\partial^2 g}{\partial z^2_i}(0)\bigg|\\
\le& C\Big[\|\D^2\Phi^{-1}\|_{C^0}+\|\D^2 g\|_{C^0}\|\mathrm{Id}-(\D\Phi^{-1})^2\|_{C^0}+\|\D^2 g\|_{C^0}\|\mathrm{Id}-\D\Phi^{-1}\|_{C^0}\Big]\\
\le& C\Big(1+\frac{1}{R}\Big)\|\mathrm{Id}-\Phi\|_{C^2}\,.
\end{aligned}
\end{equation}    
Similar computations (that are left to the reader) show that
\begin{equation}\label{a2}
\begin{aligned}
&\frac{1}{|\D h(0)|^2}\bigg|\sum_{j=1}^{d-1}\frac{\partial h}{\partial y_j}(0)\,\frac{\partial |\D h|}{\partial y_j}(0)\bigg|
\le C\Big(1+\frac{1}{R}\Big)\|\mathrm{Id}-\Phi\|_{C^2}\,.
\end{aligned}
\end{equation}
Therefore, \eqref{noia} follows from \eqref{curvmediafunzg}, \eqref{curvmediafunzPhig}, \eqref{a1} and \eqref{a2}\,.
\end{proof}
%%%%%%%%%%%%%%%%%%%%%%%%%%%%%%%%%%%%%%%
%%%%%%%%%%%%%%%%%%%%%%%%%%%%%%%%%%%%%%%
%%%%%%%%%%%%%%%%%%%%%%%%%%%%%%%%%%%%%%%
%%%%%%%%%%%%%%%%%%%%%%%%%%%%%%%%%%%%%%%
%%%%%%%%%%%%%%%%%%%%%%%%%%%%%%%%%%%%%%%
%%%%%%%%%%%%%%%%%%%%%%%%%%%%%%%%%%%%%%%
\subsection{Convergence of $k_{r}^s$-curvature flow to mean curvature flow}\label{ssc:convmcf}
We now prove that the viscosity solutions  to the  $k_{r}^s$-curvature flow converge to the classical mean curvature flow as $r\to 0^+$. To this end, we will adopt notation and we will use results in \cite{CDNP}.

We preliminarily notice that since the curvatures $\Ku^s_r$ defined in \eqref{defcurv} satisfy property (UC) in Proposition \ref{proprietàcurvature1}, they also satisfy property (UC') in Proposition \ref{procla} (with $\Ku^1$ replaced by $\Ku^s_r$). As a consequence $\Ku^1$ and $\Ku^s_r$ (for every $0<r<1$ and $s\ge 1$) satisfy the following continuity property:
\begin{itemize}
\item[(C)] Continuity: If $\{E_n\}_{n\in\N}\subset\mathfrak{C}$, $E\in\mathfrak{C}$ and $E_n\to E$ in $\mathfrak{C}$\,, then the corresponding curvatures of $E_n$ at $x$ converge to the curvature of $E$ at $x$ for every $x\in\partial E_n\cap \partial E$\,.
\end{itemize}
Such a property, together with properties (M) and (T) (see Propositions \ref{procla} and  \ref{proprietàcurvature1}), implies that the functionals $\Ku^1$ and $\Ku^s_r$ (for every $s\ge 1$ and for every $r\in (0,1)$) are non-local curvatures in the sense of \cite[Definition 2.1]{CDNP} (see also \cite{CMP}).

Moreover, since by Propositions \ref{procla} and \ref{proprietàcurvature1},  $\Ku^1$ and $\Ku^s_r$ satisfy also properties (B) and (UC') (referred to as (C') in \cite{CDNP}), they both satisfy the assumptions of \cite[Theorem 2.9]{CDNP} that guarantee existence and uniqueness of suitably defined viscosity solutions of the corresponding geometric flows. We refer to \cite[Definition 2.3]{CDNP} for the precise definition of viscosity solution in this setting, while for the reader's convenience we state the existence and uniqueness result specialized to the geometric evolutions considered in this paper.
%%%%%%%%%%%%%%%%%%%%%%%%%%%%%%%
%%%%%%%%%%%%%%%%%%%%%%%%%%%%%%%
%%%%%%%%%%%%%%%%%%%%%%%%%%%%%%%
\begin{proposition}
Let $s\ge 1$ and $r>0$\,. Let $u_0\in C(\R^d)$ be a uniformly continuous function with $u_0=C_0$ in $\R^d\setminus B(0,R_0)$ for some $C_0,R_0\in\R$ with $R_0>0$\,. Then, there exists a unique viscosity solution - in the sense of \cite[Definition 2.3]{CDNP} -  $u^s_r:\R^d\times [0,+\infty)\to\R$ to the Cauchy problem
\begin{equation}\label{cauchy}
\begin{cases}
\partial_t u(x,t)+|\mathrm{D}u(x,t)|\Ku^{s}_r(x,\{y\,:\,u(y,t)\ge u(x,t)\})=0\\
u(x,0)=u_0(x)\,.
\end{cases}
\end{equation}
Moreover, the same result holds true if $\Ku^s_r$ is replaced by (any multiple of) $\Ku^1$\,.
\end{proposition}
%%%%%%%%%%%%%%%%%%%%%%%%%%%%%%%
%%%%%%%%%%%%%%%%%%%%%%%%%%%%%%%
%%%%%%%%%%%%%%%%%%%%%%%%%%%%%%%
We will show that, as $r\to 0^+$, the scaled $k^s_r$-curvatures $\frac{1}{\sca^s(r)}\Ku^s_r$ converge to $\omega_{d-1}\Ku^1$ on regular sets. 
In view of \cite[Theorem 3.2]{CDNP}\,, such a result will be crucial in order to prove the convergence of the corresponding geometric flows.
We first prove the following result by adopting the same strategy used in \cite[Proposition 2]{Imb}.
%%%%%%%%%%%%%%%%%%%%%%%%%%%%%%%
%%%%%%%%%%%%%%%%%%%%%%%%%%%%%%%
%%%%%%%%%%%%%%%%%%%%%%%%%%%%%%%
\begin{lemma}\label{lemconvcurv}
		Let $M, N \in \mathbb{R}^{(d-1)\times (d-1)}$ and let $\{M_r\}_{r>0},\{N_r\}_{r>0}\subset \mathbb{R}^{(d-1)\times (d-1)}$ be such that $M_r \rightarrow M, \; N_r \rightarrow N$ as $r \rightarrow 0^+$\,. Then, for every $\delta>0$ it holds
		\begin{equation} \label{limserveatutconvecur}
		\begin{split}
		\lim_{r \rightarrow 0^+}& \bigg(\frac{1}{\sigma^s(r)} \Big(\int_{\Funo}k_{r}^{s}(\vert y \vert)\ud y 
		- \int_{\Fdue}k_{r}^{s}(\vert y \vert)\ud y\Big)\bigg) \\
		=&  \int_{\mathbb{S}^{d-2}} \theta^*(N-M) \theta \ud \mathcal{H}^{d-2}(\theta),
		\end{split}
		\end{equation}
		where 
\begin{equation*}%\label{AB}
\begin{aligned}
\Funo:=&\{y=(y',y_d) \in B(0,\delta)\,:\;  (y')^*M_r y'  \leq y_d \leq (y')^*N_r y' \}\\
\Fdue:=&\{y=(y',y_d) \in B(0,\delta)\,:\;  (y')^* N_r y' \leq y_d \leq (y')^*M_r y'\}\,.
\end{aligned}		
\end{equation*}
	\end{lemma}
%%%%%%%%%%%%%%%%%%%%%%%%%%%%%%%%
%%%%%%%%%%%%%%%%%%%%%%%%%%%%%%%%
%%%%%%%%%%%%%%%%%%%%%%%%%%%%%%%%
	\begin{proof}
For every $\alpha>0$ we set	
\begin{equation*}	
\begin{split}
\Gbar^1_\alpha:=&
 \{y=(y',y_d) \in \R^{d-1}\times\R\,:\, y'=\rho \theta, \; 0\le\rho\le\alpha, \; \theta \in \mathbb{S}^{d-2}\,,\\
		&\rho^2 \theta^*M_r\theta \leq y_d \leq  \rho^2\theta^* N_r \theta \} 
\end{split}	
\end{equation*}	
Therefore, for $r$ small enough,
\begin{multline*}
\Funo=\Gbar^1_\delta\cap B(0,\delta)\,, \qquad \Funo\cap B(0,r)=\Gbar^1_r\cap B(0,r)\,,\\
\Funo\setminus B(0,r)=\big((\Gbar^1_\delta\setminus\Gbar^1_r)\cap (B(0,\delta)\big)\cup \big(\Gbar^1_r\setminus B(0,r)\big)\,.
\end{multline*}
It follows that
		\begin{equation}\label{svilimiteconvcurvatura}
		\begin{split}
		&\int_{\Funo } k_{r}^{s}(\vert y \vert)\ud y\\
		=&\int_{\Gbar^1_\delta}k_{r}^{s}(\vert y \vert)\ud y-\int_{\Gbar^1_\delta\setminus B(0,\delta)}k_{r}^{s}(\vert y \vert)\ud y\\
		=&\int_{\Gbar^1_r}k_{r}^{s}(\vert y \vert)\ud y+\int_{\Gbar^1_\delta\setminus\Gbar^1_r}k_{r}^{s}(\vert y \vert)\ud y	-\int_{\Gbar^1_\delta\setminus B(0,\delta)}\frac{1}{|y|^{d+s}}\ud y\\
		=&\int_{\Gbar^1_r\cap B(0,r)}k_{r}^{s}(\vert y \vert)\ud y+\int_{\Gbar^1_r\setminus B(0,r)}k_{r}^{s}(\vert y \vert)\ud y+\int_{\Gbar^1_\delta\setminus\Gbar_r}\frac{1}{|y|^{d+s}}\ud y	\\
		&-\int_{\Gbar^1_\delta\setminus B(0,\delta)}\frac{1}{|y|^{d+s}}\ud y\\
=&\int_{\Gbar^1_r\cap B(0,r)}\frac{1}{r^{d+s}}\ud y+\int_{\Gbar^1_r\setminus B(0,r)}\frac{1}{|y|^{d+s}}\ud y+\int_{\Gbar^1_\delta\setminus\Gbar^1_r}\frac{1}{|y|^{d+s}}\ud y	\\
	&	-\int_{\Gbar^1_\delta\setminus B(0,\delta)}\frac{1}{|y|^{d+s}}\ud y\\
		=&\int_{\Gbar^1_r}\frac {1}{r^{d+s}} \ud y+\int_{\Gbar^1_\delta\setminus\Gbar^1_r}\frac{1}{|y|^{d+s}}\ud y\\
		&-\int_{\Gbar^1_\delta\setminus B(0,\delta)}\frac{1}{|y|^{d+s}}\ud y-\int_{\Gbar^1_r\setminus B(0,r)}\frac {1}{r^{d+s}} \ud y+\int_{\Gbar^1_r\setminus B(0,r)}\frac{1}{|y|^{d+s}}\ud y\,.\\
			\end{split}
		\end{equation}
We  set
$$
A^1_r:=\{\theta \in \mathbb{S}^{d-2}\,: \, \theta^*(M_r -N_r) \theta\leq 0\}
$$		
and we notice that			
	\begin{equation}\label{svilimiteconvcurvatura10}
		\begin{split}
		&\int_{\Gbar^1_r}\frac {1}{r^{d+s}} \ud y+\int_{\Gbar^1_\delta\setminus\Gbar^1_r}\frac{1}{|y|^{d+s}}\ud y\\
		=&\frac{1}{r^{d+s}} \int_{A^1_r} \ud \mathcal{H}^{d-2}(\theta) \int_{0}^{r} \ud\rho\, \rho^{d-2} \int_{\rho^2 \theta^*M_r \theta}^{\rho^2\theta^* N_r \theta}\ud y_d\\
		&+\int_{A^1_r} \ud \mathcal{H}^{d-2}(\theta) \int_{r}^{\delta}  \ud\rho\,\rho^{d-2} \int_{\rho^2 \theta^*M_r \theta}^{\rho^2\theta^* N_r \theta}\frac{1}{(\rho^2+y_d^2)^{\frac{d+s}{2}}}\ud y_d \\
		=&  \frac{1}{d+1}\frac{r^{d+1}}{r^{d+s}}\int_{A^1_r}\theta^*(N_r-M_r) \theta \ud \mathcal{H}^{d-2}(\theta) \\
		&+ \int_{A^1_r} \ud \mathcal{H}^{d-2}(\theta) \int_{r}^{\delta}  \ud\rho\,\rho^{d-2}  \int_{\theta^* M_r \theta}^{\theta^*N_r \theta}\frac{\rho^2}{(\rho^2+\rho^4 t^2)^{\frac{d+s}{2}}}\ud t\\
		=&   \frac{r^{1-s}}{d+1} \int_{A^1_r}\theta^* (N_r-M_r) \theta \ud \mathcal{H}^{d-2}(\theta)\\
		& + \int_{A^1_r} \ud \mathcal{H}^{d-2}(\theta) \int_{r}^{\delta} \ud\rho \frac{1}{\rho^{s}}  \int_{ \theta^*M_r \theta}^{\theta^*N_r \theta} \frac{1}{(1+\rho^2t^2)^{\frac{d+s}{2}}}\ud t\,,
		\end{split}
		\end{equation}
where in the last but one equality we have used the change of variable $y_d= \rho^2 t$.

Moreover, trivially we have
\begin{equation}\label{svilimiteconvcurvatura100}
\int_{\Gbar^1_\delta\setminus B(0,\delta)}\frac{1}{|y|^{d+s}}\ud y\le C(\delta)\,,
\end{equation} 
for some $C(\delta)>0$\,.
Furthermore, it is easy to see that, for $r$ small enough,
$$\Gbar^1_r\setminus B(0,r)\subset \big(B'(0,r)\setminus B'(0,r-c r^2)\big)\times [-c r^2,c r^2]
$$
for some constant $c>0$ independent of $r$\,;
as a consequence,
$$
|\Gbar^1_r\setminus B(0,r)|\le Cr^{d+2}\,,
$$
whence we deduce that
\begin{equation}\label{svilimiteconvcurvatura1000}
\begin{aligned}
&\int_{\Gbar^1_r\setminus B(0,r)}\frac {1}{r^{d+s}} \ud y\le Cr^{2-s}\,,\\
&\int_{\Gbar^1_r\setminus B(0,r)}\frac {1}{|y|^{d+s}} \ud y\le \int_{\Gbar^1_r\setminus B(0,r)}\frac {1}{r^{d+s}} \ud y\le C r^{2-s}\,.
\end{aligned}
\end{equation}
Therefore, by \eqref{svilimiteconvcurvatura}, \eqref{svilimiteconvcurvatura10}, \eqref{svilimiteconvcurvatura100} and \eqref{svilimiteconvcurvatura1000}, we obtain
\begin{equation}\label{svilimiteconvcurvatura10000}
\begin{aligned}
&\frac{1}{\sca^s(r)}\int_{\Funo } k_{r}^{s}(\vert y \vert)\ud y\\
=&  \frac{r^{1-s}}{(d+1)\sca^s(r)} \int_{A^1_r}\theta^* (N_r-M_r) \theta \ud \mathcal{H}^{d-2}(\theta)\\
&+\frac{1}{\sca^s(r)} \int_{A^1_r} \ud \mathcal{H}^{d-2}(\theta) \int_{r}^{\delta} \ud\rho \frac{1}{\rho^{s}} \int_{ \theta^*M_r \theta}^{\theta^*N_r \theta} \frac{1}{(1+\rho^2t^2)^{\frac{d+s}{2}}}\ud t\\
&+f^1(r)\,,
\end{aligned}
\end{equation}
where $f^1(r)\to 0$ as $r\to 0^+$\,.	

Now we set
$$
A^2_r:=\{\theta \in \mathbb{S}^{d-2}\,: \, \theta^*(N_r -M_r) \theta\leq 0\}\,;
$$
by arguing as in the proof of \eqref{svilimiteconvcurvatura10000} we obtain
\begin{equation}\label{svilimiteconvcurvatura20000}
\begin{aligned}
&\frac{1}{\sca^s(r)}\int_{\Fdue } k_{r}^{s}(\vert y \vert)\ud y\\
=&  \frac{r^{1-s}}{(d+1)\sca^s(r)} \int_{A^2_r}\theta^* (M_r-N_r) \theta \ud \mathcal{H}^{d-2}(\theta)\\
&+\frac{1}{\sca^s(r)} \int_{A^2_r} \ud \mathcal{H}^{d-2}(\theta) \int_{r}^{\delta} \ud\rho \frac{1}{\rho^{s}} \int_{ \theta^*N_r \theta}^{\theta^*M_r \theta} \frac{1}{(1+\rho^2t^2)^{\frac{d+s}{2}}}\ud t\\
&+f^2(r)\,,
\end{aligned}
\end{equation}
where $f^2(r)\to 0$ as $r\to 0^+$\,.	
		Therefore by  formulas \eqref{svilimiteconvcurvatura10000} and \eqref{svilimiteconvcurvatura20000}, using that $A^1_r\cup A^2_r=\mathbb{S}^{d-2}$\,, we get
		\begin{equation}\label{limcurvutile123123}
		\begin{split}
		&\bigg(\frac{1}{\sigma^s(r)} \Big(\int_{\Funo}k_{r}^{s}(\vert y \vert)\ud y 
		- \int_{\Fdue}k_{r}^{s}(\vert y \vert)\ud y\Big)\bigg) \\
		=&  \bigg[ \frac{r^{1-s}}{(d+1)\sigma^{s}(r)}\bigg] \int_{\mathbb{S}^{d-2}} \theta^*(N_r-M_r) \theta \ud \mathcal{H}^{d-2}(\theta)\\
		& +\frac{1}{\sigma^s(r)} \int_{\mathbb{S}^{d-2}} \ud \mathcal{H}^{d-2}(\theta) \int_{r}^{\delta}  \ud\rho \frac{1}{\rho^{s}} \int_{\theta^* M_r \theta}^{\theta^*N_r \theta} \frac{1}{(1+\rho^2t^2)^{\frac{d+s}{2}}}\ud t\\
		&+f^1(r)-f^2(r)\,.
		\end{split}
		\end{equation}
Since
$$
 \frac{r^{1-s}}{(d+1)\sigma^s(r)} =\left\{\begin{array}{ll}
 \frac{1}{(d+1)|\log r|}&\textrm{ if }s=1\\
 \frac{s-1}{d+s}&\textrm{ if }s>1\,,
 \end{array}\right.
$$
and recalling that $M_r$ and $N_r$ converge to $M$ and $N$, respectively, we get
		\begin{equation}\label{convcurvatura2222}
		\begin{split}
		&\lim_{r \rightarrow 0^+}  \frac{r^{1-s}}{(d+1)\sigma^s(r)}  \int_{\mathbb{S}^{d-2}}\theta^*(N_r-M_r) \theta \ud \mathcal{H}^{d-2}(\theta)\\ &=
		\begin{cases}
		\displaystyle 0 &\text{if $s=1$}, \\
		 \displaystyle \frac{s-1}{d+s}\int_{\mathbb{S}^{d-2}} \theta^*(N-M) \theta \ud \mathcal{H}^{d-1}(\theta)   &\text{if $ s> 1$\,.}
		\end{cases}
		\end{split}
		\end{equation}
Moreover, for every $s\ge 1$\,, using de l'H\^opital rule (i.e., differentiating both terms in the product below with respect to $r$) and the very definition of $\sca^s(r)$ in \eqref{scalingper}, we have
\begin{equation*}
\begin{aligned}
&\lim_{r \rightarrow 0^+}\frac{1}{\sigma^s(r)}  \int_{r}^{\delta} \frac{1}{\rho^{s}}\frac{1}{(1+\rho^2t^2)^{\frac{d+s}{2}}}\ud \rho=\frac{d+1}{d+s}\,,
\end{aligned}
\end{equation*}
which, by the Dominate Convergence Theorem, yields
\begin{equation}\label{convcurvatura1111}
\begin{split}
		&\lim_{r \rightarrow 0^+}\frac{1}{\sigma^s(r)} \int_{\mathbb{S}^{d-2}} \ud \mathcal{H}^{d-2}(\theta) \int_{r}^{\delta}\ud\rho \frac{1}{\rho^{s}}  \int_{\theta^*M_r \theta}^{\theta^*N_r \theta} \frac{1}{(1+\rho^2t^2)^{\frac{d+s}{2}}}\ud t \\
		=&\lim_{r \rightarrow 0^+} \int_{\mathbb{S}^{d-2}} \ud \mathcal{H}^{d-2}(\theta)  \int_{\theta^*M_r \theta}^{\theta^*N_r \theta} \ud t \frac{1}{\sigma^s(r)} \int_{r}^{\delta} \frac{1}{\rho^{s}}\frac{1}{(1+\rho^2t^2)^{\frac{d+s}{2}}}\ud \rho\\
		=& \frac{d+1}{d+s}\int_{\mathbb{S}^{d-2}} \theta^*(N-M) \theta \ud \mathcal{H}^{d-2}(\theta)\,.
		\end{split}
		\end{equation}
		By formulas \eqref{limcurvutile123123}, \eqref{convcurvatura2222} and \eqref{convcurvatura1111} we obtain \eqref{limserveatutconvecur}.
	\end{proof}
%%%%%%%%%%%%%%%%%%%%%%%%%%%%%%%
%%%%%%%%%%%%%%%%%%%%%%%%%%%%%%%
%%%%%%%%%%%%%%%%%%%%%%%%%%%%%%%
		\begin{theorem}\label{convcurvthmprinc}
		Let $s\geq 1$. Let $\{E_r\}_{r>0} \subset \mathfrak{C}$ be such that $ E_{r} \rightarrow E $ in $\mathfrak{C}$ as $r \rightarrow 0^+$, for some $E \in \mathfrak{C}$. Then, for every $x \in \partial E \cap \partial E_r$ for every $r>0$\,, it holds
		\begin{equation*}
		\lim_{r \rightarrow 0^+} \frac{\Ku_{r}^s(x,E_r)}{\sigma^s(r)}=\omega_{d-1}\Ku^1(x,E).
		\end{equation*}
	\end{theorem}
%%%%%%%%%%%%%%%%%%%%%%%%%%%%%%%
%%%%%%%%%%%%%%%%%%%%%%%%%%%%%%%
%%%%%%%%%%%%%%%%%%%%%%%%%%%%%%%
	\begin{proof}
	Let $x \in \partial E \cap \partial E_r$ for all $r>0$. By Proposition \ref{proprietàcurvature1} we have that the curvatures $\Ku^s_r$ satisfy properties (S) and (T); moreover, it is easy to check that $\Ku^s_r$ are invariant by rotations. Therefore, we can assume without loss of generality that $E$ and $\{E_r\}_{r>0}$ are compact, and that $x=0$, $\nu_{E}(0)=\nu_{E_r}(0)=e_d$ for all $r>0$.
	Since $E_{r} \rightarrow E$ in $\mathfrak{C}$ as $r \rightarrow 0^+$ we have that there exist $\delta>0$ and functions $\phi,\phi_r:B'(0,\delta)\to\R$ such that $\phi_r \rightarrow \phi$ in $C^2$ as $r \rightarrow 0^+$, $\phi(0)=\phi_r(0)=0$\,, $\D \phi(0)= \D \phi_r (0)=0$ and
\begin{equation*}
\begin{split}
 \partial E \cap B(0,\delta) &= \{ (y',\phi(y')): \; y' \in B'(0,\delta) \}\cap B(0,\delta)\,,\\
 \partial E_r \cap B(0,\delta)&= \{ (y',\phi_r(y')): \; y' \in B'(0,\delta) \}\cap B(0,\delta) \,,\\
 E \cap B(0,\delta)&= \{ (y',y_d): \; y' \in B'(0,\delta)\,,\; y_d \leq \phi(y') \}\cap B(0,\delta)\,,\\
 E_r \cap B(0,\delta)&= \{ (y',y_d): \; y' \in B'(0,\delta)\,, \; y_d \leq \phi_r(y') \}\cap B(0,\delta)\,.
\end{split}
\end{equation*}	
Let $\eta>0$ be fixed\,; for $\delta$ small enough we have
		\begin{equation}\label{taylorconcurvtura}
		\Big\vert \phi_{r}(y')-\frac{1}{2} (y')^*\D^2 \phi_{r}(0)y'\Big \vert \leq \eta \vert y' \vert^2\quad\textrm{ for every }0<r<1\,,\; y' \in B'(0,\delta).
		\end{equation}
		We define the following sets
		\begin{align*}
	 A(r):=& \{y=(y',y_d)\in B(0,\delta)\, :\, -\phi_{r}(y') \leq y_d \leq \phi_r(y') \}\,, \\
	 B(r):= &\{y =(y',y_d)\in B(0,\delta)\, :\, \phi_{r}(y') \leq y_d \leq -\phi_{r}(y')    \} \,, \\
	 C(r):=&\big(E_r^c \setminus B(r)\big)\cap  B(0,\delta)\\
		=&\{y =(y',y_d)\in B(0,\delta)\, :\,y_d\ge\max\{\phi_r(y'),-\phi_r(y')\}\}\,, \\ 
	 D(r):=&\big(E_r \setminus A(r) \big)\cap  B(0,\delta)\\
		=&\{y =(y',y_d)\in B(0,\delta)\, :\,y_d\le\min\{\phi_r(y'),-\phi_r(y')\}\}\,,
		\end{align*}
where the equalities above are understood in the sense of measurable sets, i.e., up to negligible sets. 
We notice that 
\begin{multline*} 
E_r\cap B(0,\delta)=A(r)\cup D(r)\,,\qquad E^c_r\cap B(0,\delta)=B(r)\cup C(r)\,,\\
\int_{C(r)}k_r^s(\vert y \vert)\ud y= \int_{D(r)}k_r^s(\vert y \vert)\ud y\,,
\end{multline*}
whence we deduce that
		\begin{equation}\label{svilengconvcurvfin}
		\begin{split}
		\Ku_{r}^s(0,E_r)=& \int_{B(0,\delta)} (\chi_{E_r^c}(y)-\chi_{E_r}(y))k_r^s(\vert y \vert)\ud y \\
		& + \int_{B^c(0,\delta)}(\chi_{E_r^c}(y)-\chi_{E_r}(y))k_r^s(\vert y \vert)\ud y \\
		=&  \int_{\mathbb{R}^d} (\chi_{B(r)} (y)-\chi_{A(r)}(y))k_{r}^s(\vert y \vert)\ud y\\
		&+\int_{\mathbb{R}^d}(\chi_{C(r)}(y)-\chi_{D(r)}(y))k_{r}^s(\vert y \vert)\ud y\\
		& + \int_{B^c(0,\delta)}(\chi_{E_r^c}(y)-\chi_{E_r}(y))k_r^s(\vert y \vert)\ud y \\
		=&  \int_{\mathbb{R}^d}(\chi_{B(r)}(y)-\chi_{A(r)}(y)) k_{r}^s(\vert y \vert )\ud y\\
	& + \int_{B^c(0,\delta)}(\chi_{E_r^c}(y)-\chi_{E_r}(y))k_r^s(\vert y \vert)\ud y \,.
		\end{split}
		\end{equation}
Trivially,
		\begin{equation*}%\label{primoer}
		\Big|\int_{B^c(0,\delta)}(\chi_{E_r^c}(y)-\chi_{E_r}(y))k_r^s(\vert y \vert)\ud y\Big| \le d\omega_d\frac{\delta^{-s}}{s}\,.
		\end{equation*}
		In order to study the limit
		\begin{equation*}
		\lim_{r \rightarrow 0^+} \frac{1}{\sigma^{s}(r)}\int_{\mathbb{R}^d}(\chi_{B(r)}(y)-\chi_{A(r)}(y)) k_r^s(\vert y \vert)\ud y\,,
		\end{equation*}
		we define the following sets
		\begin{eqnarray}  \nonumber%\label{A^-(r)}
		&&A^-(r):=\Big\{ y=(y',y_d) \in B(0,\delta)\,: \\ \nonumber
		&&\quad-\frac{1}{2}(y')^*\D^2 \phi_r(0)y' +\eta \vert y' \vert^2 \leq y_d \leq  \frac{1}{2}(y')^*\D^2 \phi_r(0)y' -\eta \vert y' \vert^2 \Big\}, \\  \nonumber %\label{A^+(r)}
		&&A^+(r) :=   \Big\{ y=(y',y_d)\in B(0,\delta)\,: \\ \nonumber
		&& \quad-\frac{1}{2}(y')^*\D^2 \phi_r(0)y' -\eta \vert y' \vert^2 \leq y_d \leq  \frac{1}{2}(y')^*\D^2 \phi_r(0)y'+\eta \vert y' \vert^2    \Big\}, \\ \nonumber %\label{B^-(r)}
		&&B^-(r):=  \Big\{ y=(y',y_d) \in B(0,\delta)\,: \\ \nonumber
		&&\quad \frac{1}{2}(y')^*\D^2 \phi_r(0)y' +\eta \vert y' \vert^2 \leq y_d \leq  -\frac{1}{2}(y')^*\D^2 \phi_r(0)y' -\eta \vert y' \vert^2    \Big\}, \\ \nonumber %\label{B^+(r)}
		&&B^+(r):= \Big\{ y =(y',y_d)\in B(0,\delta)\,: \\ \nonumber
		&& \quad\frac{1}{2}(y')^*\D^2 \phi_r(0)y'-\eta \vert y' \vert^2 \leq y_d \leq  -\frac{1}{2}(y')^*\D^2 \phi_r(0)y'  +\eta \vert y' \vert^2    \Big\}\,. 
		\end{eqnarray}
		By \eqref{taylorconcurvtura} we have that 
		$$ A^-(r) \subset A(r) \subset A^+(r), \quad B^-(r)\subset B(r) \subset B^+(r)\,,$$
		and hence
		\begin{equation}\label{teoconconvcurv}
		\begin{split}
		&\int_{\mathbb{R}^d}(\chi_{B^-(r)}(y)-\chi_{A^+(r)}(y)) k_r^s(\vert y \vert)\ud y \\ \leq& \int_{\mathbb{R}^d}(\chi_{B(r)}(y)-\chi_{A(r)}(y)) k_r^s(\vert y \vert)\ud y\\ \leq& \int_{\mathbb{R}^d}(\chi_{B^+(r)}(y)-\chi_{A^-(r)}(y)) k_r^s(\vert y \vert)\ud y. 
		\end{split}
		\end{equation}
		Then, by applying Lemma \ref{lemconvcurv} and using \eqref{teoconconvcurv}, we obtain
		\begin{equation}\label{limcurvform}
		\begin{split}
		&-\int_{\mathbb{S}^{d-2}} \theta^*(\D^2 \phi(0)+2\eta \mathrm{Id})\theta \ud \mathcal{H}^{d-2}(\theta) \\ 
		\leq& \liminf_{r \rightarrow 0^+} \frac{1}{\sigma^{s}(r)}\int_{\mathbb{R}^d}(\chi_{B(r)}(y)-\chi_{A(r)}(y)) k_r^s(\vert y \vert)\ud y\\ 
			\leq& \limsup_{r \rightarrow 0^+} \frac{1}{\sigma^{s}(r)}\int_{\mathbb{R}^d}(\chi_{B(r)}(y)-\chi_{A(r)}(y)) k_r^s(\vert y \vert)\ud y\\ 
		\leq& -\int_{\mathbb{S}^{d-2}} \theta^*(\D^2 \phi(0)-2\eta \mathrm{Id})\theta \ud \mathcal{H}^{d-2}(\theta)\,.
		\end{split}
		\end{equation}
		Therefore, by  \eqref{svilengconvcurvfin}  and \eqref{limcurvform}, we get
		\begin{equation*}
		\begin{split}
			&-\int_{\mathbb{S}^{d-2}} \theta^*(\D^2 \phi(0)+2\eta \mathrm{Id})\theta \ud \mathcal{H}^{d-2}(\theta) \\ 
			\leq& \liminf_{r \rightarrow 0^+} \frac{1}{\sigma^{s}(r)} \Ku_{r}^s(0,E_r) 
			\leq \limsup_{r \rightarrow 0^+} \frac{1}{\sigma^{s}(r)} \Ku_{r}^s(0,E_r)\\ 
			\leq& -\int_{\mathbb{S}^{d-2}} \theta^*(\D^2 \phi(0)-2\eta \mathrm{Id})\theta \ud \mathcal{H}^{d-2}(\theta)\,,
		\end{split}
		\end{equation*}
		which, sending $\eta$ to $0$ and using \eqref{cueu} implies the claim.
		\end{proof}
%%%%%%%%%%%%%%%%%%%%%%%%%%%%%%
%%%%%%%%%%%%%%%%%%%%%%%%%%%%%%%
%%%%%%%%%%%%%%%%%%%%%%%%%%%%%%%
We are now in a position to state the main result of this section.
%%%%%%%%%%%%%%%%%%%%%%%%%%%%%%%
%%%%%%%%%%%%%%%%%%%%%%%%%%%%%%%
%%%%%%%%%%%%%%%%%%%%%%%%%%%%%%%
\begin{theorem}\label{maincur}
Let $s\ge 1$ be fixed. Let $u_0\in C(\R^d)$ be a uniformly continuous function with $u_0=C_0$ in $\R^d\setminus B(0,R_0)$ for some $C_0,R_0\in\R$ with $R_0>0$\,. For every $r>0$\,, let $u^s_r:\R^d\times [0,+\infty)\to\R$ be the viscosity solution to the Cauchy problem \eqref{cauchy}\,. Then, setting $v^s_r(x,t):=u^s_r(x,\frac{t}{\sca^s(r)})$ for all $x\in\R^d$\,, $t\ge 0$\,, we have that,  for every $T>0$\,,
$v^s_r$ uniformly converge to $u$ as $r\to 0^+$ in $\R^d\times [0,T]$\,, where $u:\R^d\times[0,+\infty)\to\R$ is the viscosity solution to the classical mean curvature flow
\begin{equation}\label{cauchycl}
\begin{cases}
\partial_t u(x,t)+|\mathrm{D}u(x,t)|\omega_{d-1}\Ku^{1}(x,\{y\,:\,u(y,t)\ge u(x,t)\})=0\\
u(x,0)=u_0(x)\,.
\end{cases}
\end{equation}
\end{theorem}
%%%%%%%%%%%%%%%%%%%%%%%%%%%%%%%
%%%%%%%%%%%%%%%%%%%%%%%%%%%%%%%
%%%%%%%%%%%%%%%%%%%%%%%%%%%%%%%
\begin{proof}
We preliminarily notice that, by an easy scaling argument, the functions $v^s_r$ are viscosity solution to
\begin{equation*}%\label{cauchyre}
\begin{cases}
\partial_t v(x,t)+|\D v(x,t)|\frac{1}{\sca^{s}(r)}\Ku^{s}_r(x,\{y\,:\,v(y,t)\ge v(x,t)\})=0\\
v(x,0)=u_0(x)\,.
\end{cases}
\end{equation*}
By Theorem \ref{convcurvthmprinc} we have that , as $r\to 0^+$ the scaled $k^s_r$-curvatures $\frac{1}{\sca^s(r)}\Ku^s_r$ converge to $\omega_{d-1}\Ku^1$ on regular sets. Moreover, by Propositions \ref{proprietàcurvature1} and \ref{procla}, $\Ku^s_r$ (for every $r\in (0,1)$) and $\Ku^1$ satisfy properties (M), (T), (S), (UC'). Furthermore, for every $\rho>0$ and for every $x\in \partial B(0,\rho)$, by Proposition \ref{proprietàcurvature1}, we have that $\Ku^s_r(x,\overline{B}(0,\rho)\ge 0$ whereas, by Theorem \ref{convcurvthmprinc} we get that $\sup_{r\in(0,1)}\Ku^s_r(x,\overline{B}(0,\rho))<+\infty$\,. 
This trivially implies the following property:
\begin{itemize}
\item[(UB)] There exists $K>0$ such that $\inf_{r\in (0,1)}\Ku_r^s(x,\overline{B}(0,\rho))\ge -K\rho$ for all $\rho>1$\,, $x\in \partial B(0,\rho)$ and $\sup_{r\in (0,1)}\Ku_r^s(x,\overline{B}(0,\rho))<+\infty$ for all $\rho>0$\,, $x\in \partial B(0,\rho)$\,.
\end{itemize}
Properties (M), (T), (S), (UC') (referred to as (C') in \cite{CDNP}) and (UB), together with the convergence of the curvatures on regular sets, are exactly the assumptions of \cite[Theorem 3.2]{CDNP}, which, in our case, establishes the convergence of $v^s_r$ to $u$ locally uniformly in $\R^d\times [0,T]$ for every $T>0$\,.
\end{proof}
%%%%%%%%%%%%%%%%%%%%%%%%%%%%%%%
%%%%%%%%%%%%%%%%%%%%%%%%%%%%%%%
%%%%%%%%%%%%%%%%%%%%%%%%%%%%%%%
\section{Stability as $r\to 0^+$ and $s\to 1^+$ simultaneously}\label{doli}
%%%%%%%%%%%%%%%%%%%%%%%%%%%
%%%%%%%%%%%%%%%%%%%%%%%%%%
%%%%%%%%%%%%%%%%%%%%%%%%%%
%%%%%%%%%%%%%%%%%%%%%%%%%%
%%%%%%%%%%%%%%%%
%%%%%%%%%%%%%%%
%%%%%%%%%%%%%%%%%%
%%%%%%%%%%%%%%%%%
In this section we study $\Gamma$-convergence and compactness properties for the $s$-fractional perimeters $\tildeJ$ defined in \eqref{Jtildefrac} when $r\to 0^+$ and $s\to \bar s$ (with $\bar s\ge 1$) simultaneously. Moreover, we study the convergence of the corresponding geometric flows in such a case. 
In fact, we will consider only the critical case $\bar s=1$\,, the case $\bar s>1$ being  easier.

Let
 $\{r_n\}_{n \in \mathbb{N}} \subset (0,1)$
 and  $ \{ s_n\}_{n \in \mathbb{N}} \subset (1,+\infty)$ be such that $r_n\to 0^+$ and $s_n\to 1^{+}$  as $n\to +\infty$\,.
 Recalling the definitions of $\sigma^{s}(r)$ in \eqref{scalingper} and $\bm{\alpha}^s$ in \eqref{error}, we set
 \begin{equation}\label{defbetaultimo}
\bm{\beta}(r_n,s_n):=\sigma^{s_n}(r_n)+\bm{\alpha}^{s_n}= \frac{d+s_n}{d+1} \frac{r_n^{1-s_n}-1}{s_n-1}+ \frac{1}{d+1}
\end{equation}
and we notice that
\begin{equation}\label{limbeta}
\begin{aligned}
\lim_{n\to +\infty}\bm{\beta}(r_n,s_n)\ge&\lim_{n\to +\infty}  \frac{r_n^{1-s_n}-1}{s_n-1}
= \lim_{n\to +\infty}\int_{r_n}^{1}\rho^{-s_n}\ud \rho\\
\ge&  \lim_{n\to +\infty}\int_{r_n}^{1}\rho^{-1}\ud \rho
=\lim_{n\to +\infty}|\log r_n|=+\infty\,.
\end{aligned}
\end{equation}
%%%%%%%%%%%%%%%%%%%%%%%%%%%%%%%%%%%
%%%%%%%%%%%%%%%%%%%%%%%%%%%%%%%%%%%
%%%%%%%%%%%%%%%%%%%%%%%%%%%%%%%%%%%
\subsection{$\Gamma$-convergence and compactness }
In Theorem \ref{mainthm111} below we study the $\Gamma$-convergence of the functionals $\frac{1}{\bm{\beta}(r_n,s_n)}\tilde{J}_{r_n}^{s_n}$ as $n\to +\infty$\,. 
 %%%%%%%%%%%%%%%%%%%%%%%%%%%%%%%%%%%
%%%%%%%%%%%%%%%%%%%%%%%%%%%%%%%%%%%
%%%%%%%%%%%%%%%%%%%%%%%%%%%%%%%%%%%
\begin{theorem}\label{mainthm111}
Let  $\{r_n\}_{n \in \mathbb{N}} \subset (0,1)$ and  $ \{ s_n\}_{n \in \mathbb{N}} \subset (1,+\infty)$ be such that $r_n\to 0^+$ and $s_n\to 1^+$ as $n\to +\infty$\,.
 The following $\Gamma$-convergence result holds true.
		\begin{itemize}
		%%%%%%%%%%%%%%%%
			\item[(i)] (Compactness) Let $U \subset \mathbb{R}^d $ be an open bounded set and let $\{ E_n\}_{n \in \mathbb{N}} \subset \mathrm{M}(\mathbb{R}^d)$ be such that $ E_n \subset U$ for every $n \in \mathbb{N}$ and
			\begin{equation*}%\label{compcondsbar>1}
			\tilde{J}_{r_n}^{s_n}(E_n) \leq M \bm{\beta}(r_n,s_n) \quad \text{for every $n \in \mathbb{N}$,}
			\end{equation*}
			for some constant $M $ independent of $n $. Then up to a subsequence, $\chi_{E_n} \rightarrow \chi_{E}$ strongly in $\mathrm{L}^1(\mathbb{R}^d)$ for some set $E \in \mathrm{M}_f(\mathbb{R}^d)$ with $\Per(E) < + \infty$.
			%%%%%%%%%%%%%%%%
			\item[(ii)](Lower bound) Let $E\in\Mf$. For every $\{E_n\}_{n\in\N}\subset\Mf$ with $\chi_{E_n}\to\chi_E$ strongly in $L^1(\R^d)$ it holds
			\begin{equation*}%\label{trueliminfbars>1}
			\omega_{d-1}\Per(E)\le\liminf_{n\to +\infty}\frac{\tilde{J}_{r_n}^{s_n}(E_n)}{\bm{\beta}(r_n,s_n)}\,.
			\end{equation*}
			%%%%%%%%%%%%%%%%%%
			\item[(iii)] (Upper bound) For every $E\in\Mf$ there exists $\{E_n\}_{n\in\N}\subset\Mf$ such that $\chi_{E_n}\to\chi_E$ strongly in $L^1(\R^d)$ and
			\begin{equation*}%\label{limsupbars>1}
			\omega_{d-1}\Per(E)=\limsup_{n\to +\infty} \frac{\tilde{J}_{r_n}^{s_n}(E_n)}{\bm{\beta}(r_n,s_n)}\,.
			\end{equation*}
	\end{itemize}
\end{theorem}
%%%%%%%%%%%%%%%%%%%%%%%%%%%%
%%%%%%%%%%%%%%%%%%%%%%%%%%%%
%%%%%%%%%%%%%%%%%%%%%%%%%%%% 

\subsubsection{Proof of compactness}	
We start by proving the compactness property Theorem \ref{mainthm111}(i). To this purpose, we first prove the following lemma which corresponds to Lemma \ref{quasiiso} when both $r$ and $s$ vary.
%%%%%%%%%%%%%%%%%%%%%%%%%%%%
%%%%%%%%%%%%%%%%%%%%%%%%%%%%
%%%%%%%%%%%%%%%%%%%%%%%%%%%% 
\begin{lemma}\label{quasiiosMod}
Let  $\{r_n\}_{n \in \mathbb{N}} \subset (0,1)$ and  $ \{ s_n\}_{n \in \mathbb{N}} \subset (1,+\infty)$ be such that $r_n\to 0^+$ and $s_n\to 1^+$ as $n\to +\infty$\,.
	 Let $\Omega \in \mathrm{M}_f(\mathbb{R}^d)$ be a bounded set with Lipschitz continuous boundary and $\vert \Omega \vert=1$. 
	For every $\eta\in(0,1)$ there exist a constant $C(\Omega,d,S,\eta)>0$ and $\bar n\in\N$ such that for every measurable set $A\subset\Omega$ with $\eta\le |A|\le 1-\eta$ it holds
	\begin{equation*}%\label{f:quasiisomod}
	\quad \int_{A}\int_{\Omega\setminus A}k^{s_n}_{r_n}(|x-y|)\ud y\ud x\ge C(\Omega,d,S,\eta)\bm{\beta}(r_n,s_n)\qquad\textrm{for every }n\ge\bar{n}\,,
	\end{equation*}
	where $S:= \sup_{n \in \mathbb{N}} s_n$.
\end{lemma}
%%%%%%%%%%%%%%%%%%%%%%%%%%%%
%%%%%%%%%%%%%%%%%%%%%%%%%%%%
%%%%%%%%%%%%%%%%%%%%%%%%%%%% 
\begin{proof}
The proof is fully analogous to the one of Lemma \ref{quasiiso}; we sketch only the main differences.
For every $n \in \mathbb{N}$, let $I_n \in \mathbb{N}$ be such that $2^{-I_n-1} \leq r_n \leq 2^{-I_n}$\,. Let $\phi$ and $\phi_\delta$ (for every $\delta>0$) be as in Lemma \ref{lm:l1ker}. 
 By arguing verbatim as in the proof of \eqref{stimasututti} (see \eqref{psst}, \eqref{ssst}, and \eqref{tsst}), for every $n\in\N$ and for every $z\in\R^d$ we have
\begin{equation}\label{quasisios_nform1}
\begin{aligned}
k_{r_n}^{s_n}(|z|)\ge& \frac{2^{d+s_n}-1}{2^{d+s_n}}  \frac{1}{\sup\phi}
\sum_{i=0}^{I_n} (2^{s_n})^{i} \phi_{2^{-i}}(z) \\
\ge&\frac{2^{d+1}}{2^{d+1}-1}  \frac{1}{\sup\phi}\sum_{i=0}^{I_n} (2^{s_n})^{i} \phi_{2^{-i}}(z)\,.
\end{aligned}
\end{equation}
Moreover, since 
$$
\frac{|\log r_n|}{\log 2}-1\le I_n\le \frac{|\log r_n|}{\log 2}\,,
$$
setting $ m(S):= \inf_{s \in (1,S]} \frac{s-1}{2^{s-1}-1}$,we get
\begin{equation}\label{sommai}
\begin{split}
\sum_{i=0}^{I_n} (2^{s_n-1})^i =& 
 \frac{\big(2^{s_n-1}\big)^{I_n+1} -1}{2^{s_n-1}-1} 
\geq \frac{r_{n}^{1-s_n} -1}{2^{s_n-1}-1} \\
 =& \frac{r_{n}^{1-s_n} -1}{s_n -1} \frac{s_n-1}{2^{s_n-1}-1} \geq \frac{m(S)}{2} \frac{d+1}{d+s_n} \Big(2 \frac{d+s_n}{d+1} \frac{r_{n}^{1-s_n} -1}{s_n -1} \Big) \\
\ge & \frac{m(S)}{2} \frac{d+1}{d+S}\Big( \frac{d+s_n}{d+1} \frac{r_{n}^{1-s_n} -1}{s_n -1}+\frac{1}{d+1} \frac{r_{n}^{1-s_n} -1}{s_n -1} \Big)\\
\geq& \frac{m(S)}{2} \frac{d+1}{d+S} \bm{\beta}(r_n,s_n)\,,
\end{split} 
\end{equation}
where in the last inequality we have used that, in view of \eqref{defbetaultimo},
$\frac{r_{n}^{1-s_n} -1}{s_n -1}\ge 1$\,.
Therefore, by arguing as in \eqref{endproof}, using \eqref{quasisios_nform1} and \eqref{sommai}, we get the claim. 
\end{proof}
%%%%%%%%%%%%%%%%%%%%%%%%%
%%%%%%%%%%%%%%%%%%%%%%%%%
%%%%%%%%%%%%%%%%%%%%%%%%%
With Lemma \ref{quasiiosMod} in hand, we can prove Theorem \ref{mainthm111}(i), whose proof closely follows the one of Theorem \ref{mainthm}(i). We sketch only the main differences.
%%%%%%%%%%%%%%%%%%%%%%%%%
%%%%%%%%%%%%%%%%%%%%%%%%%
%%%%%%%%%%%%%%%%%%%%%%%%%
\begin{proof}[Proof of Theorem \ref{mainthm111}(i)]
We preliminarily notice that, up to a subsequence, the following limit exists
 \begin{equation}\label{limpossibililambda}
   \lim_{n \rightarrow + \infty} (s_n-1)\vert \log r_n \vert=: \lambda \,;
  \end{equation}	
clearly, $\lambda\in [0,+\infty]$\,.	
We first prove the claim under the assumption $\lambda\neq 0$\,.
Let $\alpha \in (0,1)$ and for every $n\in\N$ we set $l_n:= r_n^{\alpha}(s_n-1)$\,; therefore, since $\lambda\in(0,+\infty]$\,,
\begin{equation*}%\label{minomion22222}
\lim_{n \rightarrow+ \infty} \frac{r_n}{l_n}=\lim_{n\to +\infty} \frac{r_n^{1-\alpha}}{s_n-1}=0\,.
\end{equation*}
By adopting the same notation as in Subsection \ref{maincomp} we set
\begin{equation*}
\tilde{E}_n:= \bigcup_{h=1}^{H(n)} Q_{h}^{n}\,,
\end{equation*} 
where $\{Q_{h}^n\}_{h\in\N}$ is a family of pairwise disjoint cubes of sidelength $l_n$ which covers the whole $\R^d$ and satisfies \eqref{famigliacubi}.
  
By arguing verbatim as in the proof of \eqref{difsimmEeEtild} one can show that there exists $n'\in\N$ such that
\begin{equation}\label{teocomps_nform1}
\vert \tilde{E}_n \triangle E_n \vert \leq 4 l_n^{s_n} \bm{\beta}(r_n,s_n) M \quad \text{for every $n\geq n'.$}
\end{equation}
We observe that
\begin{equation}\label{l_nbeta(r_n,s_n)}
\begin{aligned}
\lim_{n \rightarrow + \infty} l_n^{s_n} \bm{\beta}(r_n,s_n)=&\lim_{n\to +\infty}r_n^{\alpha s_n}(s_n-1)^{s_n}\Big(\frac{d+s_n}{d+1}\frac{r_n^{1-s_n}-1}{s_n-1}+\frac 1 {d+1}\Big)\\
=&\lim_{n \rightarrow + \infty} r_n^{1-s_n+\alpha s_n} \frac{d+s_n}{d+1} (s_n-1)^{s_n-1}=0\,.
\end{aligned}
\end{equation}
Now, setting $S:=\sup_{n\in\N}s_n$, we claim that there exists a constant $C(\alpha,d,S)$ such that for $n$ large enough
\begin{equation}\label{teocomps_nform2}
\Per(\tilde{E}_n)\leq C(\alpha, d,S)\frac{\tilde{J}_{r_n}^{s_n}(E_n)}{\bm{\beta}(r_n,s_n)}\,.
\end{equation}
In order to prove \eqref{teocomps_nform2}, we argue as in Step 2 in Subsection \ref{maincomp}. 
We define the family $\mathcal R$ of rectangles $R=\tilde Q^n_h\cup\hat Q^n_h$ such that $\tilde Q^n_h$ and $\hat Q^n_h$ are adjacent,
 $\tilde Q^n_h\subset\tilde E_n$ and $\hat Q^n_h\subset \tilde E_n^c$\,.

Notice that
\begin{equation}\label{primastimaPscubet2}
\begin{aligned}
\Per(\tilde{E}_n)\leq & 2d l_n^{d-1}\sharp\mathcal R\,,\\
 \frac{\tilde{J}_{r_n}^{s_n}(E_n)}{\bm{\beta}(r_n,s_n)}\ge &\frac{1}{2d\,\bm{\beta}(r_n,s_n)}\sum_{R\in\mathcal R}\int_{R\cap E_n}\int_{R\setminus E_n}k^{s_n}_{r_n}(|x-y|)\ud y\ud x\,.
 \end{aligned}
\end{equation} 
Moreover, by Lemma \ref{quasiiosMod}, for every rectangle $\bar R$ given by the union of two adjacent unitary cubes in $\R^d$, there exists $\bar n\in\N$ such that 
\begin{equation}\label{applem2}
\begin{aligned}
 C(d,\lambda):=& \inf \biggl\{ \frac{1}{\bm{\beta}(\rho_n,s_n)}\int_{F} \int_{\bar R\setminus F} k^{s_n}_{\rho_n}(\vert x-y \vert)\ud y \ud x: \\
&\qquad n\ge \bar n,\; F\in \Mf\,,\; F\subset \bar R\,, \; \frac 1 2\le \vert F \vert\le  \frac{3}{2} \biggr\}>0\,.
\end{aligned}
\end{equation}
Furthermore, by the very definition of $\bm{\beta}(r_n,s_n)$ in \eqref{defbetaultimo}, we have
\begin{equation*}%\label{quobeta}
\begin{aligned}
\frac{\bm{\beta}(r_n,s_n)}{l_n^{1-s_n}\bm{\beta}\big(\frac{r_n}{l_n},s_n\big)}
=&1+\frac{\displaystyle (d+1)(l_n^{1-s_n}-1)}{\displaystyle (d+s_n)\big(r_n^{1-s_n}-l_n^{1-s_n}\big)+(s_n-1)l_n^{1-s_n}}\,,
\end{aligned}
\end{equation*}
whence, using that $l_n=r_n^{\alpha}(s_n-1)$ and \eqref{limpossibililambda}, we deduce
\begin{equation}\label{asinbetas_nl_n}
\begin{aligned}
\lim_{n \rightarrow+ \infty} \frac{\bm{\beta}(\frac{r_n}{l_n},s_n)}{\frac{\bm{\beta}(r_n,s_n)}{l_n^{1-s_n}}}=& \left\{\begin{array}{ll}
1+\frac{e^{\lambda\alpha}-1}{e^\lambda-e^{\lambda\alpha}}&\textrm{ if }\lambda\neq +\infty\\
1&\textrm{ if }\lambda=+\infty\,.
\end{array}\right.
\end{aligned}
\end{equation}
For every set $O\in\Mf$ we set $O^{l_n}:=\frac{O}{l_n}$.
By \eqref{primastimaPscubet2}, \eqref{scaling}, \eqref{asinbetas_nl_n} and by applying \eqref{applem2} with $\bar R=R^{l_n}$ for every $R\in\mathcal R$, for $n$ large enough we obtain
\begin{equation*}%\label{dimstimaperimetrocubi}
\begin{split}
 &\frac{\tilde J^{s_n}_{r_n}(E_n)}{\bm{\beta}(r_n,s_n)}\\
 \ge 
& \frac{C(d)}{\bm{\beta}(r_n,s_n) }l_n^{2d}\sum_{R\in\mathcal R} \int_{R^{l_n}\cap E^{l_n}}\int_{R^{l_n}\setminus E^{l_n}}k^{s_n}_{r_n}(|l_n(x-y)|)\ud y\ud x
\\
= & C(d)\frac{l_n^{1-s_n}}{\bm{\beta}(r_n,s_n) }l_n^{d-1}\sum_{R\in\mathcal R} \int_{R^{l_n}\cap E^{l_n}}\int_{R^{l_n}\setminus E^{l_n}}k^{s_n}_\frac{r_n}{l_n}(|x-y|)\ud y\ud x\\
 \ge &C(\alpha,d,\lambda)l_n^{d-1}\sum_{R\in\mathcal R} \frac{1}{\bm{\beta}(\frac{r_n}{l_n},s_n)} \int_{R^{l_n}\cap E^{l_n}}\int_{R^{l_n}\setminus E^{l_n}}k^{s_n}_\frac{r_n}{l_n}(|x-y|)\ud y\ud x\\
\ge & C(\alpha,d,\lambda)l_n^{d-1}\sharp\mathcal R \,C(d,\lambda)
 \ge C(\alpha,d,\lambda) \Per(\tilde{E}_n)\,,
\end{split}
\end{equation*}
i.e., \eqref{teocomps_nform2}.
Therefore, using \eqref{teocomps_nform1}, \eqref{l_nbeta(r_n,s_n)} and \eqref{teocomps_nform2}, by arguing as in Step 3 of the proof of Theorem \ref{mainthm}(i), we get the claim whenever \eqref{limpossibililambda} is satisfied.

Finally, if 
$$
\lim_{n\to +\infty}(s_n-1)\vert \log r_n \vert=0\,,
$$
taking $l_n=r_n^\alpha$ (with $\alpha\in (0,1)$), one can show that
\begin{eqnarray*}
\lim_{n\to +\infty}l_n^{s_n}\bm{\beta}(r_n,s_n)
&=&0\,,\\
\lim_{n\to +\infty}\frac{\bm{\beta}(r_n,s_n)}{l_n^{1-s_n}\bm{\beta}\big(\frac{r_n}{l_n},s_n\big)}
&=&\frac{1}{1-\alpha}\,,
\end{eqnarray*}
which used in the above proof, in place of \eqref{l_nbeta(r_n,s_n)} and \eqref{asinbetas_nl_n}, respectively, imply the claim also in this case.
\end{proof}
%%%%%%%%%%%%%%%%%%%%%%%%%%%%%%%%%%%%
%%%%%%%%%%%%%%%%%%%%%%%%%%%%%%%%%%%%
%%%%%%%%%%%%%%%%%%%%%%%%%%%%%%%%%%%%
The following result follows by arguing as in the proof of Proposition \ref{corcompthm}, using now the estimates in the proof of Theorem \ref{mainthm111}(i) instead of the ones in the proof of Theorem \ref{mainthm} (i).
%%%%%%%%%%%%%%%%%%%%%%%%%%%%%%%%%%%%
%%%%%%%%%%%%%%%%%%%%%%%%%%%%%%%%%%%%
%%%%%%%%%%%%%%%%%%%%%%%%%%%%%%%%%%%%
\begin{proposition}\label{corcompthmbars>1} 
 Let $\{ E_n\}_{n \in \mathbb{N}} \subset \mathrm{M}_{f}(\mathbb{R}^d)$ be such that $ \chi_{E_n} \rightarrow \chi_{E}$ in $\mathrm{L}^1(\mathbb{R}^d)$ as $n \rightarrow + \infty$, for some $ E \in \mathrm{M}_f(\mathbb{R}^d)$.  If
	\begin{equation*} %\label{hpcorcompps_n}
	\limsup_{n\to +\infty} \frac{\tilde{J}_{r_n}^{s_n}(E_n)}{\bm{\beta}(r_n,s_n)}  < +\infty,
		\end{equation*}
then $E$ is a set of finite perimeter.
\end{proposition}
%%%%%%%%%%%%%%%%%%%%%%
%%%%%%%%%%%%%%%%%%%%%%
%%%%%%%%%%%%%%%%%%%%%%
%%%%%%%%%%%%%%%%%%%%%%
%%%%%%%%%%%%%%%%%%%%%%
%%%%%%%%%%%%%%%%%%%%%%
\subsubsection{Proof of the lower bound}
In order to prove the $\Gamma$-liminf inequality Theorem \ref{mainthm111}(ii), we first need the following result, which is the analogous to Lemma \ref{lemma18} under our assumptions on $\{s_n\}_{n\in\N}$ and $\{r_n\}_{n\in\N}$\,. 
%%%%%%%%%%%%%%%%%%%%%%%%%%%%%%%
%%%%%%%%%%%%%%%%%%%%%%%%%%%%%%%
%%%%%%%%%%%%%%%%%%%%%%%%%%%%%%%
\begin{lemma}\label{lemma18bis}
Let $\{r_n\}_{n\in\N}\subset(0,1)$ and $\{s_n\}_{n\in\N}\subset(1,+\infty)$ be such that $r_n\to 0^+$ and $s_n\to 1^+$ as $n\to +\infty$\,. For every $\ep>0$ there exist $\delta_{0}>0$ and $\bar n\in\N$ such that for every $\nu \in \mathbb{S}^{d-1}$, for every $E \in \mathrm{M}_f(\mathbb{R}^d)$ with 
$$ \vert (E \triangle H^-_{\nu}(0)) \cap Q^{\nu} \vert < \delta_{0}$$
and for  every $n\ge\bar n$ it holds
$$ 
\int_{Q^{\nu}\cap E} \int_{Q^{\nu}\cap E^c} k_{r_n}^{s_n}(\vert x-y \vert) \ud y \ud x\geq \omega_{d-1} (1-\ep)\bm{\beta}(r_n,s_n)\,.
$$
\end{lemma}
%%%%%%%%%%%%%%%%%%%%%%%%%%%%%%%
%%%%%%%%%%%%%%%%%%%%%%%%%%%%%%%
%%%%%%%%%%%%%%%%%%%%%%%%%%%%%%%
\begin{proof}
By arguing as in the proof of Lemma \ref{lemma18}  (see \eqref{incr2}) one can prove that
\begin{equation}\label{incr2bis}
\begin{aligned}
&\int_{Q^\nu\cap E}\int_{Q^\nu\cap E^c} k^{s_n}_{r_n}(|x-y|)\ud y\ud x\\
\ge&
\omega_{d-1}\bm{\beta}(r_n,s_n)\Big(1-\eta(\delta_0)-\frac{1-\eta(\delta_0)}{\bm{\beta}(r_n,s_n)}\frac{\delta_0^{{1-s_n}}-1}{s_n-1}-2C(d)\sqrt{\delta_0}\Big)\,,
\end{aligned}
\end{equation}
where $\eta(t)\to 0$ as $t\to 0$\,.
Notice that we can choose $0<\delta_0<1$ such that 
\begin{equation}\label{perr}
\eta(\delta_0)+2C(d)\sqrt{\delta_0}\le \frac{\ep}{2}\,.
\end{equation}
Furthermore, since
$$
\lim_{n\to +\infty}\frac{\delta_0^{{1-s_n}}-1}{s_n-1}=|\log\delta_0|\quad\textrm{and}\quad \lim_{n\to +\infty}\bm{\beta}(r_n,s_n)=+\infty\,,
$$
we have that there exists $\bar n\in\N$ such that 
\begin{equation}\label{serr}
\frac{1-\eta(\delta_0)}{\bm{\beta}(r_n,s_n)}\frac{\delta_0^{1-s_n}-1}{s_n-1}\le \frac{\ep}{2}\quad\textrm{for }n\ge\bar n\,.
\end{equation}
Therefore, by \eqref{incr2bis}, \eqref{perr} and \eqref{serr}, we get the claim.
\end{proof}
%%%%%%%%%%%%%%%%%%%%%%%%%%%%%%%
%%%%%%%%%%%%%%%%%%%%%%%%%%%%%%%
%%%%%%%%%%%%%%%%%%%%%%%%%%%%%%%
\begin{proof}[Proof of Theorem \ref{mainthm111}(ii)]
The proof closely follows the one of Theorem \ref{mainthm}(ii).
%%%%%%%%%%%%%%%%%%%%%%%%%%%%%%%
%%%%%%%%%%%%%%%%%%%%%%%%%%%%%%%
%%%%%%%%%%%%%%%%%%%%%%%%%%%%%%%
We can assume without loss of generality that
\begin{equation}\label{ben}
\frac{1}{2\bm{\beta}(r_n,s_n)}\int_{\R^d}\int_{\R^d}k^{s_n}_{r_n}(|x-y|)|\chi_{E_n}(x)-\chi_{E_n}(y)|\ud y\ud x\le C\,,
\end{equation}
for some constant $C>0$ independent of $n$. Then, by Corollary \ref{corcompthmbars>1} we have that $E$ has finite perimeter.
For every $n\in\N$ let $\mu_n$ be the measure on the product space $\R^d\times \R^d$ defined by
$$
\mu_n(A\times B):=\frac{1}{2\bm{\beta}(r_n,s_n)}\int_{A}\int_{B}k^{s_n}_{r_n}(|x-y|)|\chi_{E_n}(x)-\chi_{E_n}(y)|\ud y\ud x
$$
for every $A,B\in\mathrm{M}(\R^d)$. By arguing as in the proof of Theorem \ref{mainthm}(ii) we have that, up to a subsequence, $\mu_n\to\mu$ as $n\to +\infty$ for some measure $\mu$ concentrated on the set $D:=\{(x,x)\,:\,x\in\R^d\}$\,. Therefore, by using the same Radon-Nykodym argument exploited in the proof of Theorem \ref{mainthm}(ii), it is enough to show that for $\mathcal{H}^{d-1}$ - a.e. $x_0\in\partial^*E$ 
\begin{equation}\label{tesibis}
\liminf_{l\to 0^+}\frac{\mu(Q_l^\nu(x_0)\times Q_l^\nu(x_0)}{l^{d-1}}\ge\liminf_{l\to 0^+}\liminf_{n\to +\infty}\frac{\mu_n(Q_l^{\nu}(x_0)\times Q_l^{\nu}(x_0))}{l^{d-1}}\ge \omega_{d-1}\,,
\end{equation}
where we have set $\nu:=\nu_E(x_0)$ and
 $Q^l_{\nu}(x_0):=x_0+lQ^\nu$\,.
In order to prove \eqref{tesibis} we adopt the same strategy used to prove \eqref{tesi}.
More precisely, setting $F_{n,l}=x_0+lE_n$\,, in place of \eqref{chavar} we have
\begin{equation}\label{chavarbis}
\begin{aligned}
&\frac{1}{l^{d-1}}\mu_n(Q^\nu_l(x_0)\times Q^\nu_l(x_0))\\
=&\frac{l^{1-s_n}}{2\bm{\beta}(r_n,s_n)} \int_{Q^\nu}\int_{Q^\nu}k^{s_n}_{\frac{r_n}{l}}(|\xi-\eta|) |\chi_{F_{n,l}}(\xi)-\chi_{F_{n,l}}(\eta)|\ud\xi\ud\eta\,,
\end{aligned}
\end{equation}
and, in place of \eqref{apple},
\begin{multline}\label{applebis}
\frac 1 2\int_{Q^\nu}\int_{Q^\nu}k^{s_n}_{\frac{r_n}{l}}(|\xi-\eta|) |\chi_{F_{n,l}}(\xi)-\chi_{F_{n,l}}(\eta)|\ud\xi\ud\eta\\
\ge \omega_{d-1}(1-\ep)\bm{\beta}\Big(\frac{r_n}{l},s_n\Big)\,,
\end{multline}
which is a consequence of Lemma \ref{lemma18bis}.
Therefore, since
$$
\lim_{n \rightarrow + \infty} \frac{l^{1-s_n}}{\bm{\beta}(r_n,s_n)}\bm{\beta}\Big(\frac{r_n}{l},s_n\Big)=1\,,
$$
by \eqref{chavarbis} and \eqref{applebis}, we get
$$
\liminf_{l\to 0^+}\frac{\mu(Q_l^\nu(x_0)\times Q_l^\nu(x_0))}{l^{d-1}}\ge (1-\ep)\omega_{d-1}\,,
$$
whence \eqref{tesibis} follows by the arbitrariness of $\ep$\,.
\end{proof}
%%%%%%%%%%%%%%%%%%%%%%
%%%%%%%%%%%%%%%%%%%%%%
%%%%%%%%%%%%%%%%%%%%%%
%%%%%%%%%%%%%%%%%%%%%%
%%%%%%%%%%%%%%%%%%%%%%
%%%%%%%%%%%%%%%%%%%%%%
\subsubsection{Proof of the upper bound}
In order to prove the $\Gamma$-limsup inequality, we need the following result which is the analogous of Proposition \ref{prop:point} when both $r$ and $s$ vary.
%%%%%%%%%%%%%%%%%%%%%%
%%%%%%%%%%%%%%%%%%%%%%
%%%%%%%%%%%%%%%%%%%%%%
\begin{proposition}\label{prop:points_n}
	Let $E\in\Mf$ be a smooth set.
	Then
	\begin{equation*}
	\lim_{n \to +\infty}\frac{\tilde{J}_{r_n}^{s_n}(E)}{\bm{\beta}(r_n,s_n)}=\omega_{d-1}\Per(E).
	\end{equation*}
\end{proposition}
%%%%%%%%%%%%%%%%%%%%%%
%%%%%%%%%%%%%%%%%%%%%%
%%%%%%%%%%%%%%%%%%%%%%
\begin{proof}
By  Lemma \ref{porpsvilen} and by formula \eqref{deco} we have
\begin{align*}
&\frac{\tilde{J}_{r_n}^{s_n}(E)}{\bm{\beta}(r_n,s_n)}	
=\omega_{d-1} \Per(E)+\frac 1 {\bm{\beta}(r_n,s_n)}F^{s_n}_1(E) \\ 
& -\frac{1}{\bm{\beta}(r_n,s_n) } \int_{\partial E} \ud\mathcal{H}^{d-1}(y) \int_{(E \triangle H_{\nu_E(y)}^{-}(y)) \cap B(y,r_n)} \frac{1}{{r_n}^{s_n}} \frac{\vert (y-x)\cdot \nu_{E}(y) \vert }{\vert x-y \vert^{d}} \biggl( \frac{d+s_n}{ds_n}\biggr)\ud x \\
& +\frac{1}{\bm{\beta}(r_n,s_n)} \int_{\partial E} \ud\mathcal{H}^{d-1}(y) \int_{(E \triangle H_{\nu_E(y)}^{-}(y)) \cap B(y,r_n)} \frac{1}{{r_n}^{s_n}} \frac{\vert (y-x)\cdot \nu_{E}(y) \vert }{{r_n}^{d}}  \frac{1}{d}\ud x \\
& -\frac{1}{\bm{\beta}(r_n,s_n)} \frac 1 s_n\int_{\partial E} \ud\mathcal{H}^{d-1}(y) \int_{(E \triangle H_{\nu_E(y)}^{-}(y))\cap (B(y,1)\setminus B(y,r_n))} \frac{\vert (y-x)\cdot \nu_{E}(y)\vert }{\vert x-y \vert^{d+s_n}}\ud x \label{formula3lemmautilanches_n} \\
&-\frac{1}{\bm{\beta}(r_n,s_n)}\frac{1}{s_n}\int_{E} \mathcal{H}^{d-1}(E^c \cap \partial B(x,1))\ud x\,,
\end{align*} 
where $H_\nu^{-}(y)$ is the set defined in \eqref{semispaziotang}. 
Therefore, by arguing verbatim as in the proof of Proposition \ref{prop:point} and using \eqref{limbeta}, we deduce the claim.
\end{proof}
%%%%%%%%%%%%%%%%%%%%%%%%%%%%%%%
%%%%%%%%%%%%%%%%%%%%%%%%%%%%%%%
%%%%%%%%%%%%%%%%%%%%%%%%%%%%%%%
With Proposition \ref{prop:points_n} in hand, the proof of Theorem \ref{mainthm111}(iii) is fully analogous to the one of Theorem \ref{mainthm}(iii) and is omitted.
%%%%%%%%%%%%%%%%%%%%%%%%%%%%%%%
%%%%%%%%%%%%%%%%%%%%%%%%%%%%%%%
%%%%%%%%%%%%%%%%%%%%%%%%%%%%%%%
\subsection{Convergence of the $k_{r_n}^{s_n}$-curvature flows to the mean curvature flow}
Here we study the convergence of the curvatures $\Ku^{s_n}_{r_n}$ defined in \eqref{defcurv} to the classical mean curvature $\Ku^{1}$ in \eqref{cueu} when $r_n\to 0^+$ and $s_n\to 1^+$ simultaneously. As in Subsection \ref{ssc:convmcf} we use such a result to deduce the convergence of the corresponding geometric flows. 

First we prove the following lemma which is the analogous of Lemma \ref{lemconvcurv} in the case treated in this section.
%%%%%%%%%%%%%%%%%%%%%%
%%%%%%%%%%%%%%%%%%%%%%
%%%%%%%%%%%%%%%%%%%%%%
\begin{lemma} \label{lemmtecccccc}
Let $\{s_n\}_{n\in\N}\subset(1,+\infty)$ and $\{r_n\}_{n\in\N}\subset(0,1)$ be such that $r_n\to 0^+$ and $s_n\to 1^+$ as $n\to +\infty$\,.
Let $M, N \in \mathbb{R}^{(d-1)\times (d-1)}$ and let $\{M_n\}_{n\in\N},\{N_n\}_{n\in\N}\subset \mathbb{R}^{(d-1)\times (d-1)}$ be such that $M_n \rightarrow M, \; N_n \rightarrow N$ as $n \to +\infty$\,.
		Then, for every $\delta>0$\,, it holds
		\begin{equation}  \label{limserveatutconvecubetarsn}
		\begin{split}
		\lim_{n\to +\infty}& \bigg(\frac{1}{\bm{\beta}(r_n,s_n)} \Big(\int_{\Funon}k_{r_n}^{s_n}(\vert y \vert)\ud y 
		- \int_{\Fduen}k_{r_n}^{s_n}(\vert y \vert)\ud y\Big)\bigg) \\
		=&  \int_{\mathbb{S}^{d-2}} \theta^*(N-M) \theta \ud \mathcal{H}^{d-2}(\theta)\,,
		\end{split}
		\end{equation}
		where
\begin{equation*}%\label{AB}
\begin{aligned}
\Funon:=&\{y=(y',y_d) \in B(0,\delta)\,:\;  (y')^*M_n y'  \leq y_d \leq (y')^*N_n y' \}\\
\Fduen:=&\{y=(y',y_d) \in B(0,\delta)\,:\;  (y')^* N_n y' \leq y_d \leq (y')^*M_n y'\}\,.
\end{aligned}		
\end{equation*}
	\end{lemma}
%%%%%%%%%%%%%%%%%%%%%%%%%%%%
%%%%%%%%%%%%%%%%%%%%%%%%%%%%
%%%%%%%%%%%%%%%%%%%%%%%%%%%%
\begin{proof}
By arguing verbatim as in the proof of \eqref{limcurvutile123123} one can show that
		\begin{equation}\label{limcurvutile123beta}
\begin{split}
&\frac{1}{\bm{\beta}(r_n,s_n)}\int_{\Funon} k_{r_n}^{s_n}(\vert y \vert)\ud y
-\frac{1}{\bm{\beta}(r_n,s_n)}\int_{\Fduen} k_{r_n}^{s_n}(\vert y \vert)\ud y \\
=&  \frac{{r_n}^{1-s_n}}{(d+1)\bm{\beta}(r_n,s_n)} \int_{\mathbb{S}^{d-2}} \theta^*(N_n-M_n) \theta \ud \mathcal{H}^{d-2}(\theta)\\
& +\frac{1}{\bm{\beta}(r_n,s_n)} \int_{\mathbb{S}^{d-2}} \ud \mathcal{H}^{d-2}(\theta) \int_{r_n}^{\delta}\ud\rho \frac{1}{\rho^{s_n}} \int_{\theta^* M_n \theta}^{\theta^*N_n \theta} \frac{1}{(1+\rho^2t^2)^{\frac{d+s_n}{2}}}\ud t\,,\\
&+\eta_n\,,
\end{split}
\end{equation}
where $\eta_n\to 0$ as $n\to +\infty$\,.
It is easy to see that
\begin{equation*}%\label{puredopo}
\begin{aligned}
\lim_{n \rightarrow + \infty} \frac{{r_n}^{1-s_n}}{\bm{\beta}(r_n,s_n)}=0\,,
\end{aligned}
\end{equation*}	
whence we get
\begin{equation}\label{betconvcurvlef1}
\lim_{n \rightarrow + \infty}\frac{r_n^{1-s_n}}{(d+1)\bm{\beta}(r_n,s_n)} \int_{\mathbb{S}^{d-2}}\theta^* (N_n-M_n) \theta \ud \mathcal{H}^{d-2}(\theta)=0\,.
\end{equation}
%%%%%%%%%%%%%%%%%%%%%%%%%%%%%%%
Now we claim that for every $t\in\R$
\begin{equation}\label{betaconvcurvf2}
\lim_{n\to +\infty}\frac{1}{\bm{\beta}(r_n,s_n)} \int_{r_n}^{\delta} \frac{1}{\rho^{s_n}}  \frac{1}{(1+\rho^2t^2)^{\frac{d+s_n}{2}}}\ud\rho=1\,,
\end{equation}
which in view of
 \eqref{limcurvutile123beta} and \eqref{betconvcurvlef1} and of the Dominate Convergence Theorem, implies \eqref{limserveatutconvecubetarsn}.
In order to prove \eqref{betaconvcurvf2}, we first notice that 
\begin{equation*}%\label{spezero}
\begin{aligned}
&\frac{1}{\bm{\beta}(r_n,s_n)} \int_{r_n}^{\delta}  \frac{1}{\rho^{s_n}}  \frac{1}{(1+\rho^2t^2)^{\frac{d+s_n}{2}}}\\
=&\frac{1}{\bm{\beta}(r_n,s_n)} \int_{r_n}^{\delta}  \frac{1}{\rho^{s_n}}\ud\rho-\frac{1}{\bm{\beta}(r_n,s_n)} \int_{r_n}^{\delta}  \frac{1}{\rho^{s_n}}\bigg(1-\frac{1}{(1+\rho^2t^2)^{\frac{d+s_n}{2}}}\bigg)\ud \rho\\
=&\frac{1}{\bm{\beta}(r_n,s_n)} \frac{r_n^{1-s_n}-\delta^{1-s_n}}{s_n-1}-\frac{1}{\bm{\beta}(r_n,s_n)} \int_{r_n}^{\delta}  \frac{1}{\rho^{s_n}}\bigg(1-\frac{1}{(1+\rho^2t^2)^{\frac{d+s_n}{2}}}\bigg)\ud \rho\,,
\end{aligned}
\end{equation*}
so that, by the very definition of $\bm{\beta}$ in \eqref{defbetaultimo}, it is enough to show that
\begin{equation}\label{speuno}
\begin{aligned}
&\limsup_{n\to +\infty}\frac{1}{\bm{\beta}(r_n,s_n)} \int_{r_n}^{\delta} \frac{1}{\rho^{s_n}}\bigg(1-  \frac{1}{(1+\rho^2t^2)^{\frac{d+s_n}{2}}}\bigg)\ud\rho=0\,.
\end{aligned}
\end{equation}
As for the proof of \eqref{speuno} we argue as follows. First we notice that, setting $S:=\sup_{n\in\N}s_n$,
$$
(1+\rho^2t^2)^{\frac{d+s_n} 2}\le 1+C(d,S,t)\rho^2\,,
$$
so that, for $n$ large enough,
\begin{equation*}
\begin{aligned}
&\frac{1}{\bm{\beta}(r_n,s_n)} \int_{r_n}^{\delta} \frac{1}{\rho^{s_n}}\bigg(1-  \frac{1}{(1+\rho^2t^2)^{\frac{d+s_n}{2}}}\bigg)\ud\rho\\
=&\frac{1}{\bm{\beta}(r_n,s_n)} \int_{r_n}^{\delta}\frac{1}{\rho^{s_n}}\frac{(1+\rho^2t^2)^{\frac{d+s_n}{2}}-1}{(1+\rho^2t^2)^{\frac{d+s_n}{2}}}\ud\rho\\
\le&\frac{1}{\bm{\beta}(r_n,s_n)} \int_{r_n}^{\delta} C(d,S,t)\rho^{2-s_n}\ud \rho
\to 0\qquad\textrm{as }n\to +\infty\,,
\end{aligned}
\end{equation*}
thus concluding the proof of \eqref{speuno} and of the whole lemma.
\end{proof}
%%%%%%%%%%%%%%%%%%%%%%%%%%%%%%%%%
%%%%%%%%%%%%%%%%%%%%%%%%%%%%%%%%%
%%%%%%%%%%%%%%%%%%%%%%%%%%%%%%%%%
By using Lemma \ref{lemmtecccccc} in place of Lemma \ref{lemconvcurv} in the proof of Theorem \ref{convcurvthmprinc}, one can prove the following result.
%%%%%%%%%%%%%%%%%%%%%%%%%%%%%%%%%
%%%%%%%%%%%%%%%%%%%%%%%%%%%%%%%%%
%%%%%%%%%%%%%%%%%%%%%%%%%%%%%%%%%
	\begin{theorem}\label{convcurva}
	Let $\{r_n\}_{n\in\N}\subset(0,1)$ and $s_n\subset(1,+\infty)$ be such that $r_n\to 0^+$ and $s_n\to 1^+$ as $n\to +\infty$\,.
Let $\{E_n\}_{n\in \mathbb{N}} \subset \mathfrak{C}$ such that $ E_{n} \rightarrow E $ in $\mathfrak{C}$ as $n \rightarrow + \infty$, for some $E \in \mathfrak{C}$. Then for every $x \in \partial E \cap \partial E_n$ for every $n \in \mathbb{N}$,
	\begin{equation*}
	\lim_{n \rightarrow + \infty} \frac{\Ku_{r_n}^{s_n}(x,E_n)}{\bm{\beta}(r_n,s_n)}=\omega_{d-1}\Ku^1(x,E).
	\end{equation*}
\end{theorem}
%%%%%%%%%%%%%%%%%%%%%%%%%%%%%%%%%
%%%%%%%%%%%%%%%%%%%%%%%%%%%%%%%%%
%%%%%%%%%%%%%%%%%%%%%%%%%%%%%%%%%
Finally, by  using Theorem \ref{convcurva} in place of Theorem \ref{convcurvthmprinc} in the proof of Theorem \ref{maincur}, one can prove the following result which provides the convergence of the $k^{s_n}_{r_n}$-nonlocal curvature flows when $r_n\to 0^+$ and $s_n\to 1^+$\,. 
%%%%%%%%%%%%%%%%%%%%%%%%%%%%%%%%%
%%%%%%%%%%%%%%%%%%%%%%%%%%%%%%%%%
%%%%%%%%%%%%%%%%%%%%%%%%%%%%%%%%%
\begin{theorem}\label{maincurbis}
Let $\{r_n\}_{n\in\N}\subset(0,1)$ and $s_n\subset(1,+\infty)$ be such that $r_n\to 0^+$ and $s_n\to 1^+$ as $n\to +\infty$\,.
Let $u_0\in C(\R^d)$ be a uniformly continuous function with $u_0=C_0$ in $\R^d\setminus B(0,R_0)$ for some $C_0,R_0\in\R$ with $R_0>0$\,. For every $n\in\N$\,, let $u^{s_n}_{r_n}:\R^d\times [0,+\infty)\to\R$ be the viscosity solution to the Cauchy problem \eqref{cauchy} (with $r$ and $s$ replaced by $r_n$ and $s_n$, respectively). Then, setting $v^{s_n}_{r_n}(x,t):=u^{s_n}_{r_n}(x,\frac{t}{\bm{\beta}(r_n,s_n)})$ for all $x\in\R^d$\,, $t\ge 0$\,, we have that,  for every $T>0$\,,
$v^{s_n}_{r_n}$ uniformly converge to $u$ as $n\to +\infty$ in $\R^d\times [0,T]$\,, where $u:\R^d\times[0,+\infty)\to\R$ is the viscosity solution to the classical mean curvature flow \eqref{cauchycl}.
\end{theorem}
 %%%%%%%%%%%%%%%%%%%%%%%%%%%%%%%%%%%%%%%%
 %%%%%%%%%%%%%%%%%%%%%%%%%%%%%%%%%%%%%%%%
 %%%%%%%%%%%%%%%%%%%%%%%%%%%%%%%%%%%%%%%%
 %%%%%%%%%%%%%%%%%%%%%%%%%%%%%%%%%%%%%%%%
 %%%%%%%%%%%%%%%%%%%%%%%%%%%%%%%%%%%%%%%%
 %%%%%%%%%%%%%%%%%%%%%%%%%%%%%%%%%%%%%%%%
\section{Anisotropic kernels and applications to dislocation dynamics}\label{anke}
In this section we study the asymptotic behavior of supercritical nonlocal perimeters and the corresponding geometric flows in the case of anisotropic kernels. Moreover, we present an application to the dynamics of dislocation curves in two dimensions.
  %%%%%%%%%%%%%%%%%%%%%%%%%%%%%%%%%%%%%%%%
 %%%%%%%%%%%%%%%%%%%%%%%%%%%%%%%%%%%%%%%%
 %%%%%%%%%%%%%%%%%%%%%%%%%%%%%%%%%%%%%%%% 
\subsection{Anisotropic kernels}\label{sec:ani}
Let $g\in C(\mathbb{S}^{d-1};(0,+\infty))$ be such that $g(\xi)=g(-\xi)$ for every $\xi\in\mathbb{S}^{d-1}$\,. 	
For every $s\ge 1$ and for every $r>0$ we define the function $\kg:\R^{d}\setminus\{0\}\to (0,+\infty)$
as $\kg(x):=g\big(\frac{x}{|x|}\big)k^s_r(|x|)$\,, where $k^s_r$ is defined in \eqref{kernel}\,.
Here we study the asymptotic behavior, as $r\to 0^+$ of the functionals $\Jgsr:\Mf\to[0,+\infty)$ defined by
\begin{equation}\label{aniJ}
\Jgsr(E):=\int_{E}\int_{E^c}\kg(y-x)\ud y\ud x\,.
\end{equation}
In Proposition \ref{prop:pointani} below we will show that the functionals $\Jgsr$ scaled by $\sca^s(r)$ converge as $r\to 0^+$ to the anisotropic perimeter $\Per^g$ defined on finite perimeter sets as
\begin{equation}\label{aniper}
\Per^g(E):=\int_{\partial^{*}E}\ffi^g(\nu_E(x))\ud\mathcal{H}^{d-1}(x)\,,
\end{equation}
where the density $\ffi^g$ is given by
\begin{equation}\label{density}
\ffi^g(\nu):=\int_{\{\xi\in\mathbb{S}^{d-1}\,:\,\xi\cdot\nu\ge 0\}}g(\xi)\,\xi\cdot\nu\ud\mathcal{H}^{d-1}(\xi)\,,\qquad\textrm{ for every }\nu\in\mathbb{S}^{d-1}\,.
\end{equation}
 %%%%%%%%%%%%%%%%%%%%%%%%%%%%%%%%%%%%%%%%
 %%%%%%%%%%%%%%%%%%%%%%%%%%%%%%%%%%%%%%%%
 %%%%%%%%%%%%%%%%%%%%%%%%%%%%%%%%%%%%%%%%
\begin{proposition}\label{prop:pointani}
 For every $s\ge 1$ and for every set $E\in\Mf$ of finite perimeter it holds
\begin{equation*}%\label{pointani}
\lim_{r\to 0^+}\frac{\Jgsr(E)}{\sca^s(r)}=\Per^g(E)\,.
\end{equation*}
\end{proposition}
 %%%%%%%%%%%%%%%%%%%%%%%%%%%%%%%%%%%%%%%%
 %%%%%%%%%%%%%%%%%%%%%%%%%%%%%%%%%%%%%%%%
 %%%%%%%%%%%%%%%%%%%%%%%%%%%%%%%%%%%%%%%%
 \begin{proof}
 First we claim the following anisotropic version of formula \eqref{ensvil1}:
 \begin{equation}\label{ensvil1ani}
	\begin{split}
&\int_{E}\int_{E^c\cap B(x,1)}\kg(y-x)\ud y\ud x\\
=&  \frac{d+s}{d s r^s}\int_{\partial^* E} \ud\mathcal{H}^{d-1}(y) \int_{E \cap B(y,r)}g\Big(\frac{y-x}{|x-y|}\Big)\frac{(y-x)}{|x-y|^d}\cdot \nu_{E}(y) \ud x\\
& - \frac{1}{dr^{d+s}}\int_{\partial^* E} \ud\mathcal{H}^{d-1}(y) \int_{E \cap B(y,r)}g\Big(\frac{y-x}{|x-y|}\Big)(y-x)\cdot \nu_{E}(y) \ud x\\
    &+\frac 1 s \int_{\partial^* E} \ud\mathcal{H}^{d-1}(y) \int_{E\cap (B(y,1)\setminus B(y,r))} g\Big(\frac{y-x}{|x-y|}\Big)\frac{(y-x)\cdot \nu_{E}(y)}{\vert x-y \vert^{d+s}}\ud x\\
    &-\frac 1 s\int_{E}\ud x\int_{E^c\cap \partial B(x,1)}g\Big(\frac{y-x}{|x-y|}\Big)\ud \mathcal{H}^{d-1}(y)\,.
	\end{split}
	\end{equation} 
If $g\in C^1(\mathbb{S}^{d-1})$, the proof of \eqref{ensvil1ani} is identical to the proof of \eqref{ensvil1}, noticing that $\nabla g\big(\frac{x}{|x|}\big)\cdot T^s_r(x)=0$ for every $x\in\R^d\setminus\{0\}$\,, with $T^s_r$ defined in \eqref{Trfunzdef}. The case $g\in C(\mathbb{S}^{d-1})$ follows by standard density arguments. 
 \end{proof}
  %%%%%%%%%%%%%%%%%%%%%%%%%%%%%%%%%%%%%%%%
 %%%%%%%%%%%%%%%%%%%%%%%%%%%%%%%%%%%%%%%%
 %%%%%%%%%%%%%%%%%%%%%%%%%%%%%%%%%%%%%%%%
In Theorem \ref{mainthmani} below we will see that the functionals $\Jgsr$ actually $\Gamma$-converge, as $r\to 0^+$, to $\Per^g$\,.
 %%%%%%%%%%%%%%%%%%%%%%%%%%%%%%%%%%%%%%%%
 %%%%%%%%%%%%%%%%%%%%%%%%%%%%%%%%%%%%%%%%
 %%%%%%%%%%%%%%%%%%%%%%%%%%%%%%%%%%%%%%%%
\begin{theorem}\label{mainthmani}
Let $s\ge 1$ and 
let $\{r_n\}_{n\in\N}\subset (0,+\infty)$ be such that $r_n\to 0$ as $n\to +\infty$. The following $\Gamma$-convergence result holds true.
\begin{itemize}
%%%%%%%%%%%%%%%%%%%%%%%%
\item[(i)] (Compactness) Let $U\subset\R^d$ be an open bounded set and let $\{E_n\}_{n\in\N}\subset\Me$ be such that $E_n\subset U$ for every $n\in\N$ and
\begin{equation*}
 \Jgsn(E_n)\le M \sca^s(r_n)\qquad\textrm{for every }n\in\N,
 \end{equation*}
for some constant $M$ independent of $n$.
Then, up to a subsequence, $\chi_{E_n}\to \chi_E$ strongly in $L^1(\R^d)$ for some set $E\in\Mf$ with $\Per(E)<+\infty$.
%%%%%%%%%%%%%%%%%%%%%%%%%
\item[(ii)] (Lower bound) Let $E\in\Mf$. For every $\{E_n\}_{n\in\N}\subset\Mf$ with $\chi_{E_n}\to\chi_E$ strongly in $L^1(\R^d)$ it holds
\begin{equation*}%\label{trueliminfani}
\Per^g(E)\le\liminf_{n\to +\infty}\frac{\Jgsn(E_n)}{\sca^s(r_n)}\,.
\end{equation*}
%%%%%%%%%%%%%%%%%%%%%%%%%%
\item[(iii)] (Upper bound) For every $E\in\Mf$ there exists $\{E_n\}_{n\in\N}\subset\Mf$ such that $\chi_{E_n}\to\chi_E$ strongly in $L^1(\R^d)$ and
\begin{equation*}%\label{limsupani}
\Per^g(E)=\lim_{n\to +\infty}\frac{\Jgsn(E_n)}{\sca^s(r_n)}\,.
\end{equation*}
%%%%%%%%%%%%%%%%%%%%%%%%%%
\end{itemize}
\end{theorem}
  %%%%%%%%%%%%%%%%%%%%%%%%%%%%%%%%%%%%%%%%
 %%%%%%%%%%%%%%%%%%%%%%%%%%%%%%%%%%%%%%%%
 %%%%%%%%%%%%%%%%%%%%%%%%%%%%%%%%%%%%%%%%
\begin{proof}
The proof of the compactness property (i) follows by Theorem \ref{mainthm}(i), once noticed that there exist two positive constants $c_1<c_2$ such that $c_1\le g(\theta)\le c_2$ for every $\theta\in\mathbb{S}^{d-1}$\,.
As for the proof of the $\Gamma$-liminf inequality in (ii) one can argue verbatim as in the proof of Theorem \ref{mainthm}(ii), using the following inequality
\begin{equation}\label{quasiliminfani}
\int_{Q^\nu}\int_{Q^\nu}g\Big(\frac{x-y}{|x-y|}\Big)k^s_r(|x-y|)|\chi_{E}(x)-\chi_{E}(y)|\ud y\ud x\ge (1-\ep)\sca^s(r)\ffi^g(\nu)\,,
\end{equation}
in place of \eqref{quasiliminf}.
The proof of \eqref{quasiliminfani} under the assumptions of Lemma \ref{lemma18} is identical to the proof of Lemma \ref{lemma18} (see \eqref{start3}).

Finally, the $\Gamma$-limsup inequality (iii) follows as in the isotropic case Theorem \ref{mainthm}(iii) using Proposition \ref{prop:pointani} in place of Proposition \ref{prop:point}.
\end{proof}
  %%%%%%%%%%%%%%%%%%%%%%%%%%%%%%%%%%%%%%%%
 %%%%%%%%%%%%%%%%%%%%%%%%%%%%%%%%%%%%%%%%
 %%%%%%%%%%%%%%%%%%%%%%%%%%%%%%%%%%%%%%%%
We introduce the notion of $\kg$ curvature and we study the convergence as $r\to 0^+$ of the corresponding geometric flows.
 
Let $s\ge 1$, $r>0$ and $E\in\Reg$\,.
 For every $x\in\partial E$ we define the {\it $\kg$-curvature of $E$ at $x$} as
\begin{equation}\label{defcurvg}
\Ku^{g,s}_r(x,E):=\int_{\R^d}(\chi_{E^c}(y)-\chi_{E}(y))\kg(x-y)\ud y.
\end{equation}
 %%%%%%%%%%%%%%%%%%%%%%%%%%%%%%%%%%%%%%%%
 %%%%%%%%%%%%%%%%%%%%%%%%%%%%%%%%%%%%%%%%
 %%%%%%%%%%%%%%%%%%%%%%%%%%%%%%%%%%%%%%%%
 \begin{remark}\label{rm:ovvi}
 \rm{
 We notice that for every $E\in\Reg$ and for every $x\in\partial E$ it holds
 \begin{equation}\label{form:ovvi}
 \begin{aligned}
 \Ku^{g,s}_r(x,E)=&\int_{\R^d}k^{g,s}_r(x-y)\ud y-2\int_{E}k^{g,s}_r(x-y)\ud y\\
 =& \int_{\R^d}k^{g,s}_r(z)\ud z-2 k^{g,s}_r\ast\chi_E(x)\\
=&\Big(-2 k^{g,s}_r+ \int_{\R^d}k^{g,s}_r(z)\ud z\,\bm{\delta}_0\Big)\ast \chi_{E}(x)\,,
 \end{aligned}
 \end{equation}
 where $\ast$ denotes the convolution operator and $\bm{\delta}_0$ is the Dirac delta centered at $0$\,.
By \eqref{form:ovvi} we have that $\Ku^{g,s}_r$ is exactly the type of curvatures considered in \cite[formula (1.4)]{DFM}.
We remark that the positive part of the curvature $\Ku^{g,s}_r$ is concentrated on a point. This is why, as already observed in \cite{DFM}, the curvature  $\Ku^{g,s}_r$, although having a positive contribution, still satisfies the desired monotonicity property with respect to set inclusion (see the proof of (M) in Proposition \ref{proprietàcurvature1}).  
 }
 \end{remark}
 %%%%%%%%%%%%%%%%%%%%%%%%%%%%%%%%%%%%%%%%
 %%%%%%%%%%%%%%%%%%%%%%%%%%%%%%%%%%%%%%%%
 %%%%%%%%%%%%%%%%%%%%%%%%%%%%%%%%%%%%%%%% 
 We first show that $\Ku^{g,s}_r$ is the first variation of $\Jgsr$ in the sense specified by the following proposition, which is the anisotropic analogous of Proposition  \ref{curvfirstvarper}.
  %%%%%%%%%%%%%%%%%%%%%%%%%%%%%%%%%%%%%%%%
 %%%%%%%%%%%%%%%%%%%%%%%%%%%%%%%%%%%%%%%%
 %%%%%%%%%%%%%%%%%%%%%%%%%%%%%%%%%%%%%%%%
\begin{proposition}
	Let $s\geq 1$, $r>0$, and $E \in \mathfrak{C}$.   
	Let $\Phi: \mathbb{R} \times \mathbb{R}^d \rightarrow \mathbb{R}^d$ be as in Proposition \ref{curvfirstvarper}\,.
	Setting $E_t:= \Phi_t(E)$ and $\Psi(\cdot):= \frac{\partial }{\partial t}\Phi_t(\cdot) \big|_{t=0}$\,, we have
	\begin{equation*}%\label{claig}
	\frac{\ud}{\ud t} \Jgsr(E_t){\bigg|_{t=0}}= \int_{\partial E} \,\Ku^{g,s}_r(x,E) \Psi(x)\cdot \nu_{E}(x) \ud\mathcal{H}^{d-1}(x)\,.
	\end{equation*}	
\end{proposition}
  %%%%%%%%%%%%%%%%%%%%%%%%%%%%%%%%%%%%%%%%
 %%%%%%%%%%%%%%%%%%%%%%%%%%%%%%%%%%%%%%%%
 %%%%%%%%%%%%%%%%%%%%%%%%%%%%%%%%%%%%%%%%
\begin{proof}
If $g\in C^1$\,, then the proof is fully analogous to the proof of Proposition \ref{curvfirstvarper}\,.
The case when $g\in C^0$ follows by a density argument, using that if $\{g_n\}_{n\in\N}\subset C^1(\mathbb{S}^{d-1};(0,+\infty))$ uniformly converges to $g$, $E_n\to E$ in $\Reg$ and $x_n\to x$\,, then
$\Ku^{g_n,s}_r(x_n,E_n)$ converge to $\Ku^{g,s}_r(x,E)$\,.
Such a continuity property can be proved as in Proposition \ref{proprietàcurvature1} (UC).
\end{proof}
%%%%%%%%%%%%%%%%%%%%%%%%%%%%%%%%%%%%%%%%
%%%%%%%%%%%%%%%%%%%%%%%%%%%%%%%%%%%%%%%%
%%%%%%%%%%%%%%%%%%%%%%%%%%%%%%%%%%%%%%%%
 By arguing as in the proof of Proposition \ref{proprietàcurvature1} one can show that the curvatures $\Ku^{g,s}_r$ satisfy properties (M), (T), (S), (B), (UC).
 %%%%%%%%%%%%%%%%%%%%%%%%%%%%%%%%%%%%%%%%
%%%%%%%%%%%%%%%%%%%%%%%%%%%%%%%%%%%%%%%%
%%%%%%%%%%%%%%%%%%%%%%%%%%%%%%%%%%%%%%%%
Now we introduce the (local) anisotropic curvatures $\Ku^{g,1}$ defined as follows.
Let $E\in\Reg$ and $x\in\partial E$; 
in a neighborhood of $x$,   $\partial E$ is the graph of function $f \in C^{2}(H_{\nu_E(x)}^{0}(x))$ (see \eqref{piatto} for the definition of $H^0_\nu(x)$) with $\D f(x)=0$ (here and below $\D f$ and $\D^2f$ are computed with respect to a given system of orthogonal coordinates on $H_{\nu_E(x)}^{0}(x)$).  The anisotropic mean curvature of $\partial E$ at $x$ is given by 
\begin{equation}\label{clameang} 
\Ku^{g,1}(x,E)=-\int_{H_{\nu_E(x)}^{0}(x) \cap \mathbb S^{d-1}}g(\xi)\xi^*\D^2 f(x)\xi  \ud\mathcal{H}^{d-2}(\xi)\,.
\end{equation}
One can check that $\Ku^{g,1}$ is the first variation of $\Per^g$ in the sense specified by Proposition \ref{curvfirstvarper} (we refer to \cite{Be} for the first variation formula of generic anisotropic perimeters, while we leave to the reader the computations for the specific anisotropic density $\ffi^g$ considered here, defined in \eqref{density}).
Notice that if $g\equiv 1$\,, then $\Ku^{g,1}=\omega_{d-1}\Ku^1$ where $\Ku^1$ is the classical mean curvature defined in \eqref{cueu}.
Moreover, one can check that $\Ku^{g,1}$ satisfies properties (M), (T), (S), (B), (UC') in  Proposition \ref{procla}.

In Proposition \ref{convecug} below we show that the curvatures $\Ku^{g,s}_r$ converge, as $r\to 0^+$\,, to the anisotropic curvature $\Ku^{g,1}$\,. 
%%%%%%%%%%%%%%%%%%%%%%%%%%%%%%%%%%%%%%%%
%%%%%%%%%%%%%%%%%%%%%%%%%%%%%%%%%%%%%%%%
%%%%%%%%%%%%%%%%%%%%%%%%%%%%%%%%%%%%%%%%
\begin{theorem}\label{convecug}
	Let $s\geq 1$. Let $\{E_r\}_{r>0} \subset \mathfrak{C}$ be such that $ E_{r} \rightarrow E $ in $\mathfrak{C}$ as $r \rightarrow 0^+$, for some $E \in \mathfrak{C}$. Then, for every $x \in \partial E \cap \partial E_r$ for all $r>0$, it holds
	\begin{equation}\label{convcug}
	\lim_{r \rightarrow 0^+} \frac{\Ku_{r}^{g,s}(x,E_r)}{\sigma^s(r)}=\Ku^{g,1}(x,E).
	\end{equation}
\end{theorem}
  %%%%%%%%%%%%%%%%%%%%%%%%%%%%%%%%%%%%%%%%
 %%%%%%%%%%%%%%%%%%%%%%%%%%%%%%%%%%%%%%%%
 %%%%%%%%%%%%%%%%%%%%%%%%%%%%%%%%%%%%%%%%
 \begin{proof}
 The proof of \eqref{convcug} is fully analogous to the one of Theorem \ref{convcurvthmprinc}
 and in particular it is based on a suitable anisotropic variant of Lemma  \ref{lemconvcurv}.
 In fact,  Lemma \ref{lemconvcurv} can be extended also to the anisotropic case with \eqref{limserveatutconvecur} replaced by 
 \begin{equation}\label{basta}
 \begin{split}
		\lim_{r \rightarrow 0^+}& \bigg(\frac{1}{\sigma^s(r)} \Big(\int_{\Funo}g\Big(\frac{y}{|y|}\Big)k_{r}^{s}(\vert y \vert)\ud y 
		- \int_{\Fdue}g\Big(\frac{y}{|y|}\Big) k_{r}^{s}(\vert y \vert)\ud y\Big)\bigg) \\
		=&  \int_{\mathbb{S}^{d-2}} g(\theta,0)\theta^*(N-M) \theta \ud \mathcal{H}^{d-2}(\theta)\,.
		\end{split}
 \end{equation}
If $\nu_E(x)=e_d$, one can argue verbatim as in the proof of Theorem \ref{convcurvthmprinc}, clearly using \eqref{basta} in place of \eqref{limserveatutconvecur}. The same proof with only minor notational changes can be adapted also to the case  $\nu_E(x)\neq e_d$\,.
 \end{proof}
   %%%%%%%%%%%%%%%%%%%%%%%%%%%%%%%%%%%%%%%%
 %%%%%%%%%%%%%%%%%%%%%%%%%%%%%%%%%%%%%%%%
 %%%%%%%%%%%%%%%%%%%%%%%%%%%%%%%%%%%%%%%%
We are now in a position to state our result on the convergence of the geometric flows of $\Ku^{g,s}_r$ as $r\to 0^+$\,, whose proof is omitted, being fully analogous to the one of Theorem \ref{maincur}.
%%%%%%%%%%%%%%%%%%%%%%%%%%%%%%%%%%%%%%%%
%%%%%%%%%%%%%%%%%%%%%%%%%%%%%%%%%%%%%%%%
%%%%%%%%%%%%%%%%%%%%%%%%%%%%%%%%%%%%%%%%
\begin{theorem}\label{maincurg}
Let $s\ge 1$ be fixed. Let $u_0\in C(\R^d)$ be a uniformly continuous function with $u_0=C_0$ in $\R^d\setminus B(0,R_0)$ for some $C_0,R_0\in\R$ with $R_0>0$\,. For every $r>0$\,, let $u^s_r:\R^d\times [0,+\infty)\to\R$ be the viscosity solution to the Cauchy problem 
\begin{equation*}%\label{cauchyg}
\begin{cases}
\partial_t u(x,t)+|\mathrm{D}u(x,t)|\Ku^{g,s}_r(x,\{y\,:\,u(y,t)\ge u(x,t)\})=0\\
u(x,0)=u_0(x)\,.
\end{cases}
\end{equation*}
Then, setting $v^s_r(x,t):=u^s_r(x,\frac{t}{\sca^s(r)})$ for all $x\in\R^d$\,, $t\ge 0$\,, we have that,  for every $T>0$\,,
$v^s_r$ uniformly converge to $u$ as $r\to 0^+$ in $\R^d\times [0,T]$\,, where $u:\R^d\times[0,+\infty)\to\R$ is the viscosity solution to the anisotropic mean curvature flow
\begin{equation*}%\label{cauchyclg}
\begin{cases}
\partial_t u(x,t)+|\mathrm{D}u(x,t)|\Ku^{g,1}(x,\{y\,:\,u(y,t)\ge u(x,t)\})=0\\
u(x,0)=u_0(x)\,.
\end{cases}
\end{equation*}
\end{theorem}
%%%%%%%%%%%%%%%%%%%%%%%%%%%%%%%%%%%%%%%%
%%%%%%%%%%%%%%%%%%%%%%%%%%%%%%%%%%%%%%%%
%%%%%%%%%%%%%%%%%%%%%%%%%%%%%%%%%%%%%%%%
%%%%%%%%%%%%%%%%%%%%%%%%%%%%%%%%%%%%%%%%
%%%%%%%%%%%%%%%%%%%%%%%%%%%%%%%%%%%%%%%%
%%%%%%%%%%%%%%%%%%%%%%%%%%%%%%%%%%%%%%%%
\subsection{Applications to dislocation dynamics}\label{ssc:appdis}
Here we apply the results in Subsection \ref{sec:ani} to the  motion of curved dislocations in the plane. 
To this purpose, we briefly recall and describe, in an informal way, some notions about the isotropic linearized elastic energy induced by planar dislocations; such notions are well known to experts and we refer to classic books such as \cite{HL} for an exhaustive monography on this subject. 

Let $E$ be a bounded region of the plane $\R^2=\R^3\cap \{z\in\R^3\,: \, z_3=0\}$, representing a plastic slip region with Burgers vector $b=e_1=(1,0,0)$. Formally, the elastic energy induced by such a dislocation is given by
\begin{equation}\label{hirlo}
J(E):=\frac{\bm{\mu}}{8\pi}\int_E\int_{E^c}\frac{1}{|x-y|^5}\Big(\frac{1+\bm{\nu}}{1-\bm{\nu}}x_1^2+\frac{1-2\bm{\nu}}{1-\bm{\nu}}x_2^2\Big) \ud y\ud x\,,
\end{equation}
where $\bm{\mu}>0$ and $\bm{\nu}\in (-1,\frac 1 2)$  are the shear modulus and the Poisson's ratio, respectively.
Formula \eqref{hirlo} can be deduced by \cite[formula (4-44)]{HL}, by integrating by parts.
Clearly, the energy $J$ in \eqref{hirlo} is always infinite whenever $E$ is non-empty.
It is well understood that such an infinite energy should be suitably truncated through ad hoc core regularizations, specific of the microscopic details of the underlying crystal. The specific choice of the core regularization, giving back the physically relevant (finite) elastic energy induced by the dislocation is, for our purposes, irrelevant. Here we adopt the energy-renormalization procedure introduced in \eqref{aniJ}. First we set 
\begin{equation}\label{defgdis}
g(\xi):=\frac{\bm{\mu}}{8\pi}\Big(\frac{1+\bm{\nu}}{1-\bm{\nu}}\xi_1^2+\frac{1-2\bm{\nu}}{1-\bm{\nu}}\xi_2^2\Big)\,,\quad\textrm{ for every }\xi\in\mathbb{S}^1\,,
\end{equation}
 and we notice that the energy in \eqref{hirlo} can be (formally) rewritten as
\begin{equation*}%\label{hirlo2}
\begin{aligned}
J(E)=\int_{E}\int_{E^c}g\Big(\frac{x-y}{|x-y|}\Big)\frac{1}{|x-y|^3}\ud y\ud x
=\int_{E}\int_{E^c}k^g(x-y)\ud y\ud x\,,
\end{aligned}
\end{equation*}
where $k^g$ is defined by $k^g(z):=g(\frac{z}{|z|})\frac{1}{|z|^3}$\,. 
The core-regularization of $J$ is given by the functional $\tilde{J}^{g,1}_r $ defined by \eqref{aniJ}, where the parameter $r>0$ plays the role of the core-size. 
Now, consider the dynamics of a dislocation curve, enclosing a (moving) bounded set $E$, with Burgers vector equal to $e_1$, governed by a self-energy release mechanism. We consider a geometric  evolution, that  can be formally understood as the gradient flow of the self-energy $\tilde J^{g,1}_r$ with respect to an $\mathrm{L}^2$ structure on the (graphs locally describing the) evolving dislocation curve. If the energy were the standard perimeter, this evolution would be nothing but the standard mean curvature flow. Notice that the energy considered here  is nonlocal; moreover, although it is derived  from isotropic linearized elasticity,  it has in fact an anisotropic  dependence  (induced by the direction of the given Burgers vector) on the normal to the curve. Another possible source of anisotropy is the so-called mobility, depending on the microscopic details of the underlying crystalline lattice; here, for simplicity, we assume that such a mobility is in fact isotropic, equal to one.   The dynamics discussed above corresponds 
to  the geometric evolution where the normal velocity of the evolving dislocation curve at any point $x$ is given by $-\Ku^{g,1}_r$, defined in \eqref{defcurvg}. 

In order to study the asymptotic behavior, as $r\to 0^+$, of the dynamics described above we use the results developed in Subsection \ref{sec:ani}.
First, we notice that that  the function $g$ defined in \eqref{defgdis} is continuous (actually, it is smooth) and even, so that it satisfies the assumptions required in Subsection \ref{sec:ani}. Moreover, recalling \eqref{density} and \eqref{clameang}, for the choice of $g$ in \eqref{defgdis}, an easy computation shows that
\begin{eqnarray*}
\ffi^{g}(\nu)&=&\frac{\bm{\mu}}{12\pi}\Big(\frac{1+\bm{\nu}}{1-\bm{\nu}}(1+\nu_1^2)+\frac{1-2\bm{\nu}}{1-\bm{\nu}}(1+\nu_2^2)\Big)\,,\textrm{ for every }\nu\in\mathbb{S}^1\,,\\ %\label{cudis}
\Ku^{g,1}(x,E)&=&\frac{\bm{\mu}}{8\pi}\Big(\frac{1+\bm{\nu}}{1-\bm{\nu}}(\nu_E(x))_2^2+\frac{1-2\bm{\nu}}{1-\bm{\nu}}(\nu_E(x))_1^2\Big)\,,\textrm{ for every }E\in\Reg\,, x\in \partial E\,.
\end{eqnarray*}

Therefore,  
by Theorem \ref{maincurg}, the unique (in the level set sense) dislocation dynamics described above,  converges, as $r\to 0^+$, to a degenerate evolution where the dislocation disappear instantaneously. After a logarithmic in time reparametrization, the evolution converges to the anisotropic mean curvature flow governed by the  release  of the line tension energy $\Per^g$  \eqref{aniper}, corresponding to the anisotropic energy density $\ffi^g$  defined above. Such a dynamics is nothing but   the evolution $t\mapsto \partial E_t$, where the normal velocity of the evolving dislocation curve $\partial E_t$ at any point $x\in \partial E_t$ is given by the (opposite of the) anisotropic curvature $\Ku^{g,1}(x,E_t)$ defined above.
%%%%%%%%%%%%%%%%%%%%%%%%%%%%%%%%%%%%%%%%
%%%%%%%%%%%%%%%%%%%%%%%%%%%%%%%%%%%%%%%%
%%%%%%%%%%%%%%%%%%%%%%%%%%%%%%%%%%%%%%%%
%%%%%%%%%%%%%%%%%%%%%%%%%%%%%%%%%%%%%%%%
%%%%%%%%%%%%%%%%%%%%%%%%%%%%%%%%%%%%%%%%
%%%%%%%%%%%%%%%%%%%%%%%%%%%%%%%%%%%%%%%%
%%%%% BIBLIOGRAFIA
  %%%%%%%%%%%%%%%%%%%%%%%%%%%%%%%%%%%%%%%%
 %%%%%%%%%%%%%%%%%%%%%%%%%%%%%%%%%%%%%%%%
 %%%%%%%%%%%%%%%%%%%%%%%%%%%%%%%%%%%%%%%%

\end{document}